\documentclass[twoside,12pt,leqno]{amsproc}
\usepackage[top=30mm,right=30mm,bottom=30mm,left=30mm]{geometry}
\usepackage[pagebackref]{hyperref}
\usepackage{amsmath,amssymb,amsthm,stmaryrd,tikz,verbatim}
\hypersetup{citecolor=blue, linkcolor=blue, colorlinks=true}

\usepackage{enumerate,datetime,booktabs}
\usepackage{color,graphicx,amsrefs}
\usepackage{threeparttable} 

\definecolor{darkblue}{rgb}{0,0,0.8}
\definecolor{alizarin}{rgb}{0.82, 0.1, 0.26}
\newcommand{\cmt}[1]{{\textbf{\color{darkblue}#1}}}

\newcommand{\PGammaL}{\mathrm{P}\Gamma\mathrm{L}}
\newcommand{\PGammaSp}{\mathrm{P}\Gamma\mathrm{Sp}}
\newcommand{\GammaSp}{\Gamma\mathrm{Sp}}
\newcommand{\PGammaU}{\mathrm{P}\Gamma\mathrm{U}}

\newcommand{\GammaU}{\Gamma\mathrm{U}}

\newcommand{\GammaL}{\Gamma\mathrm{L}}
\newcommand{\GammaO}{\Gamma\mathrm{O}}
\newcommand{\PGammaO}{\mathrm{P}\Gamma\mathrm{O}}
\newcommand{\PGL}{\mathrm{PGL}}

\newcommand{\PSp}{\mathrm{PSp}}
\newcommand{\CSp}{\mathrm{CSp}}
\newcommand{\Sz}{\mathrm{Sz}}
\newcommand{\G}{\mathrm{G}}

\newcommand{\GO}{\mathrm{GO}}
\newcommand{\CO}{\mathrm{CO}}
\newcommand{\PCO}{\mathrm{PCO}}

\newcommand{\Sp}{\mathrm{Sp}}
\newcommand{\PSU}{\mathrm{PSU}}
\newcommand{\SU}{\mathrm{SU}}
\newcommand{\POmega}{\mathrm{P}\Omega}
\newcommand{\PGU}{\mathrm{PGU}}
\newcommand{\GU}{\mathrm{GU}}
\newcommand{\CU}{\mathrm{CU}}
\newcommand{\PCU}{\mathrm{PCU}}

\newcommand{\PSL}{\mathrm{PSL}}
\newcommand{\PSO}{\mathrm{PSO}}
\newcommand{\SO}{\mathrm{SO}}
\newcommand{\Or}{\mathrm{O}}
\newcommand{\GL}{\mathrm{GL}}
\newcommand{\SL}{\mathrm{SL}}
\newcommand{\Sym}{\mathrm{Sym}}

\newcommand{\Alt}{\mathrm{Alt}}
\newcommand{\Aut}{\mathrm{Aut}}
\newcommand{\C}{\textup{C}}

\newcommand{\cN}{\mathcal{N}}
\newcommand{\cP}{\mathcal{P}}
\newcommand{\cU}{\mathcal{U}}
\newcommand{\eps}{\varepsilon}
\newcommand{\F}{\mathbb{F}}

\newcommand{\nonsplit}[2]{#1\raisebox{0.6ex}{$\cdot$} #2}

\newcommand{\oB}{\overline{B}}
\newcommand{\oG}{\overline{G}}
\newcommand{\oH}{\overline{H}}

\renewcommand{\P}{\mathcal{P}}
\newcommand{\soc}{\mathrm{soc}}

\newcommand{\Wr}{\mathrm{wr}}
\newcommand{\Z}{\mathrm{\sf Z}}

\newcommand{\Hbar}{\overline{H}}

\newcommand{\Vbar}{\overline{V}}

\newcommand{\la}{\langle}
\newcommand{\ra}{\rangle}

\renewcommand{\ge}{\geqslant}
\renewcommand{\le}{\leqslant}
\renewcommand{\geq}{\geqslant}
\renewcommand{\leq}{\leqslant}
\newcommand{\lhdeq}{\trianglelefteqslant}    
\newcommand{\rhdeq}{\trianglerighteqslant}   

\newtheorem{theorem}{Theorem}[section]
\newtheorem*{conj*}{Conjecture}
\newtheorem{lemma}[theorem]{Lemma}
\newtheorem{proposition}[theorem]{Proposition} 
 
\newtheorem{corollary}[theorem]{Corollary}

\theoremstyle{definition}

\newtheorem{remark}[theorem]{Remark}
\newcommand\xqed[1]{%
  \leavevmode\unskip\penalty9999 \hbox{}\nobreak\hfill\quad\hbox{#1}}
\AtEndEnvironment{remark}{\xqed{$\triangle$}}  

\numberwithin{table}{section}

\makeatletter        
\def\@adminfootnotes{%
  \let\@makefnmark\relax  \let\@thefnmark\relax
  \ifx\@empty\@date\else \@footnotetext{\@setdate}\fi
  \ifx\@empty\@subjclass\else \@footnotetext{\@setsubjclass}\fi
  \ifx\@empty\@keywords\else \@footnotetext{\@setkeywords}\fi
  \ifx\@empty\thankses\else \@footnotetext{%
    \def\par{\let\par\@par}\@setthanks}%
  \fi}\makeatother   

\title[Subgroups of Classical Groups that are Transitive on Subspaces]{Subgroups of Classical Groups\\ that are Transitive on Subspaces}
\author{Michael Giudici, S.\,P. Glasby and Cheryl E. Praeger}
\date{\currenttime\ \today}                   

\address{\newline
\noindent Centre for the Mathematics of Symmetry and Computation\newline
Department of Mathematics and Statistics\newline
The University of Western Australia\newline
35 Stirling Highway, Crawley, WA 6009, Australia.\newline
{\sc E-mail addresses:} {\tt \{michael.giudici,\,stephen.glasby,\,cheryl.praeger\}@uwa.edu.au}}

\thanks{The authors gratefully acknowledge the support of the Australian Research Council (ARC) Discovery Project Grants DP160102323 and DP190100450.}


\begin{document}

\begin{abstract}
For each finite classical group $G$, we classify the subgroups of $G$ which act transitively on a $G$-invariant set
of subspaces of the natural module, where the subspaces are either totally isotropic or nondegenerate. 
Our proof uses the classification of the maximal factorisations of almost simple groups. As a first application of these results we classify all  
point-transitive subgroups of automorphisms of finite thick generalised quadrangles.
\end{abstract}

\maketitle

\date{\today}

\section{Introduction}\label{S:intro}


Work on classifying subgroups of finite classical groups that are transitive on a family of subspaces goes back to the seminal results of Christoph Hering~\cites{H1, H2} in the 1970s and 1980s.  These results led to a complete classification of all subgroups of semilinear transformations of a finite vector space $V$ that are transitive on nonzero vectors. For each such subgroup, the corresponding subgroup of the projective group is transitive on $1$-subspaces. Conversely, for each subgroup of the projective group which is transitive on $1$-subspaces, its full preimage in $\GammaL(V)$ is transitive on nonzero vectors.  Adjoining the group of translations to a transitive subgroup $H\leq \GammaL(V)$ yields a $2$-transitive subgroup of the affine group on $V$. Indeed  Hering's classification of transitive semilinear groups is equivalent to the classification of finite $2$-transitive permutation groups of affine type. His work also led Hering to discover two new sporadic linear spaces~\cite{H3}, and inspired others to  construct new geometries and designs~\cites{B, Pf}.
 
 Over the years, subspace actions have been studied in relation to problems in group theory, geometry and combinatorics. For example, they are a natural source of primitive actions. Indeed, in the 1980s,  Oliver King~\cites{King1981,King1981a,King1982} showed that, with a few specified exceptions, each action by a finite classical group on a family of totally isotropic or nondegenerate subspaces is primitive.
 
 Our aim in this paper is to classify all subgroups of finite classical groups that act transitively on the set of all subspaces of a given isometry type. We classify all subgroups of $\GammaL_n(q)$ that are transitive on the set of all $k$-dimensional subspaces for some fixed $k$ (Theorem \ref{T:SLbigN}).  
 For unitary and symplectic groups we concentrate on the nondegenerate and totally isotropic subspaces while for orthogonal groups we focus on the nondegenerate and totally singular subspaces, and also the nonsingular subspaces of dimension one in the even characteristic case.  This completes the program begun in~\cite{Reg}.  Some low dimensional cases were dealt with by Kantor and Liebler \cite{KLrank3} but our proofs are independent of theirs.

 \begin{theorem}
 Let $V$ be a vector space over a finite field equipped with a nondegenerate quadratic, hermitian or alternating form. Let $\cU$ be the set of all totally isotropic (or totally singular) subspaces of a given dimension,  or the set of all nondegenerate subspaces of $V$ of a given isometry type. If $H$ is a group of semisimilarities of the form and acts transitively on $\cU$ then $H$ is known (see Table~$\ref{T:summary}$).
 \end{theorem}

 More details about the exact collections of subspaces $\cU$ considered in the case where $V$ is equipped with a nondegenerate quadratic form are given in Section~\ref{sub:spaces} and Table~\ref{T-defn}. For example we do not treat the case where $V$ has dimension $4$ and is equipped with a quadratic form of $+$ type.  We also determine all groups which are transitive on each of the two types of maximal totally singular subspaces for quadratic forms of $+$ type, and all groups that are transitive on all subspaces of the same similarity type for quadratic forms.
Table~\ref{T:summary} provides the relevant result numbers for each type of form and type of subspaces considered.  
As was the case in \cite{Reg}*{Chapter 4}, we often specify that a subgroup transitive on a family of subspaces $\cU$ should have a certain property and we do not guarantee that every subgroup with this property is transitive on $\cU$.
 
 \renewcommand{\arraystretch}{1.1}
 \begin{table}
 \caption{Glossary of results (we do not treat $\Or^+$ in dimension $4$).}\label{T:summary}
 \begin{tabular}{llcll}
 \toprule
Type   &Dimension $n$& $q$ & Type of $k$-subspaces & Result\\
 \hline
 ${\rm U}$& $n=2$& all  & totally isotropic &Theorem~\ref{T2}\\
 ${\rm U}$& $n=2$& all  & nondegenerate & Theorem~\ref{T2}\\
 ${\rm U}$& $n\geq 3$& all  & totally isotropic & Theorem~\ref{T1}\\
 ${\rm U}$& $n\geq 3$& all  & nondegenerate& Theorem~\ref{T1}\\
 $\Sp$&  $n=2m\geq 4$& all   & totally isotropic& Theorem~\ref{T3}\\
 $\Sp$&  $n=2m\geq 4$& all  & nondegenerate& Theorem~\ref{T3}\\
 $\Or^\circ$& $n=3$& all  & as in Table~\ref{T-defn} & Theorem~\ref{T:Omega3}\\
 $\Or^-$& $n=4$& all & as in Table~\ref{T-defn} & Theorem~\ref{T:KleinPSL2q2}\\
 $\Or^\circ$& $n=5$&   odd& as in Table~\ref{T-defn} & Theorem~\ref{T:KleinSp4}\\
 $\Or^+$& $n=6$& all & as in Table~\ref{T-defn} & Theorem~\ref{T:KleinSL4}\\
 $\Or^-$& $n=6$& all  & as in Table~\ref{T-defn} & Theorem~\ref{T:GOminus6q}\\
 $\Or^-$& $n=2m\ge8$& all  & as in Table~\ref{T-defn} & Theorem~\ref{T9}\\
 $\Or^-$& $n=2m\ge4$&  odd & nondegenerate with $k$ odd & Proposition~\ref{P1}\\
 $\Or^\circ$& $n=2m+1\ge7$&  odd & as in Table~\ref{T-defn} & Theorem~\ref{T4}\\
 $\Or^\circ$& $n=2m+1\ge5$&  even & as in Table~\ref{T-defn} & Theorem~\ref{T:Ooddqeven}\\
 $\Or^+$& $n=2m\ge8$& all  & as in Table~\ref{T-defn} & Theorem~\ref{T10}\\
 $\Or^+$& $n=2m\ge6$& all  & totally singular with $k=n/2$ & Proposition~\ref{P:Both}\\
$\Or^+$& $n=2m\ge6$&  odd &  nondegenerate with $k$ odd  & Proposition~\ref{P2}\\
\bottomrule
 \end{tabular}
 \end{table}

 A similar problem for classical algebraic groups has been studied by Liebeck, Saxl and Seitz in~\cite{LSS}. Not only are there far fewer possibilities for algebraic classical groups than for finite classical groups, some examples have no finite analogue, for example in even characteristic the group  $\Sp_{2m}$ contains the subgroup $\SO_{2m}$ which acts transitively on the set of nondegenerate subspaces of dimension $2k$ for any $k$ satisfying $1\leqslant k\leqslant m-1$, for details see~\cite{LSS}.

Our main method will be to use group factorisations. 
Whenever we have a finite almost simple classical group $G$ on $V$, a $G$-invariant family $\cU$ of proper subspaces of $V$, and a subgroup $H<G$ transitive on $\cU$, we obtain a factorisation $G=AH$ where $A$ is the stabiliser $G_U$ of a subspace $U\in\cU$. In fact $G=AH$ if and only if $H$ is transitive on $\cU$.  If both factors $A, H$ are maximal and core-free, then $G=AH$ is called a \emph{maximal core-free factorisation}.  
Such factorisations of the almost simple classical groups were classified by Liebeck, Saxl and the third author in~\cite{Factns}, their major motivation, and first application, being to classify the maximal subgroups of the finite alternating and symmetric groups~\cite{LPS4}.  Hence if $G$ acts primitively on $\mathcal{U}$ then the classification in \cite{Factns} provides us with the \emph{maximal core-free} subgroups of $G$ that are transitive on $\mathcal{U}$.  All possibilities for $H$ (that is, with the maximality assumption removed) for some of the subspace actions of finite classical groups were determined by Liebeck, Saxl and the third author in~\cite{Reg}*{Chapter~4} as part of their strategy to classify all  exact factorisations  of almost simple groups $G$  with one factor maximal and core-free. (A factorisation  $G=AB$  is called \emph{exact} if $A\cap B=1$.) This resulted firstly, in the 
 classification of  all almost simple primitive permutation groups containing a regular subgroup~\cite{Reg}*{Theorem~1.1} and secondly in the classification of certain families of so-called ‘$B$-groups’~\cite{Reg}*{Corollary 1.5} and Cayley graphs~\cite{Reg}*{Theorem~1.6}. 

We expect that our classification will be useful in many contexts. As a first application we classify all point-transitive subgroups acting on the (finite thick) classical generalised quadrangles in Theorem~\ref{C:GQ}. Previously only the point regular subgroups were known, as an application of~\cite{Reg} by Bamberg and the first author~\cite{BamG}.

The outline of this paper is as follows. Section~\ref{S:notation} describes our strategy and
establishes the notation we use. We discuss the linear case, and some small dimensional orthogonal groups, in Section~\ref{S:PGLn}.
The unitary case is considered in Sections~\ref{S:PGU2} (for $n=2$) and~\ref{S:PGUn} (for $n>2$), and the symplectic case in Section~\ref{S:Sp}. The odd dimensional
orthogonal case is dealt with in Section~\ref{S:OOdd}, and the even dimensional case in Sections~\ref{S:OEven} (minus type) and~\ref{S:OPlus} (plus type). Our classification of point-transitive groups on 
classical generalised quadrangles is proved in Section~\ref{S:ExceptIsos2}.

\begin{remark}\label{rem:Reg}
Some of the subspace actions we consider were already dealt with in \cite{Reg}*{Chapter 4}, as the results were needed for the classification in \cite{Reg} of the regular subgroups of the finite almost simple primitive permutation groups. Our study covered more subspace actions of classical groups, and we needed both to use and to extend the results in \cite{Reg}. In so-doing we uncovered a few issues:
 \begin{itemize}
 \item[$\circ$] In some of the results in \cite{Reg}*{Chapter 4} the list of possibilities is a ‘superset’: not all the listed groups mentioned are actually examples. This is mentioned in \cite{Reg}*{text and Remark on p. 16}. However we needed more detailed information about some of these cases.
\item[$\circ$] In some of the results in \cite{Reg}*{Chapter 4}, the groups listed are actually subgroups of the projective classical groups, rather than the classical groups themselves. We have clarified some of these cases.
 \item[$\circ$] In one case we found a missing family of examples.
\end{itemize}
 
Here we summarise details of these issues and point to where they are addressed in our paper. We organise this according to the relevant results in \cite{Reg}*{Chapter 4}.

\begin{enumerate}[(a)]
\item \cite{Reg}*{Lemma 4.1} determining types of subspaces admitting transitive proper subgroups of classical groups. For \cite{Reg}*{Lemma 4.1(iv)}, that is the case `$G_0=\Omega_{2m+1}(q)$ ($q$ odd, $m\geq 3$)’, the possibility `$M=P_1$ ($m=3$)’ should be added (see Theorem~\ref{T4}(a) and its proof, especially Table~\ref{T:OOddCases}).  For \cite{Reg}*{Lemma 4.1(vii)}, that is, the `exceptional cases’, the possibility `$G_0=\PSL_4(2)$, $M=P_2$’ should be added (see the proof of Theorem \ref{T:SLbigN}).

 \item \cite{Reg}*{Lemma 4.2} for $\GammaSp_{2m}(q)$ acting on nondegenerate 2-subspaces.  The  possibilities `$\Sp_{m/2}(q^4) $ ($m/2$ even,  $q=2$)' in \cite{Reg}*{Lemma 4.2(i)}, and  `$\G_{ 2}(q^4)$ ($m=12$, $q=2$)' in \cite{Reg}*{Lemma 4.2(ii)}, as subgroups of $\GammaSp_{2m}(2)$, lead to no examples. See Lemma~\ref{L:temp}.  
  
\item  \cite{Reg}*{Lemma 4.3} for $\GammaU_{2m}(q)$ acting on nondegenerate 1-subspaces. Two listed groups in this result are the images in $\PGammaU_{2m}(q)$, namely the groups  $\PSU_4(3)$ and $M_{22}$ listed in \cite{Reg}*{Lemma 4.3(iv)} as subgroups of $\GammaU_6(2)$ should be $\nonsplit{3}{\PSU_4(3)}$ and $\nonsplit{3}{M_{22}}$, respectively. (See Theorem~\ref{T1}, especially Table~\ref{T6}  and Case 7 in the proof. See also \cite{BHRD}*{Table 8.27}.)
 
\item  \cite{Reg}*{Lemma 4.4} for $\GammaO_{2m}^-(q)$ acting on nondegenerate 1-subspaces.  The group $\PSL_3(4)$ listed in \cite{Reg}*{Lemma 4.4(iii)} as a subgroup of $\GammaO_6^-(3)$ should be $\nonsplit{2}{\PSL_3(4)}$. (See Theorem~\ref{T:KleinSU4}(c).)

\item \cite{Reg}*{Lemma 4.5}  for $\GammaO_{2m}^+(q)$ acting on nondegenerate 1-subspaces. Several groups given in \cite{Reg}*{Lemma 4.5(v), (vii) and (vii)} as subgroups of $\GammaO_{2m}^+(q)$ are the projective versions of these groups, and in particular there are no examples corresponding to the group $\Alt_6$ listed there. See Theorem~\ref{T10}, especially Table~\ref{T:OPlusN1}, and also case (e) of the proof of Lemma~\ref{lem:m4q3}.  The relevant part of \cite{Reg}*{Lemma 4.5(v)} should state:
$$m=4: B\rhdeq {}^\wedge\Omega_7(q)^I \textrm{ or } \nonsplit{2}{\Omega_8^-(q^{1/2})}$$
where $ {}^\wedge\Omega_7(q)^I=\Omega_7(q)$ when $q$ is even and $\nonsplit{2}{\Omega_7(q)}$ when $q$ is odd. The relevant parts of
\cite{Reg}*{Lemma 4.5(vii)} should state:
\begin{align*}
  &m=4, q=3 : B \rhdeq  \nonsplit{2}{\Omega_8^+(2)}, \nonsplit{2}{\Sp_6(2)}, \nonsplit{2}{\SU_4(2)}, \textrm{ or }2.\Alt_9\\
  &m=6, q=2: B\rhdeq \SU_4(3) \textrm{ or } \nonsplit{3}{M_{22}}.
\end{align*}

\item  \cite{Reg}*{Lemma 4.6} for $\GammaSp_{2m}(q)$ acting on the coset space of $\Sp_{2m}(q)/\GO_{2m}^-(q)$.  There is an infinite family of examples missing from \cite{Reg}*{Lemma 4.6{iv}}, namely subgroups containing $\Omega_{2m}^+(q)$ as a normal subgroup when $q=2$ or $4$. See Theorem~\ref{T:Ooddqeven}(e) (and Case $\cU=N_{2m}^\eps$ of the proof). The relevant part \cite{Reg}*{Lemma 4.6(iv)} should state:
  \[
  q=2 \textrm{ or }4, \textrm{ and either }\Omega_{2m}^+(q)\lhdeq B, \textrm{ or $B$ is as in (i)--(vii) of Lemma 4.5.}
  \]
\end{enumerate}
\end{remark}

\begin{remark}\label{rem:KL}
In their work on rank three permutation representations of classical groups, Kantor and Liebler determined the subgroups $K$ of $\GammaSp_4(q)$, $\GammaU_4(q)$ and $\GammaU_5(q)$ that are transitive on the totally isotropic 1-dimensional or 2-dimensional subspaces \cite{KLrank3}*{Corollary 5.12} and the subgroups of $\GammaU_3(q)$ that are transitive on totally isotropic 1-subspaces. They also determined the subgroups of $\GammaO^+_6(q)$ that are transitive on totally singular 1-subspaces \cite{KLrank3}*{Lemma 5.15}. We note the following corrections/clarifications:

\begin{enumerate}[(a)]
\item  \cite{KLrank3}*{Corollary 5.12(vi)} lists $\SU_3(2)''\lhdeq K\leqslant \GammaU_4(2)$. This is our $3^{1+2}_+$ in Theorem~\ref{T1}(c) and Table~\ref{T8}. 
\item  \cite{KLrank3}*{Corollary 5.12(vii)} lists $\C_3^3\rtimes \Alt_4\lhdeq K\leqslant \GammaU_4(2)$. This corresponds with our case $\C_3^3\lhdeq K$ in Theorem~\ref{T1}(c) and Table~\ref{T8}.
\item \cite{KLrank3}*{Corollary 5.12} is missing $3_-^{1+2}\lhdeq K\leqslant \GammaU_4(2)$ acting transitively on totally isotropic 2-subspaces as in Theorem~\ref{T1}(c) and Table~\ref{T8}. This example has previously been observed in \cite{Reg}*{Table 16.1}
\item  \cite{KLrank3}*{Corollary 5.12(x)} lists $E\rtimes \C_5\lhdeq K\leqslant \GammaSp_4(3)$ where $E$ is extraspecial of order $2^5$. This should read as $E\rtimes \C_5\leqslant K$ to allow the possibility that $E\rtimes \Alt_5\lhdeq K$ as in Theorem~\ref{T3}(a) Table~\ref{T:Spa}.
\item  \cite{KLrank3}*{Corollary 5.12} is missing $\C_{19}\lhdeq K\leqslant \GammaU_4(8)$ acting transitively on totally isotropic 2-subspaces as in Theorem~\ref{T1}(c) and Table~\ref{T8}. This example has previously been observed in \cite{Reg}*{Table 16.1}. A similar example is missing in the list of subgroups of $\GammaU_3(8)$ acting transitively on totally isotropic 1-subspaces in \cite{KLrank3}*{Lemma 2.8}, see Theorem~\ref{T1}(b) and Table~\ref{T7}. This example has previously been observed in \cite{Factns}*{Table 3}.
\item \cite{KLrank3}*{Lemma 2.8(iv)} lists $K'=\Alt_7\leqslant \GammaU_3(5)$. This should be $K'=\nonsplit{3}{\Alt_7}$ as in  Theorem~\ref{T1}(b) Table~\ref{T7}. See also \cite{BHRD}*{Table 8.6}. 
\end{enumerate}
\end{remark}

\begin{remark}
We would like to take this opportunity to note an example of a regular subgroup of an almost simple group that is missing from \cite{Reg}*{Table 16.1}. There should be an extra row as follows:

\bigskip
\begin{center}
\begin{tabular}{|l|l|l|l|l|}
\hline
$L$ & $G_\alpha\cap L$ & $B$ & no. of classes & Remark\\\hline
$L_3(8)$ &$73:3$ & $[8^2]P_1(\GammaL_2(8))$ & 1& *, $G=L_.3=\PGammaL_3(8)$\\\hline
\end{tabular}
\end{center}

\bigskip
The proof should be corrected on the first line of  \cite{Reg}*{p25}. The possibility $n=3$ and $q=8$ should also be listed  as $\SL_2(8)$ has a proper subgroup of index dividing $n\log q$. The table following the first line should then include the row

\begin{center}
 \begin{tabular}{ccc}
$G$ & $A$ &$B$ \\\hline
$L_3(8).3$ &  $73:9$ & $[8^2]P_1(\GammaL_2(8))$
\end{tabular}
\end{center}

Comparing orders we see that $|G|=|A||B|$ and $|A\cap B|\leqslant 3$. However, 3 does not divide $|B\cap \PSL_3(8)|$ while all elements of order 3 in A lie in $\PSL_3(8)$. Hence $A\cap B=1$ and $G=AB$, as desired.
\end{remark}

\begin{remark}\label{rem:pablo}
We have been informed by Pablo Spiga \cite{Spiga} that there is a factorisation $\Omega_8^+(4).2=N_2^-\GO_8^-(2)$ which is missing from \cite{Factns}*{Table 4}. Here $\GO_8^-(2)=N_{\GammaO^+_8(4)}(A)$ where $A$ is an image of the subfield subgroup $\Omega_8^-(2)$ under triality. The existence of this factorisation has been confirmed in \textsc{Magma}. The gap in the argument occurs in \cite{Factns}*{(5.1.15)(b)(iii) on pp.\,106--107}  and suggests that  there is at most one other missing factorisation of this form, namely  $\Omega_8^+(16).4=N_2^-(\GO_8^-(4).2)$. However, its existence is yet to be determined. In both cases, that is for $q=4$ and $16$, we have that $N_2^- \cap \GO_8^-(q^{1/2}) \cap \Omega_8^+(q) =\SL_2(q^{1/2})\times \GO_2^-(q)$ and is contained in the stabiliser in  $\Omega_8^-(q^{1/2})$ of an elliptic 4-subspace of the 8-dimensional vector space over $\F_{q}$  equipped with a nondegenerate elliptic quadratic form. In particular, comparing orders yields  that $\Omega_8^-(q^{1/2})$ is not transitive on $\cN_2^-$ in either case and $\GO_8^-(4)$ is not transitive when $q=16$.
\end{remark}

 \section{Our strategy and notation}\label{S:notation}
 
 Let $V$ be an $n$-dimensional vector space over the field $\F_q$.  For $1\leqslant k<n$, our first goal is to determine all subgroups of $\GammaL_n(q)$ that act transitively on the set $\cU$ of all subspaces of dimension $k$. We do this in Section~\ref{S:PGLn}. 
   
In the rest of this section we discuss the families of classical groups, and their subspace actions, that we will consider. 

\subsection{Some notation for classical groups and their subgroups} 

Our notation for the classical groups is recorded in Tables~\ref{T0} and~\ref{T00}. It differs from both~\cite{KL}
  and~\cite{BHRD} but is influenced by both. Type $\Or^\circ$ refers to odd-dimensional orthogonal groups while types $\Or^-$ and $\Or^+$ refer to even-dimensional orthogonal groups. Column $\Gamma$ refers to the group of all semisimilarities of the form, column $C$ refers to the group of all similarities, column $I$ the group of all isometries and column $S$ the group of all isometries of determinant one.  We follow \cite{Taylor}*{p.\,136} and define $\Omega_n(q)$ to be the derived subgroup of $\GO_n(q)$.  We use the standard term `general linear' and
  the standard notation `$\GL$', instead of `conformal linear' and `$\textup{CL}$'
  even though the latter would be more consistent in Table~\ref{T0}.   Finally, for a 
  statement $P$, 
we use the Iverson bracket notation $[P]$ to mean 1 if $P$ is true, and 0
if $P$ is false. Hence $(q-1,2)^{[n\ge3]}$ equals 2 if $n\ge3$ and $q$ is odd,
and equals 1 otherwise.

  \begin{table}[!ht]
    \caption{Classical group notation where the groups are ordered from large to small, and indices are listed. See footnote $\tnote{(*)}$, and here $q=p^f$.}\label{T0}
  \begin{threeparttable}
  \begin{tabular}{clclclclcl}
  \toprule
  type&$\Gamma$&$|\Gamma: C|$&$C$&$|C: I|$&$I$&$|I:S|$&$S$&$|S:\Omega|$&$\Omega$\\
  \midrule
  {\bf L}&$\GammaL$&$f$&$\GL$&$1$&$\GL$&$q-1$&$\SL$&$1$&$\SL$\\
  {\bf U}&$\GammaU$&$2f$&$\textup{CU}$&$q-1$&$\GU$&$q+1$&$\SU$&$1$&$\SU$\\
  {\bf S}&$\GammaSp$&$f$&$\textup{CSp}$&$q-1$&$\Sp$&$1$&$\Sp$&$1$&$\Sp$\\
  ${\bf O}^\circ$&$\GammaO^\circ$&$f$&$\textup{CO}^\circ$&$\frac{q-1}{(q-1,2)}$&$\textup{GO}^\circ$&$(q-1,2)$&$\SO^\circ$&$(q-1,2)^{[n\ge3]}$&$\Omega^\circ$\\
  ${\bf O}^\eps$&$\GammaO^\eps$&$f$&$\textup{CO}^\eps$&$q-1$&$\textup{GO}^\eps$&$(q-1,2)$&$\SO^\eps$&$2$&$\Omega^\eps$\\
  \bottomrule
  \end{tabular}
  \begin{tablenotes}\footnotesize
  \item [$(*)$] In row $\Or^\circ$ the formula for $|S:\Omega|$ is valid
    when $n$ is odd with $q$ even or odd with two exceptions: when
    $n\in\{3,5\}$ and $q=2$, we have
    $|S:\Omega|=|\Sp_{n-1}(2):\Sp_{n-1}(2)'|=2\ne1$.
  \end{tablenotes}
  \end{threeparttable}
  \end{table}

  \begin{table}[!ht]
  \caption{Projective classical group notation and projective indices. 
    If $n$ is even, then $a_{+}=a_{-}=2$ if $q$ is even, otherwise $a_{+}=2^{[n(q-1)/4\;\;{\rm even}]}$ and $a_{-}=2^{[n(q-1)/4\;\;{\rm odd}]}$ if $q$ is odd. Let $\delta_2=(2,q-1)$.
  }\label{T00}
  \begin{tabular}{clclclclcl}
  \toprule
  type&$\overline{\Gamma}$&$|\overline{\Gamma}:\overline{C}|$&$\overline{C}$&$|\overline{C}:\overline{I}|$&$\overline{I}$&$|\overline{I}:\overline{S}|$&$\overline{S}$&$|\overline{S}:\overline{\Omega}|$&$\overline{\Omega}$\\
  \midrule
  ${\bf L}$&$\PGammaL$&$f$&$\PGL$&$1$&$\PGL$&$(q-1,n)$&$\PSL$&$1$&$\PSL$\\
  {\bf U}&$\PGammaU$&$2f$&$\textup{PCU}$&$1$&$\PGU$&$(q+1,n)$&$\PSU$&$1$&$\PSU$\\
  {\bf S}&$\PGammaSp$&$f$&$\textup{PCSp}$&$\delta_2$&$\PSp$&$1$&$\PSp$&$1$&$\PSp$\\
  ${\bf O}^\circ$&$\PGammaO^\circ$&$f$&$\textup{PCO}^\circ$&$1$&$\textup{PGO}^\circ$&$1$&$\PSO^\circ$&$\delta_2^{\;[n\ge3]}$&$\POmega^\circ$\\
  ${\bf O}^\eps$&$\PGammaO^\eps$&$f$&$\textup{PCO}^\eps$&$\delta_2$&$\textup{PGO}^\eps$&$(q-1,2)$&$\PSO^\eps$&$a_\eps$&$\POmega^\eps$\\
  \bottomrule
  \end{tabular}
  \end{table}

   We denote an extension of $N$ by $Q$ by $N\ldotp Q$, a split extension
   by $N\colon Q$ or $N\rtimes Q$, and a nonsplit extension by $\nonsplit{N}{Q}$. Finally, if $X\leqslant \GammaL_n(q)$ then we use the notation $X^R$ to indicate that $X$ is reducible on $V$.

\subsection{Formed spaces: types of subspaces}
Next we assume that $V$ is equipped with a nondegenerate hermitian, alternating or quadratic form and  $G$ is the group of all semisimilarities of the form. 
This means, in the case of a hermitian or alternating form that the radical $V^\perp=0$.
The case of a quadratic form needs more comment: a vector is \emph{singular} if its value under the quadratic form is zero, and otherwise it is \emph{nonsingular}.    Also a quadratic form determines a symmetric bilinear form, called its \emph{polar form}, that is also $G$-invariant, and the quadratic form is said to be nondegenerate if the radical $V^\perp$ of its polar form is either zero, or $q$ is even, $n$ is odd, and $V^\perp$ is a nonsingular $1$-subspace.

Thus in all cases $G$ preserves a sesquilinear form on $V$. A subspace $U$ is called \emph{nondegenerate} if $U\cap U^\perp=0$, and \emph{totally isotropic} if   $U\cap U^\perp=U$, that is to say, $U\subseteq U^\perp$. Again we comment further on the quadratic form case: a subspace $U$ is \emph{totally singular} if all its vectors are singular.  In a totally isotropic subspace $U$, the set $U^*$ of all singular vectors forms a subspace of $U$, and either $U^*=U$, or $q$ is even and $U^*$ has codimension $1$ in $U$. Moreover, the stabiliser $G_U$ leaves  $U^*$ invariant.
Thus in the quadratic form case we will usually deal with totally singular subspaces. 

Broadly speaking, our goal is to describe all subgroups of $G$ that are transitive on the set $\cU$ of all proper nonzero subspaces of $V$ of a given isometry type. For $U\in\cU$, the stabiliser $G_U$ also leaves invariant $U^\perp$ and $U\cap U^\perp$. If $U\cap U^\perp=0$ then $U$ is nondegenerate. On the other hand, if $U\cap U^\perp\ne 0$, then $U\cap U^\perp$ is totally isotropic and $G_U\leq G_{U\cap U^\perp}\leq G$. In the cases of an hermitian or alternating form, or a quadratic form in odd characteristic, we have proper containment $G_{U\cap U^\perp}< G$ and,  if we know all the subgroups $H$ of $G$ that are transitive on the set of $G$-translates of $U\cap U^\perp$, then we need to identify among these subgroups $H$ the ones which are transitive on $\cU$, or equivalently, those subgroups $H$ for which $H_{U\cap U^\perp}$ is transitive on the subspaces in $\cU$ containing $U\cap U^\perp$. The case of  a quadratic form with $q$ even is different. If $G_{U\cap U^\perp}< G$, then $G_U$ leaves invariant the unique subspace $(U\cap U^\perp)^*$ of singular vectors in $U\cap U^\perp$. Either this totally singular subspace is nonzero, or $U=U\cap U^\perp$ is a nonsingular $1$-subspace (not equal to $V^\perp$ since $G_{U\cap U^\perp}< G$), and the same comments apply for subgroups $H$ transitive on  the set of $G$-translates of the totally singular  $(U\cap U^\perp)^*$ or nonsingular $U\cap U^\perp$, respectively.  On the other hand if $G_{U\cap U^\perp}= G$ then $U\cap U^\perp=V^\perp$ and, provided $U\ne V^\perp$, the quotient $U/V^\perp$ is a proper nonzero nondegenerate subspace of the symplectic space $V/V^\perp$. 
 As $G$ is isomorphic to the corresponding symplectic group, the possible transitive subgroups are given by our results for symplectic groups.

Because of this discussion, and our need to undertake a manageable classification, we therefore restrict our attention to $G$-invariant families $\cU$ of nondegenerate subspaces and of totally isotropic subspaces; and in the case of orthogonal spaces with $q$ even, families of nondegenerate subspaces, totally singular subspaces, and nonsingular $1$-subspaces (distinct from the radical).

\subsection{Formed spaces: dimensions of subspaces}\label{sub:spaces}
We denote the set of all nondegenerate subspaces of dimension $k$ by $\cN_k$, and the set of all totally isotropic (or totally singular) subspaces of dimension $k$ by $\cP_k$. 
 In the case of orthogonal spaces with $q$ even, we also denote the set of nonsingular $1$-subspaces by $\cN_1$. We will sometimes use the notation $\cP_k(V)$ or $\cN_k(V)$ if we wish to specify the underlying vector space~$V$.  In a classical group we often use $P_k$ to denote the stabiliser of a totally isotropic subspace of dimension $k$ (or a totally singular $k$-subspace in the orthogonal case),  and $N_k$ to denote the stabiliser of a nondegenerate subspace of dimension $k$. Again, we will sometimes use the notation $P_k(G)$ or $N_k(G)$ if we wish to specify the classical group $G$. Our notation for $P_k$ is consistent with \cite{KL} but inconsistent with \cite{Factns} and \cite{Reg}, which we explain below.
  
For a totally isotropic (or totally singular) subspace $U$, since $U\subseteq U^\perp$ and $\dim(U)+\dim(U^\perp)=n$, the subspace $U$ has dimension at most $n/2$.  For a nondegenerate subspace $U$, the stabiliser $G_U$ also fixes the orthogonal complement $U^\perp$. In the hermitian and alternating cases $U^\perp$ is also nondegenerate and $G_U=G_{U^\perp}$. The groups $\SU_n(q)$ and $\Sp_n(q)$ act transitively on $\cP_k$ and $\cN_k$, so we only need to find all transitive subgroups $H$ not containing these groups, and moreover for both  $\cP_k$ and $\cN_k$,  we may assume that the dimension $k\leqslant n/2$.  
 
The remaining comments are about the orthogonal (or quadratic form) case.  It is more complicated than the other cases, and so we record in 
Table~\ref{T-defn} the families $\cU$ of subspaces that we consider. Let $U\in\cU$. First we deal with totally isotropic subspaces. If $U$ is not totally singular, then, as explained above, $q$ is even,  $U$ is restricted to be a nonsingular $1$-subspace (not $V^\perp$), and $\cU$ is as in row $6$ of Table~\ref{T-defn}. Suppose now that $U$ is totally singular. The group $\Omega_n^\eps(q)$, for $\eps\in\{\circ,+,-\}$, acts transitively on $\cU=\cP_k$ if $k=\dim(U)<n/2$, as in row $1$ of Table~\ref{T-defn}. Totally singular $n/2$-subspaces exist only if $\eps=+$, and $\Omega_n^+(q)$ has two orbits on $\cP_{n/2}$ which we denote by $\cP_{n/2}^+$ and $\cP_{n/2}^-$, as in row $2$ of Table~\ref{T-defn}. (This notation differs from that in \cite{Factns} and \cite{Reg} where $\cP_{n/2}$ and $\cP_{n/2-1}$ are used to denote the two $\Omega_n^+(q)$-orbits, but is more consistent with \cite{KL}.) When $q$ is odd these two orbits are fused by an element of $\GO^+_n(q)$ that is not in $\SO_n^+(q)$, while when $q$ is even they are fused by an element of $\GO^+_n(q)=\SO^+_n(q)$ that is not in $\Omega_n^+(q)$. Also $U,W\in\cP_{n/2}$ lie in the same $\Omega_n^+(q)$-orbit if and only  the codimension of $U\cap W$ in~$U$ (and hence in $W$) is even. 
 
Now we consider nondegenerate subspaces $U$. Here either $k$ is even, or both $k$ and $q$ are odd, and in both cases there are two isometry types of nondegenerate $k$-subspaces $U$.  If $k$ is even, these are the subspaces where the restriction of the quadratic form is hyperbolic or elliptic, and we denote the set of all such subspaces by $\cN_k^+$ or $\cN_k^-$ respectively.   If $k, q$ and $n$ are all odd then a nondegenerate $k$-subspace $U$ can be distinguished by the isometry type of $U^\perp$, and thus we use $\cN_k^\eps$ for $\eps\in\{+,-\}$ to denote the set of all nondegenerate $k$-subspaces $U$ such that $U^\perp$ is of type $\eps$.     Finally, when $k$ and $q$ are odd and $n$ is even, there are two isometry types  of nondegenerate $k$-subspaces and one can be mapped to  the other by a similarity of the whole space that is not an isometry. We still denote the two classes by $\cN_k^\eps$, where $\eps\in\{+,-\},$ and they can be distinguished by the discriminant of the restriction of the quadratic form. This is most easily seen when $k=1$: here the two isometry types are those for which the quadratic form is a square and those where it is a nonsquare. We will use $N_k^\eps$ to denote the stabiliser of an element of $\cN_k^\eps$ when we want to distinguish its isometry type.

If $k>n/2$ then $G_U$ is also the stabiliser of the $(n-k)$-subspace $U^\perp$, and moreover $U^\perp$ is nondegenerate if either $n-k$ is even or $n-k$ and $q$ are both odd. It follows that one of $U$ or $U^\perp$ satisfies the conditions of row 3, 4 or 5 of Table~\ref{T-defn}, except in the case where $q$ is even and one of $U, U^\perp$ has odd dimension less than $n/2$.  However, for $q$ even, an odd dimensional subspace is not nondegenerate, and hence in this exceptional case the nondegenerate subspace $U$ has even dimension $k>n/2$ and $n$ is odd. Here $U^\perp$ contains the nonsingular $1$-subspace 
$V^\perp$. If $U^\perp=V^\perp$ then $U$ is as in row~7 of Table~\ref{T-defn}, and $G_U$ corresponds to the stabiliser of a quadratic form in the 
$G$-action on the symplectic space $V/V^\perp$.   If $U^\perp\ne V^\perp$ then $U^\perp/V^\perp$ is a proper nonzero nondegenerate subspace of the symplectic space $V/V^\perp$. As $G$ is isomorphic to the corresponding symplectic group, the possible transitive subgroups are given by our results for symplectic groups, and so we do not consider this case.

 Thus Table~\ref{T-defn} summarises the main families $\cU$ that we consider for the orthogonal groups with `socle' $\Omega_{n}^\eps(q)$.
 In all rows of this table the group $\Omega_n^\eps(q)$, for $\eps\in\{\circ,+,-\}$, acts transitively on $\cU$, 
 \cite{BG}*{Lemma 2.5.10}.  
Unfortunately, our methods do not work for non-simple groups such as $\POmega_4^{+}(q)\cong\PSL_2(q)^2$. Hence we omit the groups $\textup{O}_{4}^+(q)$. 
 In addition, when $n$ is even, 
 we also identify subgroups of $\GammaO_n^+(q)$ that act transitively on 
$\cP_{n/2}=\cP_{n/2}^+\cup\cP_{n/2}^-$ (Proposition~\ref{P:Both}), 
and if $k$ and $q$ are both odd,  subgroups of  $\GammaO_n^\eps(q)$ that are transitive on $\cN_k=\cN_k^+\cup\cN_k^-$ (Propositions~\ref{P1} and~\ref{P2}).

   \begin{table}[!ht]
    \caption{Families $\cU$ of $k$-subspaces for $n$-dimensional orthogonal groups  $\GO_n^\eps(q)$, $\eps\in\{\circ,+,-\}$. }\label{T-defn}
  \begin{threeparttable}
  \begin{tabular}{lll}
  \toprule
  $\cU$ & $k$ & Conditions\\
  \midrule
  $\cP_k$                    &$1\leq k<n/2$&all  totally singular $k$-subspaces.\\
  $\cP_{n/2}^{\eps'}$&$n/2$            &$\eps=+$, an $\Omega_n^+(q)$-orbit on totally singular $n/2$-subspaces, $\eps' \in\{+,-\}.$\\
  $\cN_{k}^{\eps'}$&even $k\leq n/2$            & all nondegenerate $k$-subspaces of type $\eps' \in\{+,-\}$.\\
  $\cN_{k}^{\eps'}$&odd $k<n/2$            & $n, q$ odd, all nondegenerate $k$-subspaces with $U^\perp$ of  type $\eps' \in\{+,-\}$.\\
  $\cN_{k}^{\eps'}$&odd $k\leq n/2$        & $n$ even, $q$ odd, all nondegenerate $k$-subspaces of a given type, see${}^*$. \\ 
    $\cN_{1} $               &$1$            & $q$ even, all nonsingular $1$-subspaces (apart from $V^\perp$ if $n$ is odd). \\
  $\cN_{n-1}^{\eps'}$&$n-1$        & $n$ odd, $q$ even, all nondegenerate $(n-1)$-subspaces of a given  type${}^{**}$. \\   
  \bottomrule
  \end{tabular}
  \begin{tablenotes}\footnotesize
    \item[$*$] $\eps'\in\{+,-\}$ and the two classes $\cN_k^{\eps'}$ can be distinguished by the discriminant of the restriction of the quadratic form.
    \item[$**$] These subspaces correspond to quadratic forms on the symplectic space $V/V^\perp$ of type  $\eps'\in\{+,-\}$.
  \end{tablenotes}
  \end{threeparttable}
  \end{table}

  \subsection{Some useful lemmas}\label{S:prelim}
  
  We begin with the following lemma to aid switching between matrix groups and their projective versions.
  For a finite group $G$, by $G^{(\infty)}$ we mean the last term in the derived series; and we say that $G$ is \emph{quasisimple} if $G=G^{(\infty)}$ and $G/Z(G)$ is a nonabelian simple group.
  
  \begin{lemma}\label{lem:derived}
    Let $G$ and $G_0$ be groups with $G^{(\infty)}$ quasisimple and let $\pi:G\rightarrow G_0$ be a homomorphism such that $\pi(G^{(\infty)})=G_0^{(\infty)}\neq 1$. Suppose that $M\leqslant G$ satisfies $G_0^{(\infty)}\leqslant \pi(M)$. Then $G^{(\infty)}\leqslant M$.
  \end{lemma}
  \begin{proof}
  Let $K=\ker \pi$, $H=G^{(\infty)}$ and $H_0=G_0^{(\infty)}$.
    Let $Y=M\cap \pi^{-1}(H_0)$ and observe that 
    $\pi(Y)=\pi(M)\cap H_0=H_0=\pi(H)$.  Clearly $K\cap H\lhdeq H$ and $Y\cap H\lhdeq Y$. Since $G/H$ is soluble, so too is $YH/H\cong Y/(Y\cap H)$. However, $\pi(Y)=H_0$ has no nontrivial soluble quotients. Thus $\pi(H)=\pi(Y)=\pi(Y\cap H)$ and hence $H=(Y\cap H)(K\cap H)$. Since $H$ is quasisimple and $K\cap H\neq H$, it follows that $K\cap H\leqslant \Z(H)$. 
    
    Suppose that $H=AB$ where $B\le \Z(H)$. If $a\in A$, $b\in B$ and $c\in H$, then the identity $[ab,c]=[a,c]^b[b,c]=[a,c]$, shows that
    $H'=[AB,H]=[A,H]=[AB,A]=A'$. Taking $A=Y\cap H$ and $B=K\cap H$ shows that $H=H'=A'\le Y\cap H$. Therefore $H\leqslant Y\leqslant M$, as desired.
  \end{proof}
  
  We also state the following variation of \cite{Reg}*{Lemma~2.6}.
  
  \begin{lemma}\label{lem:maxmin}
Let $G=AB$ be a factorisation of an almost simple classical group with $A$ maximal and core-free in $G$. Let $L=\soc(G)$ and suppose that $B$ is maximal among core-free subgroups of $G$. Define $G^*=BL$. Then $G^*=(G^*\cap A)B$. Moreover, this is a maximal core-free factorisation of $G^*$ 
with $B=N_G(B\cap L)$ unless $L=\POmega_8^+(3)$, or $(L,A\cap L, B\cap L)$ is one of the following:
\begin{enumerate}[{\rm (a)}]
    \item $L=\PSL_{2m}(q)$ with $\gcd(q-1,m)>1$, $A\cap L={\rm Stab}(V_1\oplus V_{2m-1})$, and $B\cap L=N_L(\PSp_{2m}(q))$.
    \item $L=\POmega^+_{2m}(q)$ with $mq$ odd, $A\cap L=N_L(\GL_{m}(q)/\langle -I\rangle)$ and $B\cap L=N_1$.
\end{enumerate}
  \end{lemma}
  \begin{proof}
  By~\cite{Reg}*{Lemma~2.6} with $B^*=B$ it follows that  $G^*=(G^*\cap A)B$. Suppose now that $L\ne\POmega_8^+(3)$. 
  Since $L$ is a classical group, and $M_{12}$ is not, \cite{Reg}*{Lemma~2.6} shows that either
 the triple $(L,A\cap L, B\cap L)$ is given by parts~(a) or~(b), or $G^*=(G^*\cap A)B$ is a maximal core-free 
  factorisation. Assume that the latter holds.  It remains to prove that $B=N_G(B\cap L)$. Since $L\lhdeq G$ we have that $B\leqslant N_G(B\cap L)$. Since $B$ is maximal in $G^*$ it follows from the argument in the proof of \cite{Wilson1985}*{Lemma~2.1} that $B\cap L\neq 1$. Hence $L\not\leqslant N_G(B\cap L)$. Thus $N_G(B\cap L)$ is core-free in $G$ and so the maximality of $B$ among core-free subgroups of $G$ implies that $B=N_G(B\cap L)$.
  \end{proof}

We also need the following number theoretic lemma.

  \begin{lemma}\label{L:NT}
 If $q=p^f$ where $p$ is prime, $m\ge2$ and $f\ge1$, then either $m=2$, $f=1$ and $p\in\{3,7\}$ or
  $(q^m-1)\nmid 2mf(q-1)^2$.
\end{lemma}

\begin{proof}
Let $N=(q^m-1)/(q-1)$ and $d=(N,q-1)$.  Suppose that $N\mid 2mf(q-1)$.  
  Since $N/d$ divides $2mf(q-1)/d$ and
  $N/d$ and $(q-1)/d$ are coprime, we see that
  $N/d$ divides $2mf$. However, $N=\sum_{i=0}^{m-1}q^i\equiv m\mod (q-1)$,
  so it follows that $d=(m,q-1)$ and $N$ divides $2m^2f$.  Thus
  \[2^{(m-1)f}\le p^{(m-1)f}= q^{m-1}<\frac{q^m-1}{q-1}\le 2m^2f.\]
  However, $2^{(m-1)f}\le 2m^2f$ has only 13 solutions for $(m,f)$. Since $11^{(m-1)f}\le 2m^2f$ has no solutions, we see that $p\in\{2,3,5,7\}$. Of the $52$ possibilities for $(m,f,p)$, the only possibilities that satisfy $(p^{fm}-1)\mid 2mf(p^f-1)^2$ have $m=2$, $f=1$ and $p\in\{3,7\}$. 
\end{proof}

\section{Linear groups}\label{S:PGLn}

We begin with the following generalisation of Hering's Theorem. For more information about the case $H\le \GammaL_1(q^n)$ in Theorem~\ref{T:SLa}(a), see Foulser~\cite{Foulser}.
We note that the exceptional case Theorem~\ref{T:SLa}(b) was missed in \cite{Reg}*{Lemma 4.1}.

\begin{theorem}\label{T:SLbigN}
  Suppose that $H\leqslant \GammaL_n(q)$ where $n\ge2$ such that $\SL_n(q)\not\leqslant H$. Let $V=(\F_q)^n$ and suppose that $H$ acts transitively on the set of all $k$-subspaces of $V$. Then
\begin{enumerate}[{\rm (a)}]
\item $k=1$ or $n-1$, and either $H\le\GammaL_1(q^n)$ or $H_0\lhdeq H$ where $(H_0,n,q)$ are listed in Table~$\ref{T:SLa}$; or
\begin{table}[!ht]
\caption{Theorem~\ref{T:SLbigN}(a) subgroups $H_0\lhdeq H\le\GammaL_n(q)$ where $H$ is transitive on $1$-subspaces or $(n-1)$-subspaces.}\label{T:SLa}
\begin{tabular}{rlccccc}
\toprule
$H_0$ &&$\SL_a(q^b)$& $\Sp_a(q^b)$& $\G_2(q^b)'$& $Q_8=2^{1+2}_{-}$&$\SL_2(5)$\\
\textup{$n$}&&$ab>b$&$ab$,\textup{ $a$ even}&$6b$&$2$&$2$ \\
\textup{$q$}&&${\rm all}$&${\rm all}$&\textup{even}&$\{5,7,11,23\}$&$\{9,11,19,29,59\}$\\
\midrule
$H_0$ && $\{\Alt_6, \Alt_7\}$& $2^{1+4}_{-}\ldotp\C_5$&$2^{1+4}_{-}\ldotp\Alt_5$& $\SL_2(5)$&$\SL_2(13)$ \\
\textup{$n$}&&$4$&$4$&$4$&$4$&$6$\\
\textup{$q$}&&$2$&$3$&$3$&$3$&$3$\\
\bottomrule
\end{tabular}
\end{table}
\item $k=2$ and either  $(n,q)=(4,2)$ with $H=\Alt_7$, or $(n,q)=(5,2)$ and $H=\GammaL_1(2^5)$; or
\item $k=3$ and $(n,q)=(5,2)$ with $H=\GammaL_1(2^5)$.
\end{enumerate}
\end{theorem}
\begin{proof}
If $n=2$ then clearly $k=1$. If $n\geq 3$ then $\PSL_n(q)$ is simple and according to~\cite{Reg}*{Lemma~4.1} we have that $k=1$, $n-1$, or $(n,q)=(5,2)$ and $k=2$ or $3$. However, the proof there has missed $G=\PSL_4(2)\cong\Alt_8$ acting on $\cP_2$, since alternating groups such as $\Alt_8$ are treated in \cite{Factns}*{Theorem D} rather than in \cite{Factns}*{Tables~1 and~3}. In this case a \textsc{Magma} calculation reveals a unique example $H=\Alt_7$, as in part (b).  When $k=1$ and $n\geq 2$, the possible transitive subgroups $H$ were determined by Hering \cites{H1,H2} and listed in \cite{Reg}*{Lemma~3.1}. For $(n,q)=(4,3)$ a \textsc{Magma} calculation shows that either $\SL_2(5)\lhdeq H$, or $H$ equals $2^{1+4}_{-}\ldotp X$ where $X\in\{\C_5,D_{10},F_{20},\Alt_5,\Sym_5\}$, so $\C_5$ or $\Alt_5$ is a normal subgroup of~$X$.
The same possibilities for $H$ occur when $k=n-1$
since  the inverse-transpose automorphism  interchanges 1-subspaces and $(n-1)$-subspaces. For the same reason we only need to consider $k=2$ when $(n,q)=(5,2)$. Since there are exactly $155$ 2-dimensional subspaces and $\GL(5,2)=\SL(5,2)\cong\PSL(5,2)$, it follows from \cite{Factns} that $H=\C_{31}\rtimes \C_5=\GammaL_1(2^5)$.
\end{proof}

\begin{remark}\label{R:Sp2}
Observe that $\Sp_2(q)=\SL_2(q)$ and $\Gamma\Sp_2(q)=\GammaL_2(q)$. Moreover, all 1-dimensional subspaces of $V=(\F_q)^2$ are totally isotropic. Thus Theorem~\ref{T:SLbigN} in the case $n=2$ also provides all the subgroups of $\Gamma\Sp_2(q)$ that are transitive on the set of all totally isotropic $1$-subspaces.
\end{remark}

We  use Theorem~\ref{T:SLbigN} to deal with some small dimensional orthogonal groups in Subsection~\ref{S:Osmall} and the two-dimensional unitary  groups in Subsection~\ref{S:PGU2}.

\subsection{Some small dimensional orthogonal groups}\label{S:Osmall}

Isomorphisms between some small dimensional orthogonal groups and other classical groups make it more convenient to treat 
such groups in tandem with the other classical groups to which they are isomorphic.  Since the $2$-dimensional orthogonal groups are soluble, 
we only consider dimensions at least $3$.  Apart from the generic isomorphisms $\POmega_{2m+1}(q)\cong\PSp_{2m}(q)$ for all $m$ when $q$ is even, the exceptional isomorphisms involve groups of dimensions up to~$6$ and are the following: 
\begin{align*}
  &\POmega_3(q)\cong\PSL_2(q),&&\POmega_4^{+}(q)\cong\PSL_2(q)^2, &&\POmega_4^{-}(q)\cong\PSL_2(q^2),\\
  &\POmega_5(q)\cong\PSp_4(q)\; (q\;\;{\rm odd}),&&\POmega_6^{+}(q)\cong \PSL_4(q),&&\POmega_6^{-}(q)\cong \PSU_4(q).
\end{align*}
We omit the non-simple case $\POmega_4^{+}(q)\cong\PSL_2(q)^2$, as discussed in Subsection~\ref{sub:spaces}. 
 
These isomorphisms of simple groups  influence the subgroup structure of the corresponding automorphism groups, in particular giving useful information about subspace stabilisers. For example,  the exceptional isomorphism $\POmega_6^{+}(q)\cong\PSL_4(q)$ relates a subspace stabilizer 
in the $\cN_1$-action of $\POmega_6^{+}(q)$ on $V=(\F_q)^6$  to a symplectic group $\PSp_4(q)$ acting on $W=(\F_q)^4$ -- not a subspace stabilizer.  The action of $\PSL_4(q)$ on the set of cosets of $\PSp_4(q)$ is not one we are considering for linear groups, and so must be treated specially for the orthogonal groups $\POmega_6^{+}(q)$. On the other hand, exceptional isomorphisms can relate
subspace stabilisers, sometimes for different types of spaces, for example, the stabilisers of subspaces in  $\cP_1(V)$ and $\cP_2(W)$ correspond under the isomorphism $\POmega_6^{+}(q)\cong\PSL_4(q)$.

In Table~\ref{tab:ExceptionalIsos}, we list the subgroups corresponding (under the displayed exceptional isomorphisms) to stabilisers in the orthogonal groups of the relevant totally singular, nondegenerate, and nonsingular $1$-subspaces. 
Justifications for these correspondences will be given in the places where we deal with the individual families. Essentially, the correspondences are determined using the structure of the subspace stabilisers given in~\cite{BHRD}*{Table~2.3, p.\,61}. Parabolic subgroups have a normal subgroup (the unipotent radical) whose order is a power of $q$, so parabolic subgroups correspond to parabolic subgroups and the exponent of $q$ often determines the correspondences $P_i\leftrightarrow P_j$ in Table~\ref{tab:ExceptionalIsos}. For example,  $P_1(\POmega_6^+(q))=P_2(\PSL_4(q))$ has a unipotent radical of  order~$q^4$, and not $q^5$. The remaining cases use~\cite{BHRD}*{Table~2.3} and the maximal subgroups in dimensions $2$ and $4$~\cite{BHRD}*{pp.\,377, 381, 392}.

\begin{table}[ht!]
\centering
\caption{Exceptional isomorphisms and subgroup correspondences. For $\POmega_5(q)\cong\PSp_4(q)$ take $q$ to be odd. The case $q$ even is handled by Theorem~\ref{T:Ooddqeven}.}
  \begin{tabular}{ccccccccc}
  \toprule
   $\POmega_3(q)$&&$P_1$&$N_1^{+}$&$N_1^{-}$&&&  \\
   $\PSL_2(q)$&&$P_1$&$D_{\delta(q-1)}$&$D_{\delta(q+1)}$&$\delta\in\{1,2\}$&\\
   \midrule
   $\POmega_4^{-}(q)$&&$P_1$&$N_1$&$N_2^\pm$&&&  \\
   $\PSL_2(q^2)$&&$P_1$&$\SL_2(q)\ldotp d$&$D_{2(q^2-1)/d}$&$d\kern-2pt=\kern-2pt(q\kern-2pt-\kern-2pt1,\kern-1pt2)$&\\ 
   \midrule
   $\POmega_5(q)$&&$P_1$&$P_2$&$N_1^{+}$&$N_1^{-}$&$N_2^{+}$&$N_2^{-}$  \\
   $\PSp_4(q)$&&$P_2$&$P_1$&$\GO_4^{+}(q)$&$\GO_4^{-}(q)$&$\GL_2(q)\ldotp2$&$\GU_2(q)\ldotp 2$\\
  \midrule
   $\POmega_6^+(q)$&&$P_1$&$P_2$&$P_3$&$N_1$&$N_2^{+}$&$N_2^{-}$& \\
   $\PSL_4(q)$&&$P_2$&$P_1$&$P_3$&$\Sp_4(q)$&$\GL_2(q)\wr\C_2$&$\GL_2(q^2)\ldotp 2$&\\
   \midrule
   $\POmega_6^-(q)$&&$P_1$&$P_2$&$N_1$&$N_2^{+}$&$N_2^{-}$&  \\
   $\PSU_4(q)$&&$P_2$&$P_1$&$\Sp_4(q)$&$\GL_2(q^2)\ldotp 2$&$\GL_2(q)\wr\C_2$&\\
 \bottomrule
\end{tabular}
\label{tab:ExceptionalIsos}
\end{table}

In the remainder of this section we use Theorem~\ref{T:SLbigN} to deal with those orthogonal groups isomorphic to linear groups,
namely $\POmega_3(q), \POmega_4^-(q),$ and $\POmega_6^+(q)$. The groups $\POmega_5(q)$  are dealt with in Section~\ref{S:Sp} for $q$ odd, and in Section~\ref{S:OoddEvenq} for $q$ even, while  the groups $\POmega_6^-(q)$ are handled in Subsection~\ref{SS:O6-}.

Note that for $G=\GO_3(2)$, the set $\cN_2^-$ has size one and so we exclude this case below.

\begin{theorem}\label{T:Omega3}
  Suppose that $H\le\GammaO_3(q)$ acts transitively on a set $\cU$ of subspaces of the natural module $V=(\F_q)^3$ given in Table~\textup{\ref{T-defn}}, with $\Omega_3(q)\not\le H$ and if $q=2$ then $\cU\neq \cN_2^-$. Then one of the following holds, where $\oH$ is the image of $H$ in $\PGammaL_2(q)$ and $q=p^f$ with $p$  prime:
  \begin{enumerate}[{\rm (a)}]
  \item $\cU=\cP_1$ and $H\le N_{\GammaO_3(q)}(\GO_2^-(q))=N_2^-$, or
    $\C_2^2\lhdeq H$ and $q\in\{5,7,11,23\}$, or $\Alt_5\lhdeq H$ and $q\in\{9,11,19,29,59\}$; or
  \item $\cU=\cN_1^{-}$ with $q$ odd,  and $H\le P_1$; or
  \item $\cU=\cN_2^-$ with $q>2$ even,  and either $H\le P_1$, or $q=4$ and $H\leqslant N_G(\Omega_2^+(4))=N_2^+$, or $q=16$ and $\GO_3(4).4\lhdeq H$ where the `$4$\kern-1.5pt' induces a field automorphism of order~$2$ on $\GO_3(4)$; or
  \item $\cU=\cN_2^+$ with $q=4$ and $\GO_2^-(4).2\lhdeq H\leqslant N_2^-$, where the `\kern1pt$2$\kern-1.5pt' acts as a field automorphism.
\end{enumerate}
\end{theorem}

\begin{proof}
  By Table~\ref{T-defn}, one of the following holds: $\cU=\cP_1$, or $q$ is odd and $\cU=\cN_1^+$ or $\cN_1^-$, or $q$ is even and $\cU=\cN_1, \cN_2^+$ or $\cN_2^-$. Moreover, $q\neq 2$ when $\cU=\cN_2^-$. Thus $G=\GammaO_3(q)$ fixes $\cU$ setwise. By~\cite{Taylor}*{11.6 Theorem} we have $\SO_3(q)\cong \PGL_2(q)$ and so $\GammaO_3(q)\cong (Z\times\PGL_2(q))\rtimes \C_f$ where $q=p^f$, $Z$ is the group of scalar matrices and $\PGammaO_3(q)\cong\PGammaL_2(q)$.
  
For each $X\leqslant \GammaO_3(q)$ let $\overline{X}$ denote the image of $X$ in $\PGammaL_2(q)$. Let $H\leqslant \GammaO_3(q)$ act transitively on $\cU$ such that $\Omega_3(q)\not\leqslant H$. We note that $q\geq3$, for if $q=2$ then  $\GammaO_3(q)=\GO_3(2)\cong \Sym_3$ and $\Omega_3(2)\cong \Alt_3$, and since $|\cP_1|=|\cN_1|=|\cN_2^+|=3$ divides $|H|$ it follows that $\Omega_3(2)\lhdeq H$, a contradiction. 

When $q=3$ we have  $|Z|=2$, $\GammaO_3(3)=\CO_3(3)\cong \C_2\times \Sym_4$, and $\Alt_4\cong\Omega_3(3)\not\le H$. Here $|\cP_1|=4$, $|\cN_1^+|=6$ and $|\cN_1^-|=3$. Thus if $\cU=\cP_1$ then $\oH\leqslant D_8\cong \C_4\rtimes \C_2$, as in part (a). If $\cU=\cN_1^-$ then any subgroup of $G$ of order divisible by 3 is transitive on $\cU$ and since $H\not\ge\Omega_3(3)$, it follows that $H\leq \C_2\times \Sym_3=P_1$, as  in part (b). Finally, if $\cU=\cN_1^+$ then a \textsc{Magma} calculation shows that $\Alt_4\cong \Omega_3(3)\lhdeq H$, a contradiction.

From now on we assume that $q\geq4$. Thus $\POmega_3(q)\cong\PSL_2(q)$ is a nonabelian simple group and so by Lemma~\ref{lem:derived}, we have $\PSL_2(q)\not\leqslant \overline{H}$.

  \smallskip
  \textsc{Case $\cU=\cP_1$.} Comparing orders and group structure we see that $\overline{P_1}$ is a maximal parabolic subgroup of $\PGammaL_2(q)$. In particular, $\oH$ acts transitively on the set of 1-dimensional subspaces of a 2-dimensional vector space over $\F_q$. Thus by Theorem~\ref{T:SLbigN} either $\C_2^2$ or $\Alt_5\lhdeq \oH$ for particular values of $q$, or $\oH\leqslant \C_{q+1}\rtimes \C_{2f}$.  Since $\SO_3(q)\cong\PSL_2(q)$ the first two cases imply that $\C_2^2$ or $\Alt_5\lhdeq H$, while comparing \cite{BHRD}*{Tables~8.1 and~8.7} we deduce that  the last case implies  $H\leqslant N_G(\Omega_2^-(q))=N_2^-$. Thus $H$ satisfies part (a) of the theorem.

  \smallskip
  \textsc{Case} $\cU=\cN_1^{-}$ with $q$ odd, or  $\cU=\cN_2^{-}$ with $q>2$ even. In these cases $\GO^-_2(q)\cong D_{2(q+1)}$, and we have $G_U=(Z\times D_{2(q+1)})\rtimes \C_f$, whence $\overline{G_U}=N_{\overline{G}}(D_{\delta(q+1)})$ where $\delta=\gcd(2,q)$. The latter subgroup is maximal in $\oG$ by \cite{BHRD}*{Tables~8.1 and~8.2},  so the conditions of Lemma~\ref{lem:maxmin} hold for the factorisation  
$\oG=\oH N_{\overline{G}}(D_{\delta(q+1)})$. Hence by Lemma~\ref{lem:maxmin} we get a maximal core-free factorisation of an almost simple group $G^*$ with socle $\PSL_2(q)$ and one factor being $N_{G^*}(D_{\delta(q+1)})$. Thus by~\cite{Factns}*{Theorem A, Tables~1 and 3}, either $q=4, 5, 9$ or $16$, or  $\oH\leqslant \overline{P_1}$. The latter case implies that $H\leqslant P_1$ as in part (b) or (c), and so it remains to deal with the small values of $q$.  When $q=4$ we have $|\cN_2^-|=6$ and a \textsc{Magma} calculation implies that either $H\leqslant P_1$ or $\overline{H}\cong \C_6$ or $D_{12}$. Since $\GO_3(4)\cong\Alt_5$ has a unique class of subgroups isomorphic to $\C_3$ and $\Omega_2^+(4)\cong \C_3$, it follows that $H\leqslant N_G(\Omega_2^+(4))=N_2^+$ as in part (c).
When $q=5$ we have $|\cN_1^-|=10$ and so any subgroup of $G$ with order divisible by 10 is contained in $P_1$ as in part (b). When $q=9$ we have $|\cN_1^-|=36$, and all subgroups with order divisible by $9$ are contained in $P_1$ as in part (b). When $q=16$ we have $|\cN_2^-|=120$, and either $H\leq P_1$ or $\overline{H}=(\Alt_5\times 2)\cdot 2$ as in part (c), see~\cite{Factns}*{5.1.1, p.\,91}.  Moreover, by Lemma~\ref{lem:derived}, $\Alt_5\cong \GO_3(4)\lhd H$. Since $N_G(\GO_3(4))=(\GO_3(4)\times Z).4$ where the 4 induces a field automorphism of order 2 on $\GO_3(4)$, it follows that $\GO_3(4).4\lhdeq H$ as in part (c).

  \smallskip
  \textsc{Case} $\cU=\cN_1^{+}$ with $q$ odd, or  $\cU=\cN_2^{+}$ with $q$ even.  In these cases $\GO_2^+(q)\cong D_{2(q-1)}$, and we have  $\overline{G_U}=N_{\overline{G}}(D_{\delta(q-1)})$ where $\delta=\gcd(2,q)$. By~\cite{BHRD}*{Tables~8.1 and~8.2},  $\overline{G_U}$ is maximal in $\overline{G}$ unless $q=4$ or $5$.  When $q=4$ we have $|\cN_2^+|=10$ and a \textsc{Magma} calculation shows that $\overline{H}=F_{20}$. Now $\C_5\cong\Omega_2^-(4)\leqslant \GO_3(4)\cong \Alt_5$. Since $G$ has a unique conjugacy class of subgroups of order 5 and $|Z|=3$, it follows that $\Omega_2^-(4)\lhdeq H$. Now $N_G(\Omega_2^-(4))= (\GO_2^-(4)\times  Z).2=N_2^-$ where the 2 acts as a field automorphism on $\GO_2^-(4)$. Since $|Z|=3$ and $\overline{H}=F_{20}$ it follows that $\GO_2^-(4).2\lhdeq H$ as in part (d).  When $q=5$ we have  $\PGL_2(5)\cong S_5$ and $|\cN_1^+|=15$. A \textsc{Magma} calculation then shows that $\Alt_5\cong\Omega_3(5)\lhdeq H$, which is not the case. Hence $q\ne 5$ and  the conditions of Lemma~\ref{lem:maxmin} hold for the factorisation  
$\oG=\oH N_{\overline{G}}(D_{\delta(q-1)})$. By Lemma~\ref{lem:maxmin} we get a maximal core-free factorisation of an almost simple group $G^*$ with socle $\PSL_2(q)$ and one factor being $N_{G^*}(D_{\delta(q-1)})$. However~\cite{Factns}*{Theorem~A, Tables~1 and 3}  implies that there are no such factorisations unless possibly when $q=9$. When $q=9$ we have  $|\cU|=|\cN_1^+|=45$. However, all subgroups of $\PGammaL_2(9)$ with order divisible by $45$ contain $\PSL_2(9)$.  
  
\smallskip
  \textsc{Case} $\cU=\cN_1$ with $q$ even. Here $\cU$ consists of the $q^2-1$ nonsingular $1$-subspaces of $V$,  apart from $V^\perp$. For $U\in\cU$, the $2$-subspace $U+V^\perp$ contains a unique totally singular $1$-subspace, say $W$, and $G_{U+V^\perp}=G_W$ contains $G_U$.  It follows that 
  $G$ acts imprimitively on $\cU$ inducing an action on a system of imprimitivity that is permutationally isomorphic to its action on $\cP_1$. In particular, 
  the subgroup $H$ is transitive on $\cP_1$ since it is transitive on $\cU$. Then by part (a), and since $q$ is even, we have $\oH\le \C_{q+1}\rtimes \C_{2f}$.
  However, since the scalar subgroup $Z$ acts trivially on $\cU$ the size $|\cU|=q^2-1$ must divide $|\oH|$ and hence must divide $2f(q+1)$. Thus $q-1$ divides $2f$,  a contradiction.   
  \end{proof}

\begin{theorem}\label{T:KleinPSL2q2}
  Suppose that $H\leqslant \GammaO_4^{-}(q)$ acts transitively on a set $\cU$ of subspaces of the natural module $V=(\F_q)^4$ given by Table~\textup{\ref{T-defn}},  and where $\Omega_4^{-}(q)\not\le H$.
   Then one of the following holds, where $\oH$ is the image of $H$ in $\PGammaL_2(q^2)$:
  \begin{enumerate}[{\rm (a)}]
  \item $\cU=\cP_1$ and $H\leqslant N_{\GammaO^-_4(q)}(\Omega_2^-(q^2))$,  or $q=3$ and  $\Alt_5\lhdeq H$; or
  \item $\cU=\cN_1^\pm$, $q=3$ and $\Alt_5\lhdeq H$ (one class of subgroups for each of $\cN_1^+$ and $\cN_1^-$); or $\cU=\cN_1$ and either $q=2$ and  $H=\GO_2^-(4).2$ where the $`2$\kern-1.5pt' acts as a field automorphism, or   $q=4$ and $\GO_2^-(4^2).4\lhdeq H$ where the `$4$\kern-1.5pt' acts as a field automorphism; or
  \item $\cU=\cN_2^\pm$, $q=2$ and  $H=\GO_2^-(4).2$ where the `$2$\kern-1.5pt' acts as a field automorphism.
\end{enumerate}
\end{theorem}

\begin{proof}
By Table~\ref{T-defn}, $\cU$ is one of $\cP_1$, $\cN_1^\eps$ (with $q$ odd and $\eps\in\{+,-\}$),  $\cN_1$ (with $q$ even), or $\cN_2^\eps$ (with $\eps\in\{+,-\}$).  Fix $U\in\cU$ and set $G=\textup{Stab}_{\GammaO_4^{-}(q)}(\cU)$. Then $G=\GammaO_4^-(q)$ unless $q$ is odd and $\cU=\cN_1^\eps$, in which case $G$ is an index two subgroup of $\GammaO_4^-(q)$.

By~\cite{KL}*{Proposition~2.9.1} we have $\Omega_4^-(q)\cong \POmega_4^-(q)\cong \PSL_2(q^2)$ and we deduce that $\PGammaO_4^-(q)\cong \PGammaL_2(q^2)$. For  $X\leqslant \GammaO^-_4(q)$ we denote the image of $X$ in $\PGammaL_2(q^2)$ by $\overline{X}$.
  
  \smallskip
(a) Suppose first that $\cU=\cP_1$. Comparing orders and group structure, we find that, for $P_1=G_U$, the image $\overline{P_1}$ is a maximal parabolic subgroup of $\PGammaL_2(q^2)$. In particular $\oH$ acts transitively on the set of 1-dimensional subspaces of a 2-dimensional vector space over $\F_{q^2}$. Thus  $\oH$ is given by Theorem~\ref{T:SLbigN}, and so, since  $ \PSL_2(q^2)\not\le \oH$,  either $\oH\leqslant \C_{q^2+1}\rtimes \C_{4f}$ or $q=3$ and $\Alt_5\lhdeq \oH$. Since $\Omega_4^-(q)\cong\PSL_2(q^2)$ the latter implies that $\Alt_5\lhdeq H$,  while comparing \cite{BHRD}*{Tables~8.1 and~8.17} we deduce from the former that $H\leqslant N_G(\Omega_2^-(q^2))$, so part (a) holds.

\smallskip
(b) Next suppose that $\cU=\cN_1$ with $q$ even, or $\cN_1^\eps$ with $q$ odd. Here $G=\GammaO_4^-(q)$ when $q$ is even, 
and is an index two subgroup when $q$ is odd. Since $\Omega_3(q)\cong\PSL_2(q)$, we see from \cite{BHRD}*{Table~8.17} that $\overline{G_U}=N_{\oG}(\PSL_2(q))$, which is maximal in $\oG$ unless $q=2$. For $q=2$ we have $\GammaO_4^-(2)=\SO_4^-(2)\cong \PGL_2(4)\cong \Sym_5$ and $|\cN_1|=10$; and a \textsc{Magma} calculation shows, since $\Omega_4^-(2)\not\le H$, that $H\cong F_{20}$. Since $\SO_4^-(2)$ has a unique conjugacy class of subgroups of order 5, it follows that $H=\GO_2^-(4).2$ where the `$2$' acts as a field automorphism, as in (b). When $q>2$, the conditions of  Lemma~\ref{lem:maxmin} hold for the factorisation $\oG=\overline{G_U}N_{\oG}(\PSL_2(q))$, and so  there is maximal core-free factorisation of an almost simple group $G^*$ with socle $\PSL_2(q^2)$ with one factor being $N_{G^*}(\PSL_2(q))$.  Thus by \cite{Factns}, either $q=3$, or $q=4$ and $\oH=N_{\oG}(D_{34})=\C_{17}\rtimes \C_{8}$. When $q=4$ the group $Z$ of scalar matrices has order 3 and so $\C_{17}\cong\Omega_2^-(4^2)\lhdeq H\leqslant (\GO_2^-(4^2)\times Z).4$ where the $4$ acts a field automorphism of $\GO_2^-(4^2)$. Since  $\oH=\C_{17}\rtimes \C_{8}$ it follows that $\GO_2^-(4^2).4\lhdeq H$ as in part (b). For $q=3$, a \textsc{Magma} calculation shows that $\Alt_5\lhdeq\oH$, and since $\Omega_4^-(3)\cong\PSL_2(9)$ it follows that $\Alt_5\lhdeq H$. Moreover, $G$ has two conjugacy classes of $\Alt_5$-subgroups fused in $\CO_4^-(q)$, and each class of subgroups acts transitively on only one of the two sets $\cN_1^+$ or $\cN_1^-$.  Thus part (b) holds

\smallskip
(c) Finally consider $\cU=\cN_2^\eps$. If $U\in\cN_2^+$ then $U^\perp\in \cN_2^-$. Thus $H$ is transitive on $\cN_2^+$ if and only if it is transitive on $\cN_2^-$, and so we assume that $\cU=\cN_2^+$. By~\cite{BHRD}*{Table~8.17} we have that $\overline{G_U}=N_{\oG}(D_{q^2-1})$ and is maximal in $\oG$ unless $q=3$.  Thus by Lemma~\ref{lem:maxmin} and  \cite{Factns} we have $q=2$ or $3$.  When $q=2$ we have $|\cN_2^+|=10$ and by a \textsc{Magma} calculation, since $\Alt_5\cong\Omega_4^-(5)\not\leqslant H$,  we see that $H\cong F_{20}$. We deduce as above that $H=\GO_2^-(4).2$ as in part (c). When $q=3$ we have $|\cN_2^+|=45$ and a \textsc{Magma} calculation reveals that there are no possibilities for $H$.
\end{proof}
  
 Finally in this section we deal with $\POmega_6^+(q)\cong \PSL_4(q)$. We note that part (c) of the following result was proved in \cite{Reg}*{Lemma 4.5} while part (a) for $\cU=\cP_1$ was proved in \cite{KLrank3}*{Lemma 5.15}.  Recall that $\cP_{m/2}^+$, $\cP_{m/2}^-$ denote the two orbits of $\Omega^+_{2m}(q)$ on the family $\cP_m$ of totally singular $m$-subspaces.  Also, given a group $X$, we use $X^R$ to denote a reducible copy of $X$ in $\GammaO_6^+(q)$.
  
\begin{theorem}\label{T:KleinSL4}
  Suppose that $H\le\GammaO_6^{+}(q)$  acts transitively on a set $\cU$ of subspaces, given by Table~\textup{\ref{T-defn}}, of the natural module $V=(\F_q)^6$ and $\Omega_6^{+}(q)\not\le H$.  Then one of the following holds:
  \begin{enumerate}[{\rm (a)}]
\item $\cU=\P_1$ or $\cP_2$, $q=2$ and $\Alt_7\lhdeq H$; or

  \item $\cU=\cP_3^\eps$ and either $H\leqslant N_{A}(\Omega_2^-(q^2))\leqslant N_4^-$, where $A$ denotes the setwise stabiliser in $\GammaO_6^+(q)$ of $\cP_3^+$, or $H_0\lhdeq H$ where $(H_0,q)$ is given in Table~$\ref{Tab:abKleinSL4}$; or
  
  \begin{table}[!ht]
\caption{Theorem~\ref{T:KleinSL4}(b) Subgroups $H_0\lhdeq H\le\GammaO_6^{+}(q)$ where $H$ is transitive on $\cP_3^+$ or $\cP_3^-$.}\label{Tab:abKleinSL4}
  \begin{threeparttable}
\begin{tabular}{rlccccccc}
\toprule
$H_0$ &&$\Omega_4^-(q)^R$  &$\Omega_5(q)^R$ &$\Alt_7$&$2^4\ldotp\C_5$&$2^4\ldotp\Alt_5$&$\Alt_5$\\
\textup{$q$}&&${\rm all}$&${\rm all}$&$2$&$3$& $3$& $3$\\
\bottomrule
\end{tabular}
\begin{tablenotes}
\item $X^R$ indicates that the group $X$ acts reducibly on $V$.
\end{tablenotes}
\end{threeparttable}
\end{table}

\item $\cU=\cN_1^\eps$ (with $q$ odd and $\eps\in\{+,-\}$), or $\cN_1$ (with $q$ even), and either $\SL_3(q)\lhdeq H$, or $H$ stabilises a totally singular $3$-subspace and modulo its unipotent radical induces a subgroup of $\GammaL_3(q)$ that is transitive on $1$-subspaces; or
\item $\cU=\cN_2^{-}$ and either $[q^3]\rtimes\SL_3(q)\lhdeq H$, or  $H_0$ is given by Table~$\ref{T:O6+N2-}$.

\begin{table}[!ht]
\caption{Theorem~\ref{T:KleinSL4}(d) Subgroups $H_0\lhdeq H\le\GammaO_6^{+}(q)$ where $H$ is transitive on $\cN_2^-$.}\label{T:O6+N2-}
\begin{tabular}{rlcccccc}
\toprule
$H_0$ && $\SL_3(2)$& $2^3\rtimes \C_7$ & $3^3\rtimes\GammaL_1(3^3)$ & $\SL_3(4)$ &$[4^3]\rtimes\GammaL_1(4^3)$\\
\textup{$q$}& &$2$&$2$&$3$&$4$&$4$\\
\bottomrule
\end{tabular}
\end{table}

\end{enumerate}
\end{theorem}

\begin{proof}
By Table~\ref{T-defn}, $\cU$ is one of $\cP_1$, $\cP_2$, $\cP_3^\eps$, or $\cN_2^\eps$ (with $\eps\in\{+,-\}$), or  $\cN_1^\eps$ or $\cN_3^\eps$ (with $q$ odd and $\eps\in\{+,-\}$), or $\cN_1$ (with $q$ even). 
Fix $U\in\cU$ and set $G=\textup{Stab}_{\GammaO_6^{+}(q)}(\cU)$. Then $G=\GammaO_6^+(q)$ unless $\cU=\cP_3^\eps$, or $q$ is odd and $\cU=\cN_1^\eps$ or $\cN_3^\eps$, in which case $G$ is an index two subgroup of $\GammaO_6^+(q)$.

  By~\cite{Taylor}*{12.18 Theorem}, $\PGammaO^+_6(q)\cong\Aut(\PSL_4(q))$.  For any $X\leqslant \GammaO_6^+(q)$ we denote the image of $X$ in $\Aut(\PGammaL_4(q))$ by $\overline{X}$. Moreover, if $\cU=\cP_3^\eps$ then it follows from \cite{Taylor}*{12.16 Theorem} that $\oG\cong\PGammaL_4(q)$. We denote the 4-dimensional vector space over $\F_q$ acted on by $\SL_4(q)$ by $W$.

\smallskip
(a1)~$\cU=\cP_1$. By~\cite{Taylor}*{p.\,187}, $\cP_1$ is in one-to-one correspondence with the set of 2-dimensional subspaces of $W$. Thus $H$ acts transitively on $\cP_1$ if and only if $\oH$ acts transitively on the set of 2-dimensional subspaces of $W$. Thus by Theorem~\ref{T:SLbigN}, either $\PSL_4(q)\lhdeq \oH$ or $q=2$ and $\Alt_7\lhdeq H$. In the first case, Lemma~\ref{lem:derived} implies that $\Omega_6^+(q)\lhdeq H$ as $\Omega_6^+(q)$ is quasisimple, and this is a contradiction. Thus part (a) holds for this case. 

\smallskip
(a2)~$\cU=\cP_2$. Comparing composition factors of subspace stabilisers, we see that there is a one-to-one correspondence between the elements of $\cP_2$ and the flags $X<Y$ where $X$ is a 1-dimensional subspace of $W$ and $Y$ is a hyperplane of $W$. By~\cite{BHRD}*{Table~8.8}, the stabiliser of such a flag is maximal in $\oG=\Aut(\PSL_4(q))$. Then by Lemma~\ref{lem:maxmin} and  \cite{Factns} (except potentially when $q=2$ as here $\PSL_4(2)\cong\Alt_8$), $\PSL_4(q)\lhdeq \oH$ and we again have that $\Omega_6^+(q)\lhdeq H$, a contradiction. If $q=2$ then $|\cP_2|=105$ and a \textsc{Magma} calculation reveals that  $\Alt_7\lhdeq H$  is the unique example, as in part (a).

\smallskip
(b)~$\cU=\P_3^\eps$. In this case $\oG=\PGammaL_4(q)$ and $G=A$. Moreover,  if $\tau\in\GammaO^+_6(q)\backslash A$ then $H$ is transitive on $\cP_3^+$ if and only if $H^\tau$ is transitive on $\cP_3^-$, so we may assume that $\eps=+$.  By~\cite{Taylor}*{12.16 Theorem}, $\cP_3^+$ is in one-to-one correspondence with the set of 1-dimensional subspaces of $W$ (and $\cP_3^-$ is in one-to-one correspondence with the set of 3-dimensional subspaces of $W$). Thus $\oH$ is transitive on the set of 1-dimensional subspaces of $W$ and so is provided by Theorem~\ref{T:SLbigN}. We note that $\Sp_4(q)$ appears in $\GammaO_6^+(q)$ as a copy of $\Omega_5(q)$ fixing a hyperplane (see \cite{BHRD}*{Tables~8.31 and~8.32}).  Also, the $\SL_2(q^2)$-subgroup of $\SL_4(q)$ that is transitive on 1-subspaces is a $\mathcal{C}_3$-subgroup of $\SL_4(q)$ and then comparing \cite{BHRD}*{Tables~8.8 and~8.31} we see that it occurs as $\Omega_4^-(q)$ as a subgroup of $\Omega_6^+(q)$ fixing a nondegenerate 4-subspace. Also note that $\Omega_6^+(q)\cong\SL_4(q)/\langle \pm I_4\rangle$ and so the examples $2_-^{1+4}\ldotp\C_5$, $2_-^{1+4}\ldotp\Alt_5$ and $\SL_2(5)$ and appearing in Table~\ref{T:SLa} for $q=3$ become $2^4\ldotp\C_5$, $2^4\ldotp\Alt_5$ and  $\Alt_5$ in this case. Furthermore, $\Alt_6$ occurs as $\Omega_5(2)^R\leqslant \Omega_6^+(2)$. Finally, for $\GammaL_1(q^4)\leqslant \GammaL_4(q)$, note that $\GammaL_1(q^4)\leqslant \GammaL_2(q^2)\leqslant\GammaL_4(q)$. Moreover, comparing \cite{BHRD}*{Tables~8.1 and~8.31} we see that the middle subgroup in this chain yields an $N_2^-$ subgroup of $A$ as $\Omega_4^-(q)\cong\PSL_2(q^2)$. Then as $D_{2(q^2+1)}\cong\GO_2^-(q^2)$ and $D_{2(q+1)}\cong\GO_2^-(q)$ we deduce that $H\leqslant N_A(\Omega_2^-(q^2))\leqslant N_4^-$.
Thus case (b) holds.

  \smallskip
  (c)~$\cU=\cN_1^\eps$ (with $q$ odd and $\eps\in\{+,-\}$), or $\cN_1$ (with $q$ even). This case was dealt with in \cite{Reg}*{Lemma~4.5} and $H$ is as given in part (c).

  \smallskip
  (d1)~$\cU=\cN_2^{-}$. Here $G=\GammaO_6^+(q)$ and since $\Omega^-_4(q)\cong\PSL_2(q^2)$ we deduce from \cite{BHRD}*{Table~8.8} that $\overline{G_U}=N_{\oG}(\PSL_2(q^2))$, which is maximal in $\oG$. Let $\oB$ be a subgroup of $\Aut(\PSL_4(q))$ that contains $\oH$ and is maximal subject to being core-free. Then by Lemma~\ref{lem:maxmin} and  \cite{Factns} we have that $\overline{B}$ is the stabiliser of a 1-dimensional subspace of $W$, a 3-dimensional subspace of $W$, or  a decomposition of $W$ into a 1-dimensional space and a complementary 3-subspace. Moreover, in the third case $q=2$ or $4$. 
  
  Suppose that we are in one of the first two cases. As discussed in case (b), 1-dimensional subspaces and 3-dimensional subspaces of $W$ correspond to totally singular 3-subspaces of $V$. Thus $H$ is contained in the stabiliser $B$ of a totally singular 3-subspace $E$ of $V$ and we have a factorisation $B=HB_U$. Note that $U\cap E=0$ and following \cite{Factns}*{(3.6.2b), p.\,68} we see that $Y:=E\cap U^\perp$ has dimension 1 so that $V=U^\perp + E$. Moreover $Y$ is $B_U$-invariant, so $H$ acts transitively on the 1-dimensional subspaces of~$E$. Thus by Theorem~\ref{T:SLbigN}, modulo its unipotent radical $H$ induces either $\SL_3(q)$ or a subgroup of $\GammaL_1(q^3)$ on $E$. Now 
  $$|\cN_2^-|=\frac{|\GO^+_6(q)|}{|\GO_2^-(q)||\GO_4^-(q)|}=\frac{1}{2}q^4(q^3-1)(q+1)$$
  divides $|H|$, and $H\leqslant [q^3]\rtimes\GammaL_3(q)$. Thus $|H^E|$ is divisible by $q/(2,q)$. Suppose that $H^E\leqslant \GammaL_1(q^3)$. Thus $q/(2,q)$ divides $3f$ and so $q\leq 4$. Similarly, if $\SL_3(q)\leqslant H^E\leqslant \GammaL_3(q)$ and $q>4$ then $q^4/(2,q)$ does not divide $|H^E|$ and so $H\cap R\neq 1$, where $R$ is the unipotent radical of $B$. Since $\SL_3(q)$ acts irreducibly on $R$ (view it as a subgroup of $\SL_4(q)$) it follows that either $q\leqslant 4$, or $R\rtimes \SL_3(q)\leqslant H$ as in part (d). A \textsc{Magma} calculation then reveals the list of groups given by Table~\ref{T:O6+N2-} when $q\leq 4$.

  \smallskip
  
  (d2)~$\cU=\cN_2^{+}(V)$. Since $\POmega_4^+(q)\cong\PSL_2(q)^2$, by comparing \cite{BHRD}*{Table~8.8  and Table~8.31} we deduce that $\overline{G_U}$ is the stabiliser in $\oG$ of a decomposition of $W$ into two complementary $2$-subspaces. The usual argument combined with \cite{Factns} shows that (except possibly for $q=2$ as $\PSL_4(2)\cong \Alt_8$), no possibility for $H$ occurs.  A \textsc{Magma} calculation reveals that in the exceptional case $q=2$ there are no examples either.

\smallskip
(e)~$\cU=\cN_3^{\eps}(V)$ (with $q$ odd and $\eps\in\{+,-\}$). By comparing \cite{BHRD}*{Tables~8.8 and~8.31}, $\overline{G_U}$ is an index two subgroup of $N_{\oG}(\PSL_2(q)^2)$, which is maximal in $\oG$ and which, as a subgroup of $\PGammaL_4(q)$, is of type $\GO_4^+(q)$. Let $\oB$ be a subgroup of $\Aut(\PSL_4(q))$ that contains $\oH$ and is maximal subject to being core-free. Then by Lemma~\ref{lem:maxmin} we have a maximal core-free factorisation of an almost simple group with socle $\PSL_4(q)$ where one of the factors has socle $\PSL_2(q)^2$. However, by  \cite{Factns} there is no such factorisation for $q$ odd.  \end{proof}

\subsection{Unitary groups in small dimension}\label{S:PGU2}

Since $\PSL_2(q)\cong\PSU_2(q)$ we now deal with unitary groups in dimension $n=2$. Let $K=\F_{q^2}$ where $q=p^r$ for some prime~$p$, let $k=\F_{q}\leqslant K$,  and  let $V$ be a 2-dimensional vector space over $K$. Let $(\cdot,\cdot)$ be a nondegenerate hermitian form on $V$ and let $\{e,f\}$ be a basis of $V$ such that $(e,e)=(f,f)=0$ and $(e,f)=1$.  Choose $0\ne\zeta\in K$ such that $\zeta^q+\zeta=0$. We take a new basis $\{e, f'\}$ for $V$ where $f'=\zeta^q f$, and note that the Gram matrix for $(\cdot,\cdot)$ with respect to this new basis is $\zeta J$ where $J=\left(\begin{smallmatrix}0&1\\ -1&0\end{smallmatrix}\right)$. Hence $(\alpha e+\beta f',\gamma e+\delta f')=(\alpha\delta^q-\beta\gamma^q)\zeta$.

Consider the group $\GammaU_2(q)$ of all semisimilarities of $(\cdot,\cdot)$.  Let $A\in \GL_2(q^2)$.  Then $A$ lies in $\CU_2(q)$ precisely when $AJA^{\sigma T}=\delta_A J$ for some $\delta_A\in K^\times$, where $\sigma$ is the automorphism of $\GL_2(q^2)$ that raises each entry of a matrix to its $q^{\mathrm{th}}$-power. A simple calculation shows that
$AJA^{\sigma T}=\det(A)J$ for all $A\in\GL_2(k)$.  Thus $\GL_2(k)\leqslant \CU_2(q)$ and $\SL_2(k)\leqslant \GU_2(q)$.  Comparing orders shows that $\SL_2(k)=\SU_2(q)$.  Note that for $A=\lambda I_2$ with $\lambda\in K^\times$, we have $\delta_A=\lambda^{q+1}\in k^\times$. Thus for each $A\in\GL_2(k)$, there exists $\lambda\in K^\times$ such that $\lambda A\in\GU_2(q)$. Letting $Z$ be the subgroup of all $K$-scalars we have $\GL_2(k)Z\leqslant \CU_2(q)$ and comparing orders we get equality. In particular,  $\PGL_2(q)\cong \PCU_2(q)= \PGU_2(q)$.

Let $\phi$ be the transformation of $V$ given by $(\alpha e+\beta f')^\phi=\alpha^p e+\beta^p f'$. Then $\phi$ is a semisimilarity with $(u^\phi,v^\phi)=\zeta^{1-p}(u,v)^p$ for all $u,v\in V$.  Thus $\GammaU_2(q)=\langle \CU_2(q),\phi\rangle$ is a group of order $2r|\CU_2(q)|$. Note that $\phi$ normalises the subgroup $\GL_2(k)$ and induces a field automorphism  of order $r$. In particular, letting $z=\phi^r$ we see that $z$ centralises $\GL_2(k)$. Thus we have.  
\begin{equation}\label{eq:pi}
\begin{array}{ll}
&\mbox{$G=\langle \GL_2(k),Z,\phi\rangle=\GammaU_2(q)$, $\overline{G}=\PGammaU_2(q)$, and  
$\overline{G}/\langle z\rangle\cong \PGammaL_2(q)$, and}\\
&\mbox{let $\pi\colon\PGammaU_2(q)\to \PGammaL_2(q)$ be the natural epimorphism with kernel $\langle z\rangle$. }
\end{array}
\end{equation}

\begin{lemma}\label{lem:quotient}
Let $U$ be a $1$-subspace of $V=K^2$, let $G, \overline{G}$, and
$\pi$ be as in $\eqref{eq:pi}$.
\begin{enumerate}[{\rm (a)}]
\item If $U$ is totally isotropic then $\pi(\overline{G}_U)$ is the stabiliser in $\PGammaL_2(q)$ of a $1$-subspace of~$k^2$.
\item If $U$ is nondegenerate then $\pi(\overline{G}_{\{U,U^\perp\}})=N_{\PGammaL_2(q)}(D_{2(q+1)})\cong \C_{q+1}\rtimes \C_{2r}$.
\end{enumerate}
\end{lemma}

\begin{proof}
(a)~ Since $\overline{G}$ is transitive on the set of totally isotropic $1$-subspaces of $V$ we may assume that $U=\langle e\rangle_K$. Since $Z$ and $\phi$ fix $U$ it follows that $\overline{G}_U=\langle \GL_2(k)_U,Z,\phi\rangle$. Noting that $\GL_2(k)_U=\GL_2(k)_W$, where $W=\langle e\rangle_k$, we deduce that $\pi(\overline{G}_U)=\PGammaL_2(q)_W$.

  (b)~Note that the totally isotropic $1$-subspaces of $V$ are  $\langle e\rangle$ and $\langle \lambda e+f'\rangle$ for $\lambda\in k$. Thus the nondegenerate $1$-subspaces are those of the form $\langle \lambda e+f'\rangle$ with $\lambda \in K\backslash k$. Note that $\langle \lambda e+f'\rangle^\perp=\langle \lambda^q e+f'\rangle$ and so the element $z$ interchanges each nondegenerate $1$-subspace with its orthogonal complement. Thus $z$ lies in the kernel of the action of $G$ on pairs of orthogonal nondegenerate 1-subspaces. Note that there are $q^2-q$ nondegenerate 1-subspaces in $V$ and so there are $q(q-1)/2$ orthogonal pairs. Moreover,  $\overline{G}$ acts transitively on the set of nondegenerate 1-subspaces. Let $U=\langle v_1\rangle$ and $U^\perp=\langle v_2\rangle$, be nondegenerate $1$-subspaces over $K$. Then $\{v_1\,v_2\}$ is a basis for $V$ and we see that $\GU_2(q)_{\{U,U^\perp\}}\cong \C_{q+1} \,\Wr\, \C_2$. Moreover, the map 
  \[
  \varphi:\lambda_1 v_1+\lambda_2v_2\mapsto \lambda_1^pv_1+\lambda_2^pv_2
  \] 
  is a semisimilarity of order $2r$ that fixes $U$ and $U^\perp$. Thus 
  \[
  \GammaU_2(q)_{\{U,U^\perp\}}= \GU_2(q)_{\{U,U^\perp\}}Z \rtimes\langle \varphi\rangle
  \]
  and so $\pi(\overline{G}_{\{U,U^\perp\}})\cong \C_{q+1}\rtimes \C_{2r}=N_{\PGammaL_2(q)}(D_{2(q+1)})$.
\end{proof}

As discussed in Subsection~\ref{sub:spaces}, the $\overline{G}$-orbits on subspaces that we consider are  $\cP_1$ and $\cN_1$. 
We now determine the subgroups of $G$ that are transitive on $\cP_1$ or $\cN_1$. 

\begin{theorem}\label{T2}
  Let $G, \overline{G}$, and $\pi$ be as in $\eqref{eq:pi}$, and suppose that $H\le G$ acts transitively on the set $\cU = \cP_1$ or $\cN_1$ of subspaces  of the natural module $V=(\F_{q^2})^2$, such that
  $\SU_2(q)\not\leqslant H$ and $q\geqslant 2$. 
Then either
  \begin{enumerate}[{\rm (a)}]
  \item $\cU=\cP_1$ and one of $\pi(\overline{H})\leqslant N_{\PGammaL_2(q)}(D_{2(q+1)})$, or $\C_2^2\lhd \pi(\overline{H})$ with $q\in\{5,7,11,23\}$, or $A_5\lhd \pi(\overline{H})$ with $q\in\{9,11,19,29,59\}$;
 or
  \item $\cU=\cN_{1}$ and $H \leqslant P_1$, or $q=16$ and $\pi(\overline{H})=N_{\PGammaL_2(16)}(\PSL_2(4))$, or $q=2$.
  \end{enumerate}
\end{theorem}

\begin{proof}
If $\PSU_2(q)=\pi(\overline{H})$, then Lemma~\ref{lem:derived} would imply that $\SU_2(q)\leqslant H$, contrary to our assumption. Hence $\PSU_2(q)\not\leqslant\pi(\overline{H})$.

\smallskip
  (a)~Suppose that $\cU=\cP_1$ and let $U\in \cU$. Then $G=G_UH$ and so by Lemma~\ref{lem:quotient} we have $\PGammaL_2(q)=\PGammaL_2(q)_W\pi(\overline{H})$, where $W$ is a 1-subspace of $k^2$.  Thus the possibilities for $\pi(\overline{H})$ are given by Theorem~\ref{T:SLbigN}, and so either $\pi(\overline{H})\leqslant N_{\PGammaL_2(q)}(D_{2(q+1)})$, or $\C_2^2\lhd \pi(\overline{H})$ with $q\in\{5,7,11,23\}$, or $A_5\lhd \pi(\overline{H})$ with $q\in\{9,11,19,29,59\}$, as in part (a).

\smallskip
  (b)~Suppose that $\cU=\cN_1$ and let $U\in \cU$. Since $H$ acts transitively on $\cN_1$ it also acts transitively on the set of pairs of orthogonal nondegenerate 1-subspaces. Hence by Lemma~\ref{lem:quotient}, $\PGammaL_2(q)=\pi(\overline{H})N_{\PGammaL_2(q)}(D_{q+1})$. If $q\geq 4$ then Lemma~\ref{lem:maxmin} and  \cite{Factns} imply that $\pi(\overline{H})\leqslant P_1$, or $q=16$ and $\pi(\overline{H})=N_{\PGammaL_2(16)}(\PSL_2(4))$, as in part (b). If $q=3$ then a \textsc{Magma} calculation implies that either $\SU_2(3)\leqslant H$ or $H\leqslant P_1$. Finally, if $q=2$ then $|\cN_1|=2$ and so any subgroup $H$ of even order that is not contained in $N_1$ is transitive. 
  \end{proof}

\section{Unitary groups}\label{S:PGUn}

Let $V=(\F_{q^2})^n$  be equipped with a nondegenerate hermitian form.  As discussed in Subsection~\ref{sub:spaces}, the families $\cU$ of subspaces of $V$ that we consider are the sets $\cP_k$ of totally isotropic $k$-subspaces, and $\cN_k$  of nondegenerate $k$-subspaces, for $1\leq k\leq n/2$. For each of these sets $\cU$, it follows from  Witt's Lemma that $\GU_n(q)$ is transitive on  $\cU$, and it can be shown that $\SU_n(q)$ is also transitive on $\cU$, and moreover $\cU$ is invariant under $\GammaU_n(q)$.
In Theorem~\ref{T1} we classify the subgroups of $\GammaU_n(q)$ that are transitive on~$\cU$. Part (a) of this result is essentially a summary of \cite{Reg}*{Lemma~4.3}. Part (b) for $n=4$ and 5, and part (c) for $n=4$ were done in \cite{KLrank3}*{Corollary~5.12} (see Remark~\ref{rem:KL}). We use Theorem~\ref{T1} to deal with the six-dimensional orthogonal groups of minus type.

\begin{theorem}\label{T1}
  Suppose that $H\le\GammaU_n(q)$   such that  $\SU_n(q)\not\le H$. Let  $V=(\F_{q^2})^n$ and suppose that $H$ 
  acts transitively on the set $\cU=\cP_k$ or $\cN_k$ of $k$-subspaces of $V$, where $1\le k\le n/2$. If $n=2$ then $H$ is given by Theorem~$\ref{T2}$, while if $n\geqslant 3$  then one~of
  \begin{enumerate}[{\rm (a)}]
  \item $\cU=\cN_1$, and either $(n,q)=(3,2)$ and $H=3^{1+2}_+\rtimes \C_8$,
    or $n=2m$, $H\leq P_m$ and  modulo the unipotent radical of $P_m$, $H$ induces a subgroup of $\GammaL_m(q^2)$ which
    is transitive on $1$-subspaces, or
    $n=2m$ and $H_0\lhdeq H$ where $H_0$ is given in Table~\textup{\ref{T6}}; or
\begin{table}[!ht]
\caption{Theorem~\ref{T1}(a) subgroups $H_0\lhdeq H\le\GammaU_{2m}(q)$ where $H$ is transitive
on $\cN_1$.}\label{T6}
\begin{tabular}{rlccccccc}
\toprule
$H_0$ &&$\Sp_{2m}(q)$& $\Sp_{2a}(q^2)$& $\Sp_{2a}(q^4)$& $\SL_{m}(q^2)$& $\SL_{a}(q^4)$&\\
\textup{$m$}&&$\ge2$&$2a$&$4a$&$\ge2$ &$2a>2$&&\\
\textup{$q$}&&${\rm all}$&$2,4$&$2$&$2,4$&$2$&&\\
\midrule
$H_0$ && $\G_2(q)'$& $\G_2(q^2)$& $\G_2(q^4)$&$\nonsplit{3_1}{\PSU_4(3)}$& $\nonsplit{3}{M_{22}}$& $3\ldotp \textup{Suz}$\\
\textup{$m$}&&$3$&$6$&$12$&$3$&$3$&$6$\\
\textup{$q$}&&\textup{$q$ even}&$2,4$&$2$&$2$&$2$&$2$\\
\bottomrule
\end{tabular}
\end{table}
    \item $\cU=\cP_1$, and $H_0\lhdeq H$ where $(H_0,n, q)$ are given in
      Table~\textup{\ref{T7}}; or
\begin{table}[!ht]
  \caption{Theorem~\ref{T1}(b) subgroups $H_0\lhdeq H\le\GammaU_n(q)$ where $H$
    is transitive on $\cP_1$.}\label{T7}
\begin{threeparttable}
\begin{tabular}{rlcccccccc}
\toprule
$H_0$ && $3_{+}^{1+2}$& $3_{-}^{1+2}$& $\PSL_2(7)$& $\nonsplit{3}{\Alt_7}$& $\C_{19} {}{}^\dagger$& $\nonsplit{4}{\PSL_3(4)}^{*}$&$\nonsplit{3}{\textrm{J}_3}$\\
\textup{$(n,q)$}&&$(3,2)$&$(3,2)$&$(3,3)$&$(3,5)$&$(3,8)$&$(4,3)$&$(9,2)$\\
\bottomrule
\end{tabular}
  \begin{tablenotes}\footnotesize
  \item [$\dagger$] More information about $H$ is given in the proof.
  \end{tablenotes}
  \end{threeparttable}
\end{table}

\item $\cU=\cP_m$, $n=2m$ and either $\SU_{2m-1}(q)\lhdeq H$,
     or $m=2$ and $H_0\lhdeq H$ where $(H_0,q)$ are listed in Table~\textup{\ref{T8}}.
\begin{table}[!ht]
\caption{Theorem~\ref{T1}(c) subgroups $H_0\lhdeq H\le\GammaU_{4}(q)$ where $H$
    is transitive on $\cP_2$.}\label{T8}
\begin{threeparttable}
\begin{tabular}{rlccccccc}
\toprule
$H_0$ && $ 3_{+}^{1+2}$& $ 3_{-}^{1+2}$&$(\C_3)^3$& $\PSL_2(7){}{}^*$&$\nonsplit{4}{\PSL_3(4)}$&$\nonsplit{3}{\Alt_7}{}{}^*$&$\C_{19}{}{}^{*,\dagger}$\\
\textup{$q$}&&$2$&$2$&$2$&$3$&$3$&$5$&$8$\\
\bottomrule
\end{tabular}
  \begin{tablenotes}\footnotesize
  \item [$*$]  Contained in $N_1$.
  \item [$\dagger$] More information about $H$ is given in the proof.
  \end{tablenotes}
    \end{threeparttable}

\end{table}
  \end{enumerate}
\end{theorem}

\begin{proof}
Let $G=\GammaU_n(q)$ and $H\leqslant G$ be transitive on $\cU$ such that $\SU_n(q)\not\leqslant H$. Let $U\in \cU$. Then we have a factorisation $G=HG_U$.  If $U$ is a nondegenerate subspace of dimension $n/2$ then we also have a factorisation $G=HX$, where $X$ is the stabiliser of the decomposition $V=U\oplus U^\perp$, which has $G_U$ as an index two subgroup. By  \cites{BHRD,KL} $X$ is maximal in $G$ unless $(n,q)=(4,2)$.  When $(n,q)=(4,2)$, computations in \textsc{Magma} show that $\SU_4(2)\lhdeq H$, while for $(n,q)\neq(4,2)$,  Lemma~\ref{lem:maxmin} and \cite{Factns}*{3.3.3 and 3.3.4} imply that there are so such factorisations. Thus 
 for actions on nondegenerate $k$-subspaces we have $k<n/2$.   Suppose that $(n,q)\neq (3,2)$. Then by \cite{King1981} and \cite{King1981a}, $\SU_n(q)$ acts primitively on $\cU$, and so $G_U$ is maximal and core-free in $G$. 

We denote the image  in $\PGammaU_n(q)$ of a subgroup $X\leqslant G$ by $\overline{X}$. Then we have a factorisation $\overline{G}=\oH\,\oG_U$ of the almost simple group $\overline{G}$. Since $\SU_n(q)$ is quasisimple, Lemma~\ref{lem:derived} implies that $\PSU_n(q)\not\leqslant \oH$. Let $\oB$ be maximal among core-free subgroups of $\overline{G}$ containing $\oH$ and let $G^*=\oB\,\PSU_n(q)$. Then by Lemma~\ref{lem:maxmin}, $\overline{B}$ is maximal in $G^*$ and we have a maximal core-free factorisation $G^*=\oB(G^*)_U$.  Moreover, $\oB=N_{\overline{G}}(\oB\cap L)$ and so letting $B$ be the full preimage of $\oB$ in $G$, the Correspondence Theorem implies that $B=N_G(B\cap \SU_n(q))$.  Since $(n,q)\neq (3,2)$ the possibilities for $\cU$ and $\oB$ are given by~\cite{Factns}. We list the possibilities for $\cU$ in Table~\ref{T:UnitaryCases} along with the excluded case $(n,q)=(3,2)$.

\begin{table}[!ht]
\caption{Unitary group cases for Theorem~\ref{T1}.}\label{T:UnitaryCases}
\begin{tabular}{rlcccccccc}
\toprule
\textup{Cases} && $1$& $2$& $3$& $4$& $5$& $6$&$7$& $8$\\
$\cU$&&$\cP_1\textup{ or }\cN_1$&$\cP_1$&$\cP_1$&$\cP_1$&$\cP_1$&$\cP_1$&$\cN_1$&$\cP_m$\\
$(n,q)$&&$(3,2)$&$(3,3)$&$(3,5)$&$(3,8)$&$(4,3)$&$(9,2)$&$(2m,q)$&$(2m,q)$\\
\bottomrule
\end{tabular}
\end{table}

By~\cite{BG}*{pp. 150--151} for the unitary case
we have $|\cP_m|=\prod_{i=1}^m(q^{2i-1}+1)$ and
\begin{equation}\label{E8}
  |\cN_1|=\frac{q^{n-1}(q^n-(-1)^n)}{q+1},\quad\textup{and}\quad
  |\cP_1|=\frac{(q^{n-1}+(-1)^{n-2})(q^n+(-1)^{n-1})}{q^2-1}.
\end{equation}
The formulas~\eqref{E8} work for $n$ odd and even, and
$|\cN_1|+|\cP_1|=(q^{2n}-1)/(q^2-1)$ holds.

\smallskip
{\sc Case} 1:~$(n,q)=(3,2)$.\quad 
Here $\GammaU_3(2)=3_+^{1+2}\rtimes\GL_2(3)$ is solvable and $\cU=\cP_1$ or $\cN_1$. In this case $|\cP_1|=9$ and $|\cN_1|=12$.
Calculations with
{\sc Magma} show that when $\cU=\cP_1$ we have that  $H_0\lhdeq H$ where $H_0\in\{3_{+}^{1+2},3_{-}^{1+2}\}$, as in Table~\ref{T7}, while when $\cU=\cN_1$ we have $H=3^{1+2}_+\rtimes \C_8$, as in part (a).

\smallskip
{\sc Case} 2:~$(n,q)=(3,3)$.\quad 
Here $\cU=\cP_1$ and it follows from~\cite{Factns}*{p.\,13} that $\overline{H}\leqslant N_{\overline{G}}(\PSL_2(7))$. By~\cite{BHRD}*{Table~8.6}, the preimage of $\PSL_2(7)$ in $\SU_3(3)$ is $\PSL_2(7)$ and so $H\leqslant N_G(\PSL_2(7))$. Calculations with {\sc Magma} show that $\PSL_2(7)\lhdeq H$ as in Table~\ref{T7}.

\smallskip
{\sc Case} 3:~$(n,q)=(3,5)$.\quad 
It follows from~\cite{Factns}*{p.\,13} that $\cU=\cP_1$ and
$\overline{H}\leqslant N_{\overline{G}}(\Alt_7)$. The preimage in $\GU_3(5)$ of $\Alt_7$ is
by~\cite{BHRD}*{Table~8.6} the nonsplit and absolutely irreducible
group $\nonsplit{3}{\Alt_7}$.
Calculations with {\sc Magma} show that $\nonsplit{3}{\Alt_7}\lhdeq H$.
Therefore $\nonsplit{3}{\Alt_7}\lhdeq H\le\nonsplit{3}{\Sym_7}$ agreeing with Table~\ref{T7}.

\smallskip
{\sc Case} 4:~$(n,q)=(3,8)$.\quad 
It follows from~\cite{Factns}*{p.\,13} that $\cU=\cP_1$ and
$\overline{H}\leqslant N_{\overline{G}}(19.3)$, and moreover 
 $ L\ldotp 3^2\le\oG$, where $L=\PSU_3(8)$.
Note that $|\cP_1|=q^3+1=3^3\cdot19$ divides $|\oH|$.  Now
$L\ldotp 3^2=P_1(3\times 19\ldotp9)$ is an exact factorisation
by~\cite{Factns}*{p.\,98}. Since $|\Aut(L):L|=3^2\cdot2$, it follows that
$3\times 19\ldotp9 \lhdeq \overline{H}$.
The preimage of $3\times 19\ldotp9$ in $G$
is isomorphic to
$\C_{513}\rtimes \C_9=\langle a,b\mid a^9=b^{513}=1,b^a=b^4\rangle$ by {\sc Magma},
and so we have the entry in Table~\ref{T7}.

\smallskip
{\sc Case} 5:~$(n,q)=(4,3)$ and $\cU=\cP_1$.\quad 
Here $\overline{H}\leqslant N_{\overline{G}}(\PSL_3(4))$ by~\cite{Factns}*{p.\,13} and a calculation in \textsc{Magma} shows that $\PSL_3(4)\lhdeq \oH$. Thus Lemma~\ref{lem:derived} implies that $\nonsplit{4}{\PSL_3(4)}\lhdeq H$ as in Table~\ref{T7}. 

\smallskip
{\sc Case} 6:~$(n,q)=(9,2)$.\quad 
It follows from~\cite{Factns}*{p.\,13} that $\cU=\cP_1$ and
$\overline{H}\leqslant N_{\overline{G}}(\textrm{J}_3)$. The preimage of $\textrm{J}_3$ in $\GU_9(2)$
 is the nonsplit absolutely irreducible subgroup
$\nonsplit{3}{\textrm{J}_3}$ by~\cite{BHRD}*{Table~8.57}. There are no
smaller cases to consider since $\textrm{J}_3$ has no factorisations
by~\cite{Factns}*{p.\,16}. Hence 
$\nonsplit{3}{\textrm{J}_3}\lhdeq H$ agreeing with Table~\ref{T7}.

\smallskip
{\sc Case} 7:~$\cU=\cN_1$.\quad 
Here  the possibilities for $H$
are given by~\cite{Reg}*{Lemma~4.3}.  However, the groups $M_{22}$ and $\PSU_4(3)$ listed there are only the images in $\PGammaU_n(q)$ and should be $\nonsplit{3}{M_{22}}$ and $\nonsplit{3_1}{\PSU_4(3)}$ as in \cite{BHRD}*{Table 8.27} and listed in Table~\ref{T6}.
 
\smallskip
{\sc Case} 8:~$n=2m$ and \ $\cU=\cP_m$.\quad 
By Lemma~\ref{lem:maxmin} and \cite{Factns}*{p.\,13}, there are three subcases: (a)~$(m,q)=(2,2)$, 
(b)~$(m,q)=(2,3)$, or (c)~$(m,q)\not\in\{(2,2),(2,3)\}$ and $B=N_1$
by~\cite{Factns}*{p.\,11}, where $N_1$ is the stabiliser of a nondegenerate
$1$-subspace $\langle v\rangle$.

\smallskip
{\sc Case} 8(a):~$(m,q)=(2,2)$.\quad 
Here we use {\sc Magma} to construct
the subgroups $H$ of $\GammaU_4(2)$ that act transitively on $\cP_2$ with
$\SU_4(2)\not\leqslant H$. Such subgroups $H$ either contain $\SU_3(2)$ as a normal subgroup, or have a normal subgroup $H_0$ such that
$H_0\in\{3_{+}^{1+2}, 3_{-}^{1+2}, \C_3^3\}$ as in
Table~\ref{T8}. Thus case (c) holds in this case.

\smallskip
{\sc Case} 8(b):~$(m,q)=(2,3)$.\quad 
Here we use {\sc Magma} to determine all
$H\leq\GammaU_4(3)$ that act transitively on $\cP_2$ with
$\SU_4(3)\not\leqslant H$. There are 14 conjugacy classes of such $H$,
and each normalizes and contains a subgroup conjugate to $H_0$ where
$H_0\in\{\PSL_2(7), \nonsplit{4_b}{\PSL_3(4)},\SU_3(3)\}$. The first two
possibilities are in Table~\ref{T8}, while the last is listed in the main statement of Theorem~\ref{T1}(c). Note also that the group $\PSL_2(7)$ given here is reducible and contained in $N_1$.

\smallskip
{\sc Case} 8(c):~$(m,q)\not\in\{(2,2),(2,3)\}$.\quad 
Here $G_U=P_m$ and $H\leqslant B=N_1=G_{\langle v\rangle}$ where $\langle v\rangle$
is a nondegenerate $1$-subspace. Then $G=G_UH$ and $B=B_UH$ where
$B_U=B\cap G_U=G_{\langle v\rangle,U}$. 
Since $B$ fixes the nondegenerate $(2m-1)$-subspace
$\langle v\rangle^{\perp}$, it also preserves the decomposition
$V=\langle v\rangle\oplus\langle v\rangle^{\perp}$. However, 
\[
  \SU(\langle v\rangle)\times\SU(\langle v\rangle^{\perp})\lhdeq B=
   G\cap (\GammaU(\langle v\rangle)\times\GammaU(\langle v\rangle^{\perp}))
\]
by~\cite{KL}*{Lemma~4.1.1(ii)}. Moreover, the projection
$\pi\colon B\rightarrow \GammaU(\langle v\rangle^{\perp})$ satisfies 
$$
\SU(\langle v\rangle^\perp) \lhdeq \pi(B)=\GammaU(\langle v\rangle^{\perp}).
$$ 
Note that $m>1$. If $\SU_{2m-1}(q)\leqslant \pi(H)$ then since $\SU_{2m-1}(q)$ is quasisimple (because $(m,q)\ne(2,2)$), it follows from Lemma~\ref{lem:derived} that $\SU_{2m-1}(q)\lhdeq H$ as in the statement of  Theorem~\ref{T1}(c).   Suppose then that $\SU_{2m-1}(q)\not\leqslant \pi(H)$. 
As $B=B_UH$,
we have a factorisation $\GammaU_{2m-1}(q)=\pi(B)=\pi(B_U)\pi(H)$.  The totally isotropic subspace $W:=U\cap\langle v\rangle^\perp$ is fixed by
$B_U=G_{\langle v\rangle,U}$ and so $\pi(B_U)\leqslant P_{m-1}(\pi(B))$.
Thus factoring out by scalars we obtain a core-free factorisation of an almost simple group with socle $\PSU_{2m-1}(q)$ and with one factor being $P_{m-1}$. 
Hence Lemma~\ref{lem:maxmin} and \cite{Factns}*{p.\,10, p.\,13} imply that $(2m-1,q)$ equals $(3,3), (3,5)$ or $(3,8)$. Thus, since $(m,q)\ne(2,3)$, we have
 $(m,q)=  (2,5)$ or $(2,8)$.  For $(m,q)=(2,5)$, a {\sc Magma} computation shows that there are also transitive subgroups containing $H_0=\nonsplit{3}{\Alt_7}$ as a normal subgroup (and contained in $N_1$), as listed in Table~\ref{T8}. When $(m,q)=(2,8)$ a similar computation
with {\sc Magma} shows that there are two possibilities for $\oH$, namely $\oH=\C_{513}\rtimes \C_9$ or $(\C_{513}\rtimes \C_9)\rtimes \C_2$, where in both cases $\oH\cap \PSU_4(8)=\C_{513}\rtimes \C_3$.
Thus $\C_{19}\lhd H$. This covers
all the cases in Table~\ref{T8}.
\end{proof}

\subsection{Six dimensional orthogonal groups of minus type}\label{SS:O6-} 
As we observed in Subsection~\ref{S:Osmall}, $\POmega_6^-(q)\cong \PSU_4(q)$, and it makes sense to treat these orthogonal groups here, now that we have completed the analysis for unitary groups. Recall Remark \ref{rem:Reg}(d) which points out a small discrepancy between one of the groups in Theorem \ref{T:KleinSU4}(c)  and the statement of \cite{Reg}*{Lemma 4.4(iii)}.

 \begin{theorem}\label{T:KleinSU4}
  Suppose that $H\le\GammaO_6^{-}(q)$ acts transitively on a set $\cU$ of subspaces of the natural module $V=(\F_q)^6$ given by~Table~\textup{\ref{T-defn}}, and where $\Omega_6^{-}(q)\not\le H$.  Then one of the following holds, where $\oH$ is the image of $H$ in $\PGammaU_4(q)$:
  \begin{enumerate}[{\rm (a)}]
  \item $\cU=\cP_1$ and $\SU_3(q)\lhdeq H$, or $H_0\lhdeq H$ where $(H_0,m,q)$ are listed in Table~\textup{\ref{T:GOminus6q}}; or
  
  \begin{table}[!ht]
\caption{Theorem~\ref{T:KleinSU4}(a) subgroups $H_0\lhdeq H\le\GammaO^-_{6}(q)$ where $H$
    is transitive on $\cP_1$.}\label{T:GOminus6q}
\begin{threeparttable}
\begin{tabular}{rlccccccc}
\toprule
$H_0$ && $3_{+}^{1+2}$&$3_{-}^{1+2}$&$(\C_3)^3$& $\PSL_2(7)$&$\nonsplit{2}{\PSL_3(4)}$&$\nonsplit{3}{\Alt_7}$&$\C_{19}{}{}^\dagger$\\
\textup{$q$}&&$2$&$2$&$2$&$3$&$3$&$5$&$8$\\
\bottomrule
\end{tabular}
  \begin{tablenotes}\footnotesize
  \item [$\dagger$] More information about $H$ is given in the proof.
  \end{tablenotes}
    \end{threeparttable}
\end{table}

  \item $\cU=\cP_2$ and $q=3$ and $\nonsplit{2}{\PSL_3(4)}\lhdeq H$; or
  \item $\cU=\cN_1^\eps$ (with $q$ odd and $\varepsilon\in\{+,-\}$) or $\cN_1$ (with $q$ even), and either $\SU_3(q)\lhdeq H$, or $q=3$ and $\nonsplit{2}{\PSL_3(4)}\lhdeq H$, or $q=2$ and $3_+^{1+2}\lhdeq H$; or
  \item $\cU=\cN_2^{+}$,  and either $q=4$ and $\SU_3(4)\lhdeq H$, or $q=2$ and $3^{1+2}_+\lhd H$.
\end{enumerate}
\end{theorem}

\begin{proof}
  Note that $\cU=\cP_1,\cP_2$, $\cN_2^\eps$ (with $\eps=\pm$), or $\cN_1$ (with $q$  even), or $\cN_1^\eps$  or $\cN_3^\eps$  (with $\eps=\pm$ and $q$  odd) (see Table~\textup{\ref{T-defn}}).  Fix $U\in\cU$ and set $G=\textup{Stab}_{\GammaO_6^{-}(q)}(\cU)$.    Then $G=\GammaO_6^-(q)$ unless $q$ is odd and $\cU=\cN_1^\eps$ or $\cN_3^\eps$, in which case $G$ is an index two subgroup of $\GammaO_6^-(q)$.

 By~\cite{Taylor}*{Corollary 12.35}, $\Omega_6^-(q)\cong\SU_4(q)/\langle -I_4\rangle$. Then by comparing orders and the fact that $\Aut(\PSU_4(q))=\PGammaU_4(q)$ we deduce that 
  $\PGammaO_6^-(q)\cong\PGammaU_4(q)$.  For any $X\leqslant \GammaO_6^-(q)$ we denote the image of $X$ in $\PGammaU_4(q)$ by $\overline{X}$. Denote by $W$ the natural module $(\F_{q^2})^4$ for $\SU_4(q)$.

\smallskip
  {\sc Case}~$\cU=\cP_1$.  Comparing composition factors we see that if $X$ is the stabiliser in $\GammaO_6^-(q)$ of a totally singular 1-subspace of $V$ then $\overline{X}$ is the stabiliser of a totally isotropic 2-subspace of $W$. Thus $\oH$ acts transitively on the set of totally isotropic 2-subspaces of $W$.  Thus using Theorem~\ref{T1}(c) and the fact that $\Omega_6^-(q)\cong\SU_4(q)/\langle \pm I_4\rangle$, implies that either $\SU_3(q)\lhdeq H$ or $H_0\lhdeq H$ where $H_0$ is given in Table~\ref{T:GOminus6q}, so part (a) holds. Note from the proof of Theorem~\ref{T1} that, in the case where
  $H_0\cong \C_{19}$, the group $\overline{H}$ satisfies $\overline{H}\cap \POmega_6^-(q)=\C_{513}\rtimes \C_3$. 
  \smallskip

  {\sc Case}~$\cU=\cP_2$. Again, by comparing composition factors we see that if $X$ is the stabiliser of an element of $\cP_2$ then $\overline{X}$ is the stabiliser of a totally isotropic 1-subspace of $W$. Thus $\oH$ is transitive on the set of totally isotropic 1-subspaces of $W$ and so, by Theorem~\ref{T1}(b) with $n=4$, we have that $q=3$ and $\PSL_3(4)\lhdeq \oH$. Then \cite{BHRD}*{Table~8.34} implies that $\nonsplit{2}{\PSL_3(4)}\lhdeq H$, and part (b) holds.
 
 \smallskip
  {\sc Case}~$\cU=\cN_1^\eps$ (with $q$ odd) or $\cN_1$ (with $q$ even). Note that when $q$ is odd, a similarity in $\CO_6^-(q)$ that is not an isometry, interchanges $\cN_1^+$ and $\cN_1^-$ and so we may assume that $\cU=\cN_1^+$ in this case. Since $G_U$ has $\POmega_5(q)\cong \PSp_4(q)$ as a composition factor, we see from \cite{BHRD}*{Table~8.10} that $\overline{G_U}$ is a $\mathcal{C}_5$-subgroup normalising $\PSp_4(q)$, and is maximal in $\oG$. Let $\oB$ be a subgroup of $\PGammaU_4(q))$ that contains $\oH$ and is maximal subject to being core-free. Then by Lemma~\ref{lem:maxmin} and  \cite{Factns} we see that either $\oB$ is the stabiliser a nondegenerate 1-subspace $W_1$ of $W$, or $q=3$ and $\oB=N_{\oG}(\PSL_3(4))$.  Suppose first that $\oB$ is the stabiliser in $\oG$ of $W_1$ and let $W_3=W_1^\perp.$ Then considering $\oB^{W_3}$ either $q=2$ or we obtain a factorisation of an almost simple group with socle $\PSU_3(q)$. Thus by \cite{Factns}, either $\PSU_3(q)\lhdeq \oH$, or $q=2,3,5$ or 8. In the first case, Lemma~\ref{lem:derived} implies that $\SU_3(q)\lhdeq H$. We deal with these small values of $q$, and also the case where $q=3$ and $\oB=N_{\oG}(\PSL_3(4))$, with \textsc{Magma}. When $q=2$, there are five possibilities for $H$ and all contain $3_+^{1+2}$ as a normal subgroup.  When $q=3$ we see that either $\SU_3(3)$ or $\nonsplit{2}{\PSL_3(4)}$ is normal in $H$. For $q=5$, we see that $\SU_3(q)\lhdeq H$. Finally, for $q=8$ note that $|\cN_1|=32832$  and \cite{Factns}*{Table~3} implies that either $\SU_3(8)\lhdeq H$ or $\oH\leqslant N_{\oG}(\C_{19})$. However, $|N_{\oG}(\C_{19})|=9234$, which eliminates the latter possibility. Thus part (c) holds.

\smallskip
  {\sc Case}~$\cU=\cN_2^{+}$. Here $G_U$ has $\POmega_4^-(q)\cong\PSL_2(q^2)$ as a unique insoluble composition factor but $\overline{G_U}$ is not a parabolic subgroup of $\oG$. Thus comparing \cite{BHRD}*{Tables~8.10 and~8.33} we see that $\overline{G_U}$ is a $\mathcal{C}_2$-subgroup of $\oG\cong\PGammaU_4(q)$ preserving a decomposition of $W$ into a pair of complementary totally isotropic 2-subspaces. Moreover, $\overline{G_U}$ is maximal for all $q>2$.  If $q>2$ we let $\oB$ be a subgroup of $\oG$ that contains $\oH$ and is maximal subject to being core-free. Then Lemma~\ref{lem:maxmin} and \cite{Factns} imply that $q=4$ and $\oB$ is the stabiliser in $\oG$ of a nonsingular 1-subspace of $W$. A \textsc{Magma} calculation and application of Lemma~\ref{lem:derived} then reveals that $\SU_3(4)\lhdeq H$. For $q=2$, a \textsc{Magma} calculation shows that $3_+^{1+2}\lhd H$.
   This verifies part~(d)

\smallskip
  {\sc Case}~$\cU=\cN_2^{-}$. Here $G_U$ has $\POmega_4^+(q)\cong\PSL_2(q)^2$ as a section. Thus by comparing \cite{BHRD}*{Tables~8.10 and~8.33} we see that $\overline{G_U}$ is a $\mathcal{C}_2$-subgroup of $\oG\cong\PGammaU_4(q)$ preserving a decomposition of $W$ into a pair of orthogonal nondegenerate 2-subspaces, and is maximal for $q\geq 3$. Thus by Lemma~\ref{lem:maxmin}, when $q\geq 3$  we get a core-free maximal core-free factorisation of $\oG$. However, by \cite{Factns} no such factorisation exists. Thus $q=2$. A \textsc{Magma} calculation then shows that there are no examples with $\Omega_6^-(q)\not\leqslant H$.

\smallskip
  {\sc Case}~$\cU=\cN_3^{\eps}$. Here $q$ is odd and the two isometry classes are interchanged by a similarity.  Hence $G$ is an index two subgroup of $\GammaO_6^-(q)$ and 
 $G_U=G_{U^\perp}$ preserves an orthogonal decomposition $V=U\oplus U^\perp$.   Thus we may assume that $\cU=\cN_3^+$ in this case. Also $G_U$ has $\PSL_2(q)^2$ as a section. Comparing \cite{BHRD}*{Tables~8.10 and~8.33} we see that $\overline{G_U}$ is a $\mathcal{C}_5$-subgroup of $\oG\cong\PGammaU_4(q)$ of type $\SO_4^+(q)$, and is maximal in $G$ for $q\geq 5$. By~\cite{Factns}, there are no such maximal core-free factorisations and so $q=3$. In this case, a \textsc{Magma} calculation implies that $\Omega_6^-(3)\lhdeq H$.
\end{proof}

\section{Symplectic groups}\label{S:Sp}

Let $V=(\mathbb{F}_q)^n$ be equipped with a nondegenerate alternating form and let $\cU$ be the set of all totally isotropic subspaces of $V$ of dimension $k\leq n/2$, or the set of all nondegenerate subspaces of dimension $k\leq n/2$. Note that $n$ is always even, and also that  every 1-subspace of $V$ is totally isotropic, so the case $n=2$ was considered in Remark~\ref{R:Sp2}. We assume henceforth that $n\ge4$. By Witt's Lemma, $\Sp_n(q)$ is transitive on $\cU$. In Theorem~\ref{T3} we classify the subgroups $H$ of $\GammaSp_n(q)$ that are transitive on~$\cU$. We use this result  in Subsection~\ref{SS:O5} to deal with the five-dimensional orthogonal groups in odd characteristic.
We note that $\PSp_n(q)$ is simple unless $(n,q)=(4,2)$, in which case $\PSp_4(2)\cong \Sym_6$. 

Part (c) of Theorem~\ref{T3} was considered in \cite{Reg}*{Lemma~4.2} while parts (a) and (b) for $m=2$ were considered in \cite{KLrank3}*{Corollary~5.12}. We begin with the following lemma, which enables us to eliminate some possibilities given by \cite{Reg}*{Lemma~4.2(i)}.

\begin{lemma}\label{L:temp}
Let  $G=\Sp_{4m}(2)$ with natural module $V=\F_2^{4m}$, where  $m$ is even, and either let $L=\Sp_m(16).4$ preserving an extension field structure on $V$, or let $L= \G_2(16).4$ if $m=6$. Then $L$ is not transitive on the set $\cN_2$ of nondegenerate $2$-subspaces of $V$.  In particular the possibilities \emph{`$\Sp_{m/2}(q^4)$ ($m/2$ even, $q=2$)'} and \emph{`$\G_{2}(q^4)$ ($m=12$, $q=2$)'} in \cite{Reg}*{Lemma 4.2(ii)} lead to no examples.
\end{lemma}
 
 \begin{proof}
 Since in the case $m=6$ the subgroup $\G_2(16).4$ is contained in an extension field subgroup 
 $\Sp_{6}(16).4$, it is sufficient to prove the lemma in the case  $L=\Sp_m(16).4$. 
 If $m=2$ then a {\sc Magma} computation shows that no subgroup $\Sp_2(16).4$ of $\Sp_8(2)$ acts transitively on $\cN_2$. Assume from now on that $m\geq 4$. Let $M=\Sp_{2m}(4).2$ be the maximal extension field subgroup of $G$ containing $L=\Sp_m(16).4$.  Using the notation from \cite{Factns}*{p.\,47}, let $E_1,\dots, E_m, F_1,\dots, F_m$ be a standard basis for $V=\F_4^{2m}$ (regarded as the natural module for $M$) relative to the alternating form preserved by $M$,  let $Y=\la E_1,F_1\ra_{\F_4}$, a nondegenerate $2$-subspace of $\F_4^{2m}$, and let $\lambda\in\F_4\setminus\F_2$ with trace 1 relative to $\F_2$ so that $X=\la E_1,\lambda F_1\ra_{\F_2}\in\cN_2$, a nondegenerate $2$-subspace of $V$ contained in $Y$. Note that   $G_X=\Sp_{4m-2}(2)\times \Sp_2(2)$ and   $M_Y=(\Sp_{2m-2}(4)\times \Sp_2(4)).2$, and also that $M_X$ leaves $Y$ invariant so $M_X\leq M_Y$.

It is proved in \cite{Factns}*{pp.\,47--48}  that $G=MG_X$ and $M=LM_Y$ are maximal core-free factorisations,  and that (importantly)  
 \[
 M_X=(M')_X = \Sp_{2m-2}(4)\times\Sp_2(2)\quad\mbox{and}\quad L_Y=(L')_Y = \Sp_{m-2}(16)\times\Sp_2(4),
 \]
 where $M', L'$ denote the derived subgroups.  Now $L_X\leq M_X$ (since $L\leq M$) which is contained in $M_Y$ (as noted above). Thus $L_X\leq M_Y\cap L=L_Y$, and hence $L_X\leq M_X\cap L_Y=\Sp_{m-2}(16)\times\Sp_2(2)$. On the other hand, 
 the argument on  \cite{Factns}*{p.\,47} shows that $L_X$ contains $\Sp_{m-2}(16)\times\Sp_2(2)$, so equality holds.
Thus $M_X$ is a subgroup of $M_Y$ of index $|M_Y:M_X|=2 |\Sp_2(4):\Sp_2(2)|$, while 
$L_X$ is a subgroup of $L_Y$ of index $|L_Y:L_X|= |\Sp_2(4):\Sp_2(2)|=|M_Y:M_X|/2$.  Since $G=MG_X$ and $M=LM_Y$ we have 
\[
|G:G_X|=|M:M_X|=|M:M_Y|.|M_Y:M_X| = |L:L_Y|.\left( 2 |L_Y:L_X|\right) = 2|L:L_X|,
\]
and hence $L$ has two equal length orbits on $\cN_2$.  
 \end{proof}

\begin{theorem}\label{T3}
Suppose  that $H\leqslant\GammaSp_{2m}(q)$, where $m\geq 2$, such that $\Sp_{2m}(q)'\not\leqslant H$. Let $V=(\mathbb{F}_q)^{2m}$
and suppose that $H$ acts transitively on  the set $\cU=\cP_k$ or $\cN_k$ of $k$-subspaces of $V$, where $1\le k\le m$. Then one of
 \begin{enumerate}[{\rm (a)}]
  \item $\cU=\cP_1$ and $H_0\lhdeq H$ where $(H_0,m,q)$ is given in
    Table~\textup{\ref{T:Spa}}; or
\begin{table}[!ht]
\caption{Theorem~\ref{T3}(a) Subgroups $H_0\lhdeq H\le\GammaSp_{2m}(q)$ where $H$ is transitive
on $\cP_1$.}\label{T:Spa}
\begin{tabular}{rlcccccc}
\toprule
$H_0$ &&$\Sp_{2a}(q^b)$ &$\textup{G}_2(q^b)'$&$2_{-}^{1+4}\ldotp\C_5$&$2_{-}^{1+4}\ldotp\Alt_5$ &$\SL_2(5)$&$\SL_2(13)$\\
\textup{$m$}&&$ab>a$&$3b$&$2$&$2$&$2$&$3$\\
\textup{$q$}&&${\rm all}$&${\rm all}$&$3$&$3$&$3$&$3$\\
\bottomrule
\end{tabular}
\end{table}

\item $\cU=\cP_m$ and $\Omega_{2m}^{-}(q)\lhdeq H$ where $m\geq2$ and
  $q$ is even, or $(m,q)=(2,3)$ and either $2_-^{1+4}.\Alt_5\lhdeq H$, or $H=2_-^{1+4}.F_{20}$; or
\item $\cU=\cN_2$ and $H_0\lhdeq H$ where $H_0$
  is given in~Table~\textup{\ref{T:Spc}}.
\begin{table}[!ht]
\caption{Theorem~\ref{T3}(c) Subgroups $H_0\lhdeq H\le\GammaSp_{2m}(q)$ where $H$ is transitive
on $\cN_2$.}\label{T:Spc}
\begin{tabular}{rlccccc}
\toprule
$H_0$ &&$\Sp_{2a}(q^2)$ &$\G_2(q)'$&$\G_2(q^2)$&$\SL_2(q^2)$\\
\textup{$m$}&&$2a$&$3$&$6$&$2$\\
\textup{$q$}&&$2,4$&\textup{even}&$2,4$&\textup{even}\\
\bottomrule
\end{tabular}
\end{table}

\end{enumerate}
\end{theorem}

\begin{proof}
Let $G=\GammaSp_{2m}(q)$, so $H\leqslant G$ is transitive on $\cU$ with $\Sp_{2m}(q)\not\leqslant H$.  Let $U\in \cU$. Then we have a factorisation $G=HG_U$.
By~\cite{King1981} and \cite{King1981a}, either $\Sp_{2m}(q)$ acts primitively on $\cU$ and so $G_U$ is maximal in $G$, or 
$\cU=\cN_m$. 

Given any $X\leqslant G$ we denote the image of $X$ in $\PGammaSp_{2m}(q)$ by $\overline{X}$. Then we have a factorisation $\overline{G}=\overline{H}\,\overline{G}_U$ of the almost simple group $\overline{G}=\PGammaSp_{2m}(q)$. When $(m,q)\neq (2,2)$, the group $\Sp_{2m}(q)$ is quasisimple and so Lemma~\ref{lem:derived} implies that $\PSp_{2m}(q)\not\leqslant \overline{H}$. When $(m,q)=(2,2)$ there are no non-identity scalar matrices and so we again have $\PSp_{2m}(q)\not\leqslant \oH$. Let $L=\PSp_{2m}(q)$, let $\overline{B}$ be maximal among the core-free subgroups of $\PGammaSp_{2m}(q)$ containing $\overline{H}$, and let $G^*=\overline{B}\,L$. Suppose first that $\cU\neq \cN_m$. Then by Lemma~\ref{lem:maxmin}, $\overline{B}$ is maximal in $G^*$ and we have $G^*=\overline{B}(G^*)_U$, which is a maximal core-free factorisation. Thus  $\cU$ and $\overline{B}$ are given by~\cite{Factns}. Moreover, $\overline{B}=N_{\PSp_{2m}(q)}(\overline{B}\cap L)$ and so letting $B$ be the full preimage of $\oB$ in $\GammaSp_n(q)$, the Correspondence Theorem implies that  $B=N_G(B\cap \Sp_{2m}(q))$. 
By  \cite{Reg}*{Lemma~4.1(iii)} we have that $\cU=\cP_1$, $\cP_m$ or $\cN_2$. 
On the other hand, if $\cU=\cN_m$ with $m\geq 3$ then $m$ is even, $(G^*)_U$ has index~$2$ in the stabiliser $(G^*)_{\{U,U^\perp\}}$ of the decomposition $V=U\oplus U^\perp$. Then Lemma~\ref{lem:maxmin} implies that  $G^*=\overline{B} (G^*)_{\{U,U^\perp\}}$ is a maximal core-free factorisation. Moreover since $\overline{B}$ is transitive on $\cN_m$, $\overline{B}_U$ must have index $2$ in $\overline{B}_{\{U,U^\perp\}}$. It follows from  \cite{Factns}*{pp.\,10, 12, 13} that $\overline{B}=N_{G^*}(\SO_{2m}^-(q))$ with $q$ even. However, by  \cite{Factns}*{(3.2.4b)}, one of $U, U^\perp$ has $+$ type while the other has $-$ type relative to the quadratic form for $\overline{B}$, and it follows that $\overline{B}_{\{U,U^\perp\}}=\overline{B}_U$, which is a contradiction. Thus the cases to be considered are $\cU=\cP_1$, $\cU=\cP_m$, and $\cU=\cN_2$.

\smallskip
{\sc Case}~$\cU=\cN_2$.  This case was dealt with 
in~\cite{Reg}*{Lemma~4.2}, even when $m=2$. The possibilities $\Sp_{m/2}(16)\lhdeq H$ with $q=2$, so that $H\leqslant \Sp_{m/2}(16).4$, and also $G_2(16)\lhdeq H$ with $(m,q)=(12,2)$, are listed in~\cite{Reg}*{Lemma~4.2(i)}. However these groups are eliminated by Lemma~\ref{L:temp}. The remaining groups are listed in Table~\ref{T:Spc}.

\smallskip
{\sc Case}~$\cU=\cP_1$. 
Every 1-subspace of the symplectic space $V$ is isotropic, so
here $\cU$ comprises all 1-subspaces and $|\cU|=(q^n-1)/(q-1)$. 
The subgroups of $\GammaL_n(q)$ acting transitively on 1-subspaces are
listed in Theorem~\ref{T:SLbigN}(a). The subgroups $\SL_{2a}(q^b)$ ($m=ab$)
preserve an extension field structure, and so are not contained in $\Sp_{2m}(q)$ (unless $a=1$ in which case they are the groups $\Sp_{2}(q^m)$ as in Table~\ref{T:Spa}).
The subgroups $\Sp_{2a}(q^b)$ ($m=ab$), and $\G_2(q^b)'$ ($m=3b$)
do lie in $\Sp_{2m}(q)$ and are listed in Table~\ref{T:Spa}. We argue that  $\GammaL_1(q^n)\cap\GammaSp_n(q)$ 
is usually too small to act transitively on $\P_1$. It follows
from~\cite{Ber}*{Table~1} that $\GL_1(q^{2m})\cap\Sp_{2m}(q)$ is cyclic
of order $q^m+1$ and hence $\GL_1(q^{2m})\cap\CSp_{2m}(q)$ has order dividing $(q-1)(q^m+1)$.
We conclude that
$|\GammaL_1(q^{2m})\cap\GammaSp_{2m}(q)|$ divides $2mf(q-1)(q^m+1)$. As 
$|\P_1|=(q^{2m}-1)/(q-1)$ divides $|H|$, it also divides $2mf(q-1)(q^m+1)$
and hence $(q^m-1)\mid 2mf(q-1)^2$. By Lemma~\ref{L:NT}, the only possibilities are
$m=2$ and $q\in\{3,7\}$. A calculation with {\sc Magma} shows that $\GammaSp_4(7)$ has no solvable
transitive subgroups and $\GammaSp_4(3)$ has
three, but none are metacyclic. 
This eliminates the case
$H\le\GammaL_1(q^{2m})$ in~Theorem~\ref{T:SLbigN}(a). 

It remains to consider the small
dimensional cases with $m\ge2$ in the second row of Table~\ref{T:SLa}.
Neither of the possibilities $\Alt_7$ and $\Alt_6$ for $(m,q)=(2,2)$ arise since 
$H$ does not contain $\Sp_4(2)'\cong\Alt_6$. Next suppose $(m,q)=(2,3)$. 
Then the first two possibilities listed  satisfy
$2_{-}^{1+4}\lhdeq H\le 2_{-}^{1+4}\ldotp\Sym_5$ and both arise, and are listed in Table~\ref{T:Spa}, 
as $2_{-}^{1+4}\ldotp\Sym_5\leqslant \CSp_4(3)$ (see \cite{BHRD}*{Table~8.12, p.\,383}).
Similarly $\SL_2(5)\le\Sp_4(3)$, so we list
$H_0=\SL_2(5)$ in Table~\ref{T:Spa}. 
Finally for $(m,q)=(3,3)$,~\cite{BHRD}*{Table~8.29,\, p.\,392}
or a \textsc{Magma} computation shows that $\nonsplit{2}{\PSL_2(13)}=\SL_2(13)$
is a subgroup of $\Sp_6(3)$.
This justifies all the entries of Table~\ref{T:Spa}.

\smallskip
{\sc Case}~$\cU=\cP_m$. 
Suppose now that $H$ is transitive on the set $\cU=\mathcal{P}_m$, so $|\cU|=\prod_{i=1}^m(q^i+1)$ 
by~\cite{BG}*{pp.\;150--151}. Then by \cite{Factns} there are two possibilities for the subgroup $B$:
\begin{enumerate}
  \item[(b1)] $B=N_G(\SO_{2m}^-(q))$ with $q$ even by \cite{Factns}*{p.\,10}; or
  \item[(b2)] $\oB\cap L=2^4.\Alt_5$ with $(m,q)=(2,3)$ by \cite{Factns}*{p.\,13}.
\end{enumerate}

(b1)~ Suppose that $B=N_G(\SO_{2m}^-(q))=\GammaO^-_{2m}(q)$ with $q$ even. 
By \cite{Factns}*{p.\,49},  $\SO_{2m}^-(q)$
is transitive on $\cP_m$, and since $|\cP_m|$ is odd and
$|\SO_{2m}^-(q):\Omega_{2m}^-(q)|=2$, it follows that
$\Omega_{2m}^-(q)$ is also transitive on $\cP_m$.
Therefore any $H$ satisfying $\Omega_{2m}^-(q)\lhdeq H$ is necessarily
transitive on $\cP_m$ as in (b). Henceforth suppose that $\Omega_{2m}^-(q)\not\le H$.  Then we have a factorisation $B=HB_U$.
Since $\Omega_{2m}^-(q)$ is irreducible on $V$, we get a core-free factorisation of the almost simple group $\PGammaO_{2m}^-(q)$.  

 By~\cite{Factns}*{3.2.4(a), p.\,49}, $B_U\leqslant B_Y$, where $Y$ is an $(m-1)$-dimensional subspace of $U$ which is totally singular with respect to the quadratic form preserved by $\SO^-_{2m}(q)$.  Thus we have a core-free factorisation of $\PGammaO_{2m}^-(q)$, where one of the factors is the stabiliser of a totally singular $(m-1)$-subspace.
No such factorisation exists for $m\ge4$ by \cite{Factns}*{p.\,11--13}, and if $m=3$ then none exists with $q$ even by Theorem~\ref{T:KleinSU4}. Thus $m=2$.  Since $q$ is even $\PSp_4(q)$ has an outer automorphism $\tau$ that interchanges $P_1$ and $P_2$. Hence $(\oH)^\tau$ acts transitively on $\cP_1$. Thus by part (a) we have $\PSp_2(q^2)\lhdeq (\oH)^\tau$ and so by  \cite{Asch}*{Section 14} we have that $\Omega_4^-(q)\lhdeq H$.

(b2)~Suppose $(m,q)=(2,3)$ and  $\oB\cap L=2^4.\Alt_5$, with $G$ acting on $\cP_2$.
By~\cite{BHRD}*{Table~8.12} we have $B\cap \Sp_4(3)=2_-^{1+4}.\Alt_5$.  A computation with {\sc Magma} shows that a transitive subgroup of $\CSp_4(3)$ either  contains $2_-^{1+4}.\Alt_5$ or $\Sp_4(3)$ as a normal subgroup, or is equal to $2_-^{1+4}.F_{20}$, as given in part~(b) of the statement.
\end{proof}

\subsection{Five dimensional orthogonal groups with $q$ odd}\label{SS:O5} 

As we observed in Subsection~\ref{S:Osmall}, $\POmega_5(q)\cong \PSp_4(q)$. We treat these orthogonal groups here if $q$ is odd, using our analysis of the symplectic groups.
If $q$ is even,  the groups $\POmega_5(q)$ are treated in Section~\ref{S:OoddEvenq} where we deal uniformly with the family $\POmega_n(q)$ for  even $q$ and odd $n\geq 5$.

\begin{theorem}\label{T:KleinSp4}
  Suppose that $H\leqslant \GammaO_5(q)$ acts transitively on a set $\cU$ of subspaces of the natural module $V=(\F_q)^5$ given by Table~$\ref{T-defn}$, where $q$ is odd and $\Omega_5(q)\not\le H$. Then
  \begin{enumerate}[{\rm (a)}]
  \item $\cU=\cP_1$ and $q=3$ where either $2^4.\Alt_5\lhdeq H$, or $H=2^4.F_{20}$; or
  \item $\cU=\cP_2$ and either $\Omega_4^-(q)\lhdeq H$, or $q=3$ and either
    $2^4\ldotp\C_5$, $2^4\ldotp\Alt_5$ or $\Alt_5\lhdeq H$; or
  \item $\cU=\cN_1^-$ and $H\leqslant P_2$ and $H$ induces a subgroup of $\GammaL_2(q)$ that is transitive on $1$-subspaces.

\end{enumerate}
\end{theorem}

\begin{proof}
By Table~\ref{T-defn}, $\cU$ is one of $\cP_1$, $\cP_2$, $\cN_1^\pm$ or $\cN_2^\pm$. 
Fix $U\in\cU$ and set
$G=\textup{Stab}_{\GammaO_5(q)}(\cU)$. By~\cite{Taylor}*{12.31~Theorem}, $\Omega_5(q)\cong\PSp_4(q)$. Also $\Aut(\PSp_4(q))=\PGammaSp_4(q)$, and comparing orders we deduce that  $\PGammaO_5(q)\cong\PGammaSp_4(q)$. For any $X\leqslant \GammaO_5(q)$ we denote the image of $X$ in $\PGammaSp_4(q)$ by $\overline{X}$. We denote by $W$ the 4-dimensional vector space over $\F_{q}$ acted on by $\Sp_4(q)$.

\smallskip
{\sc Case}~$\cU=\cP_1$. By~\cite{Taylor}*{p.\,196}, there is a one-to-one correspondence between $\cP_1$ and the set of totally isotropic 2-dimensional subspaces of $W$. Thus by Theorem~\ref{T3}(b),  and the facts that $q$ is odd and $\Omega_5(3)\cong\PSp_4(3)$, we have  $q=3$ and either $2^4.\Alt_5\lhdeq H$ or $2^4.F_{20}= H$, as in part (a). 

\smallskip
{\sc Case}~$\cU=\cP_2$. By~\cite{Taylor}*{p.\,196}, there is a one-to-one correspondence between $\cP_2$ and the set of totally isotropic 1-dimensional subspaces of $W$. Then by Theorem~\ref{T3}(a) and again using the fact that $\Omega_5(q)\cong\PSp_4(q)$, we see that either $\PSp_2(q^2)\lhdeq H$, or  $q=3$ and
one of $2^4\ldotp\C_5\lhdeq H$, $2^4\ldotp\Alt_5\lhdeq H$ or $\Alt_5\lhdeq H$. Now $\PSp_2(q^2)\cong\Omega_4^-(q)$ and comparing \cite{BHRD}*{Table~8.22} we see that in the first case $\Omega_4^-(q)\lhdeq H$. Thus part (b) holds.

\smallskip
  {\sc Case}~$\cU=\cN_1^+$.  Noting that $\POmega_4^+(q)\cong\PSL_2(q)^2$ and comparing composition factors we see from Table~\cite{BHRD}*{Table~8.12} that $\overline{G_U}$ is a $\mathcal{C}_2$-subgroup
  of $\PGammaSp_4(q)$, and is maximal for all $q$. Thus by Lemma~\ref{lem:maxmin} we get a  maximal core-free factorisation of a group with socle $\PSp_4(q)$, and by~\cite{Factns} there are no such factorisations.
  
  \smallskip
   {\sc Case}~$\cU=\cN_1^-$. Since $\POmega_4^-(q)\cong\PSL_2(q^2)$ and comparing composition factors we see from Table~\cite{BHRD}*{Table~8.12} that $\overline{G_U}$ is a $\mathcal{C}_3$-subgroup of $\oG$ and is maximal for all $q$. It follows from Lemma~\ref{lem:maxmin} and \cite{Factns} that $\oH$ stabilises a totally isotropic 1-subspace of $W$ and so $H$ stabilises a totally singular 2-subspace $X$ of $V$, that is $H\leqslant P_2$.  Let $B=G_X$ so that $H\leqslant B$. Then arguing as in \cite{Factns}*{p.\.57} we see that $B_U$ fixes a 1-dimensional subspace of $X$. Since we have a factorisation $B=B_UH$ it follows that $H$ acts transitively on the set of 1-dimensional subspaces of $X$, and part (c) holds.
   
   \smallskip
   {\sc Case}~$\cU=\cN_2^\eps$. Again, comparing composition factors and consulting \cite{BHRD}*{Table~8.12} we see that $\overline{G_U}$ is a $\mathcal{C}_2$-subgroup stabilising a decomposition into two totally isotropic 2-subspaces or a $\mathcal{C}_3$-subgroup of type $\GU_2(q)$. Moreover, they are only maximal for $q\geq 5$. By Lemma~\ref{lem:maxmin} and \cite{Factns} there are no such maximal factorisations and so $q=3$. However, in this case a \textsc{Magma} calculation finds no examples.
   \end{proof}

\section{Orthogonal groups in even dimension of minus type}\label{S:OEven}

Let $V$ be a $2m$-dimensional vector space over $\mathbb{F}_q$ equipped with a nondegenerate quadratic form of minus type.  By~\cite{BG}*{Lemma~2.5.10}, $\Omega_{2m}^-(q)$ acts transitively on each set $\cU$ given in Table~\ref{T-defn}, and we identify transitive subgroups in these actions in Theorem~\ref{T9}. If $kq$ is odd then there are two isometry types of nondegenerate $k$-subspaces but only one similarity class. Thus the two orbits $\cN_k^\pm$ of $\Omega_{2m}^-(q)$ are fused in $\CO_{2m}^-(q)$, and we find the subgroups transitive on $\cN_k=\cN_k^+\cup\cN_k^-$ in Proposition~\ref{P1}.

First we need the following lemma.

\begin{lemma}\label{lem:SUorthog}
Let $m\geq 4$. If $m$ is odd let $A=\SU_m(2)\leqslant \Omega^-_{2m}(2)$ and if $m$ is even let $A=\SU_m(2)\leqslant \Omega^+_{2m}(2)$. Then $A$ is transitive on $\cN_2^+$.
\end{lemma}
\begin{proof}
  Following \cite{Factns}*{p.60--61},  let $F=\F_2$, $K=\F_{4}$ and view $V=F^{2m}=K^m$ as an $A$-module with $A=\SU_m(2)$. Let  $[\ ,\,]$ be the $A$-invariant $K$-Hermitian form 
 on $V$ and define $Q:V\rightarrow F$ by $Q(v)=[v,v]$ for all $v\in V$. Then $Q$ is a quadratic form of $+$ type when $m$ is even and of $-$ type when $m$ is odd (see for example \cite{BG}*{Construction~2.5.14}).  Since $A$ preserves $Q$ this gives us an embedding $A\leqslant \Omega_{2m}^+(2)$ when $m$ is even and an embedding $A\leqslant \Omega_{2m}^-(2)$ when $m$ is odd. Let $\cN_2^+$ be the set of all $2$-dimensional subspaces $W$ of $V$ over $F$  such that the restriction of $Q$ to $W$ is nondegenerate of $+$ type.

Choose $a,b\in V$ such that $[a,a]=0=[b,b]$ and $[a,b]=1$, and let $\lambda\in K$ such that $\lambda+\lambda^2=1$. Then if $(\ ,\ )$ is the symmetric bilinear form associated with $Q$ we have that $(\lambda a,b)=Q(\lambda a+b)+Q(\lambda a)+Q(b)=\lambda+\lambda^2=1$. Hence $U:=\langle \lambda a,b\rangle_F\in \cN_2^+$.  Now $A_U$ fixes $\langle U\rangle_K=\langle a,b\rangle_K$, which is a nondegenerate $K$-subspace of dimension 2.  By \cite{KL}*{Lemma~4.1.1}, $A_{\langle U\rangle_K}= (\SU_{m-2}(2)\times \SU_2(2)).(q+1)$ where the $(q+1)$ acts diagonally on each factor so that $A_{\langle U\rangle_K}$ induces $\GU_2(2)$ on $\langle U\rangle_K$ and $\GU_{m-2}(2)$ on $(\langle U\rangle_K)^\perp$. Since $[\lambda a,b]=\lambda$ while $[b,\lambda a]=\lambda^2\neq \lambda$, there is no $g\in\GU_2(2)$ interchanging $\lambda a$ and $b$. As these are the only singular vectors in $U$, it follows that $\GU_2(2)_U=1$ and hence $A_U=\SU_{m-2}(2)$.  Thus $|A:A_U|=2^{2m-3}(2^m+1)(2^{m-1}-1)=|\cN_2^+|$ and so $\SU_m(2)$ is transitive on $\cN_2^+$.
\end{proof}

\begin{theorem}\label{T9}
  Suppose that $H\le\GammaO_{2m}^{-}(q)$, where  $m\geq 2$, such that
  $\Omega_{2m}^-(q)\not\leqslant H$. Let $V=(\F_{q})^{2m}$ and suppose
  that $H$ acts transitively on a  set $\cU$ of subspaces
  of $V$  given by Table~\textup{\ref{T-defn}}.  Then either $m=2$ and $H$
  is given by Theorem~$\ref{T:KleinPSL2q2}$, or $m=3$ and $H$ is given
  by Theorem~$\ref{T:KleinSU4}$, or $m\geq 4$ and  one of 
\begin{enumerate}[{\rm (a)}]
  \item $\cU=\cP_1$ and $H_0\lhdeq H$ where $(H_0,m,q)$ are listed in Table $\ref{T:OEvenP1}$; or
  
\begin{table}[!ht]
\caption{Theorem~\ref{T9}(a) subgroups $H_0\lhdeq H\le\GammaO_{2m}^{-}(q)$ with $m\geq 4$ and $H$ transitive on $\cP_1$.}\label{T:OEvenP1}
\begin{tabular}{rlccccccccc}
\toprule
$H_0$&& $\SU_m(q)$&$\nonsplit{3}{\textup{J}_3}$&$\Alt_{12}$ & $  \textup{M}_{12}$&&\\
$m$&& $\textup{odd}$&$9$&$5$&$5$&\\
$q$&& ${\rm all}$&$2$&$2$&$2$&\\
\bottomrule
\end{tabular}
\end{table} 

\item $\cU=\cN_1^\eps$ (with $q$ odd and $\eps\in\{+,-\}$) or $\cN_1$ (with $q$ even) and either $m$ is odd and  $\SU_m(q)\lhdeq H$, or $q$ is even and $H_0\lhdeq H$ where $(H_0,m,q)$ are listed in Table 
$\ref{T:OEvenN1}$; or

\begin{table}[!ht]
\caption{Theorem~\ref{T9}(a) subgroups $H_0\lhdeq H\le\GammaO_{2m}^{-}(q)$ with $m\geq 4$ and  $H$ transitive on $\cN_1$ with $q$ even.}\label{T:OEvenN1}
\begin{tabular}{rlccccccccc}
\toprule
$H_0$	&&$\SU_{m/2}(q^2)$&$\SU_{m/4}(q^4)$&$\Omega^{-}_m(q^2)$&$\Omega^{-}_{m/2}(q^4)$\\
$m$		&&$4a+2\geq6$&$8a+4\geq12$&$2a$&$4a$\\
$q$    	&&$2,4$&$2$&$2,4$&$2$\\
\bottomrule
\end{tabular}
\end{table}

\item $\cU=\cN_2^+$, $m$ is odd, $q=2$ or $4$, and $\SU_m(q)\lhdeq H$. Moreover,  when $q=4$ the groups $H$ and  $\GammaU_m(4)$ project onto the same 
  subgroup of $\PGammaO_{2m}^-(4)$.

\end{enumerate}
\end{theorem}

\begin{proof}
Let $G$ be the setwise stabiliser in $\GammaO^-_{2m}(q)$ of $\cU$ and let $H\leqslant G$ act transitively on $\cU$ such that $\Omega^-_{2m}(q)\not\leqslant H$. If $m=2$ then the result follows from Theorem~\ref{T:KleinPSL2q2}, while if $m=3$ it follows from Theorem~\ref{T:KleinSU4}. Thus from now on we assume that $m\geq 4$. Note that $G=\GammaO_{2m}^-(q)$ unless $\cU$ is the set of nondegenerate $k$-subspaces of a given isometry type and $kq$ is odd. In the latter case $G$ is an index two subgroup of $\GammaO_{2m}^-(q)$ that contains $\GO_{2m}^-(q)\leqslant G$.  Let $U\in \cU$. Then we have a factorisation $G=HG_U$. By
 \cites{King1981, King1981a, King1982}
(or see  \cite{Factns}*{Theorems~7.0.1 and~8.1.1} and \cite{BHRD}*{Theorem~2.1.1}), $\GO_{2m}^-(q)$ acts primitively on $\cU$ unless $\cU=\cN_2^+$ and $q=2$ or $3$. Further, if $\cU=\cN_2^+$ and $q=3$, then it follows from \cite{BHRD}*{Table~8.52} (for $m=4$), \cite{BHRD}*{Proposition~2.3.2} (for $m=5, 6$), and \cite{KL}*{Main Theorem~(c) and Table~3.5F} (for $m\geq 7$), that    $G=\GammaO_{2m}^-(3)$ is primitive on $\cU$.
We will treat the exceptional case where $\cU=\cN_2^+$ and $q=2$ separately at the end of the proof, and so for now we assume that $G_U$ is maximal in $G$. Then by \cite{Reg}*{Lemma~4.1(v)}, $\cU$ is  $\cP_1$, or $\cN_1$ ($q$ even), or $\cN_1^\pm$ ($q$ odd),
or $\cN_2^+$.  
The possibilities when $\cU=\cN_1$ ($q$ even) or $\cN_1^\pm$ ($q$ odd) are determined in~\cite{Reg}*{Lemma~4.4}. We list these in part (b) (especially Table~\ref{T:OEvenN1}) for completeness. 

So assume now that  $\cU$ is  $\cP_1$, or  $\cU$ is  $\cN_2^+$ ($q>2$). Denote the image of any subgroup $X$ of $\GammaO_{2m}^-(q)$ in $\PGammaO_{2m}^-(q)$ by $\overline{X}$ and let $L=\POmega_{2m}^-(q)$. Then we have a factorisation $\oG=\oH\,\oG_U$. Since $\Omega_{2m}^-(q)$ is quasisimple, $L\not\leqslant \oH$. Let $\oB$ be maximal among core-free subgroups of $\oG$ containing $\oH$ and let $G^*=\oB L$. Then by Lemma~\ref{lem:maxmin}, $\oB$ is maximal in $G^*$ and we have a maximal core-free factorisation $G^*=\oB(G^*)_U$. In particular, for each $\cU$, the possibilities for $\oB$ 
are given by \cite{Factns} and are collated in Table~\ref{T:OEvenCases}. Moreover, $\oB=N_{\oG}(\oB\cap L)$ and so letting $B$ be the full preimage of $\oB$ in $\GammaO^-_n(q)$, the Correspondence Theorem implies that 
 $B=N_G(B\cap \Omega_{2m}^-(q))$. We consider separately the three cases in  Table~\ref{T:OEvenCases}.

\begin{table}[!ht]
\caption{Cases for $L=\POmega_{2m}^{-}(q)$ where $m\ge4$.}\label{T:OEvenCases}
\begin{tabular}{rlcccc}
\toprule
$\cU$ && $\cP_1$&$\cP_1$&  $\cN_2^{+}$\\
$(m,q)$&&$(\textup{odd},\,{\rm all})$&$(5,2)$&$(\textup{odd},4)$\\
$B$&&$N_G(\SU_m(q))$&$N_G(\Alt_{12})$&$N_G(\SU_m(4))$\\
\bottomrule
\end{tabular}
\end{table}

   {\sc Case}~$\cU=\cP_1$ \emph{with  $m$ odd and $B=N_G(\SU_m(q))$}.\quad 
Set
$A=\SU_m(q)$, $F=\F_q$, $K=\F_{q^2}$, view $V=F^{2m}=K^m$ as
an $A$-module, and let $[\ ,\,]$ be the $A$-invariant $K$-Hermitian form  
such that $Q(v)=[v,v]$ for all $v\in V$, where $Q$ is the quadratic form corresponding to $G$. Let $U\in\cU$ and let 
$U=\langle u\rangle_F$, the 1-dimensional $F$-subspace containing $u$. Then $Q(u)=0$ since $U$ is singular. Let $\langle u\rangle_K$ be the 
1-dimensional $K$-subspace containing $u$. Since $[u,u]=Q(u)=0$, the subspace $\langle u\rangle_K$ 
is isotropic  with respect to $[\,,\,]$. By~\cite{KL}*{Lemma~2.10.5}, $A$ is 
transitive on the 1-dimensional isotropic $K$-subspaces of $V$. Since $m\ge 4$, the stabiliser in $A$ of $\langle u\rangle_K$ induces on $\langle u\rangle_K$  multiplications by each nonzero element of $K$, and hence acts transitively on the set of $F$-subspaces of $\langle u\rangle_K$. Thus $A$ acts transitively on $\cP_1$
and we list the entry $\SU_m(q)\lhdeq H$ in Table~\ref{T:OEvenP1}.

Suppose now that $\SU_m(q)\not\leq H$. Then we have a factorisation $B=HB_U$. Note that $B_U$ is contained in the stabiliser in $B$ of a totally isotropic subspace of dimension 1 over $K$ and so we also have a factorisation $B=HP_1(B)$. Since $\SU_m(q)\not\leq H$ we obtain a core-free factorisation of an almost simple group with socle $\PSU_m(q)$. Since $m\geq 4$ is odd, it follows from Lemma~\ref{lem:maxmin} and \cite{Factns} that $(m,q)=(9,2)$ and $\oH\leqslant N_{\oB}(\textup{J}_3)$. It follows from \cite{BHRD}*{Table~8.57} that $H\leqslant N_B(\nonsplit{3}{\textup{J}_3})$. By~\cite{Factns}, $3.\textup{J}_3$ is transitive on the 1-dimensional isotropic $\F_4$-subspaces of $V$, and since it contains all $\F_4$-scalars, it is transitive on the set $\cP_1$  of singular 1-dimensional $\F_2$-subspaces of $V$. This leads to the example $\nonsplit{3}{\textup{J}_3}\lhdeq H$ in Table~\ref{T:OEvenP1}. Since $\textup{J}_3$ has no factorisations by \cite{Factns}, there are no further examples.

\smallskip

 {\sc Case}~$\cU=\cP_1$ \emph{with $m=5$, $q=2$ and $B=N_G(\Alt_{12})= \Sym_{12}$}.\quad  We view $V=\F_2^{\,10}$ as a section $S/T$ of a 12-dimensional
space $\F_2^{\,12}$ following~\cite{Factns}*{(5.2.16), p.\,115}, where $$S=\{x\in\F_2^{\,12}\mid \sum_{i=1}^{12}x_i=0\}$$ and $T=\langle (1,\dots,1)\rangle$. 
Clearly, $\Sym_{12}$ preserves the quadratic form 
$$Q(x)=\left\{\begin{array}{cr}
                       1& \mbox{if}\ |\{i \mid x_i=1\}| \equiv 2 \pmod 4 \\
                       0& \mbox{if}\ |\{i \mid x_i=0\}| \equiv 0 \pmod 4
                       \end{array}\right. $$
on $S$ with associated symmetric bilinear form $(x,y)=\sum_{i=1}^{12}x_iy_i$. Moreover, $T=S^\perp$.
 Thus the action of $\Sym_{12}$ on $S/T$ gives an embedding of $\Sym_{12}$ in  $\GO_{10}^{-}(2)$ (see \cite{Factns}*{p.\,34}).  Now  $u=(1,1,1,1,0,\dots,0)$ (8 zeroes) lies in $S$ and $Q(u)=0$.  We may take
$U=\langle u+T\rangle\in\cP_1$. Now \cite{Factns} implies that $\Alt_{12}$ is transitive on $\cP_1$, and the possibility $\Alt_{12}\lhdeq H$ is 
listed in Table~\ref{T:OEvenP1}. It remains to consider the case where $\Alt_{12}\not\leqslant H$. In this case we have a factorisation $B=HB_U$. Now $B_U=B\cap (\Sym_4\times\Sym_8)$ and so $H$ is a 4-homogeneous subgroup of $\Sym_{12}$. Thus \cite{Kantor} implies that $H$ is 4-transitive and so $H=\textup{M}_{12}$. Indeed $B=M_{12}B_U$, so $G=M_{12}P_1$, and the example $M_{12}$ is listed in Table~\ref{T:OEvenP1}, concluding  this case.

\smallskip
 {\sc Case}~$\cU=\cN_2^{+}$ \emph{with $m$ odd, $q=4$, and $B=N_G(\SU_m(q))$}. \quad Here $G=\GammaO^-_m(4)$ and so $B=\GammaU_m(4)$.  Suppose first that $\SU_m(4)\not\leqslant H$ and let $U\in\cU$. Then we have a core-free factorisation $B=HB_U$. However, since $m$ is odd and $q=4$ there are no such factorisations by Lemma~\ref{lem:maxmin} and \cite{Factns}, a contradiction. Thus $\SU_m(4)\lhdeq H$ as stated in Theorem~\ref{T9}(c). Moreover,   by  the Remark column of \cite{Factns}*{p.\,11} we have $\overline{H}=N_{\overline{G}}(\SU_m(4))$ and so $H$ and  $\GammaU_m(4)$ project onto the same subgroup of $\PGammaO_{2m}^-(4)$, proving the assertions of part (c) in this case.

\smallskip
 {\sc Case}~$\cU=\cN_2^{+}$ \emph{with $q=2$}. When $q=2$, if $U\in\cN_2^+$ then $U$ contains a unique nonsingular 1-subspace. Thus if $H\leqslant \GO_{2m}^-(2)$ acts transitively on $\cN_2^+$ then it is also transitive on $\cN_1$. Hence $H$ is given in part (b) of the theorem. By~\cite{BG}*{Table~4.1.2} and the fact that $|\cN_2^+|=|\cN_{2m-2}^-|$ we have that $|\cN_2^+|=2^{2m-3}(2^m+1)(2^{m-1}-1)$. Hence, unless $m=7$, it follows that $|H|$ is divisible by a primitive prime divisor of $2^{m-1}-1$. Looking at the five possibilities for $H_0$  and using the fact that $m\geq 4$, it follows that $m$ is odd and $\SU_m(2)\lhdeq H$. By Lemma~\ref{lem:SUorthog} this is indeed an example as listed in part (c) of the Theorem.
\end{proof}

As mentioned in Subsection~\ref{sub:spaces}, when $kq$ is odd we also determine the subgroups which are transitive on $\cN_k=\cN_k^+\cup\cN_k^-$. 

\begin{proposition}\label{P1}
Suppose that $H\leqslant \GammaO_{2m}^-(q)$, where $m\geq 2$ and $q$ is odd, such that  $\Omega_{2m}^-(q)\not\leqslant H$. Let $V=(\F_q)^{2m}$ and suppose that $H$ acts transitively on the set of all nondegenerate $k$-subspaces of $V$ for some odd $k$. Then $k=1$, $m$ is odd, and either $\SU_m(q)\lhd H$, or $(m,q)=(3,3)$ and $\nonsplit{2}{\PSL_3(4)}\lhd H$. 
\end{proposition}

\begin{proof}
Let $\cU=\cN_k=\cN_k^+\cup\cN_k^-$, the set of all nondegenerate $k$-subspaces of $V$. 
Note that $\GammaO^-_{2m}(q)$ has an index two subgroup with orbits $\cN_k^+$ and $\cN_k^-$ on $\cU$. Thus $H$ has an index two subgroup $H^+$ which is transitive on both $\cN_k^+$ and $\cN_k^-$. Hence by Theorem~\ref{T9}, $k=1$,  and either $m$ is odd with $H$ is as in the statement of the proposition, or $(m,q)=(2,3)$ and $\Alt_5\lhd H^+$. However, as noted in the proof of Theorem~\ref{T:KleinPSL2q2}, each of the latter groups (containing an $\Alt_5$-subgroup) acts transitively on only one of the $\cN_1^\varepsilon$ and so does not provide an example. The other possibilities do indeed provide examples: by \cite{KL}*{Proposition~4.3.18}, $\Omega_{2m}^-(q)$ contains only one conjugacy class of $\SU_m(q)$-subgroups and so each such subgroup acts transitively on each orbit $\cN_1^\eps$. Moreover, since there is only one conjugacy class, the normaliser in $\GammaO_{2m}^-(q)$ of $\SU_m(q)$ contains an element of $\CO_{2m}^-(q)\backslash \GO^-_{2m}(q)$ and so $N_{\GammaO_{2m}^-(q)}(H^+)$ is transitive on $\cU$. Similarly, when $(m,q)=(3,3)$ we see from \cite{BHRD}*{Table~8.34} that $H^+=\nonsplit{2}{\PSL_3(4)}$ is normalised by an element of $\CO^-_{6}(3)\backslash \GO^-_{6}(3)$ and so $N_{\CO^-_6(3)}(\nonsplit{2}{\PSL_3(4)})$ acts transitively on $\cU$.
\end{proof}

\section{Orthogonal groups in odd dimension}\label{S:OOdd}

Let $V=(\mathbb{F}_q)^n$,  with $n$ odd,  equipped with a nondegenerate quadratic form $Q$.  
The form $Q$ and the associated polar form $(\cdot,\cdot)$ may be chosen as follows.  Relative to a basis
$e_1,\dots,e_m,f_1,\dots,f_m,d$ for $V$ we 
define
\[
  Q\left(zd+\sum_{i=1}^m (x_ie_i+y_if_i)\right)=z^2+\sum_{i=1}^m x_iy_i\quad\textup{and}\quad (u,v)=Q(u+v)-Q(u)-Q(v).
\]
Note that $Q(d)=1$ and $(d,d)=2$. If $q$ is even then $V^\perp=\langle d\rangle$, a  nonsingular $1$-subspace, while if 
$q$ is odd then the Gram matrix for $(\cdot,\cdot)$ is invertible (as $(d,d)=2\ne0$) and the radicals of both $Q$ and $(\cdot,\cdot)$ are trivial. Thus the cases where $q$ is odd and even are quite different and are treated in separate Subsections~\ref{S:OOddOddq} and~\ref{S:OoddEvenq}.

\subsection{Orthogonal groups in odd dimension \texorpdfstring{$n=2m+1\geq 3$}{} and \texorpdfstring{$q$}{} odd}\label{S:OOddOddq}

Let  $\cU$  be the set of all totally singular subspaces of $V$ of dimension $k\leq m$, or the set of all nondegenerate subspaces of dimension $k\leq m$ of a fixed isometry type. Then by \cite{BG}*{Lemma~2.5.10}, $\Omega_n(q)$ is transitive on $\cU$. Note that if $k$ is odd then there are two isometry types of nondegenerate subspaces of dimension $k$ but only one similarity type. However, their orthogonal complements have even dimension and are not similar (for some isometry type they are of plus type and for the other isometry type they are of minus type.) Hence $\cU$ is also a $\GammaO_n(q)$-orbit.

 The following formulas for $|\cU|$ are useful references.
\begin{equation}\label{E:OSizeN1Pm}
  |\cP_m|=\prod_{i=1}^m(q^i+1),\qquad\textup{and}\qquad 
  |\cN_1^\eps|=\frac{q^m(q^m+\eps1)}{2}.   
\end{equation}

\begin{theorem}\label{T4}
  Suppose that  $H\le\GammaO_{2m+1}(q)$, where $m\geq1$ and $q$ is odd, such that  $\Omega_{2m+1}(q)\not\le H$. Suppose that $H$ acts transitively on a set $\cU$ of subspaces of $V=(\F_q)^{2m+1}$ given by Table~$\ref{T-defn}$. Then either $m=1$ and $H$ is given by Theorem~$\ref{T:Omega3}$, or $m=2$ and $H$ is given by Theorem~$\ref{T:KleinSp4}$, or $m\geq3$ and one of the following holds:
  \begin{enumerate}[{\rm (a)}]
    \item $\cU=\cP_1$, $m=3$, and $\G_2(q)\lhdeq H$; or
  \item $\cU=\cP_m$ and $H_0\lhdeq H$ where $(H_0,m,q)$ is given in
    Table~\textup{\ref{T:Oa}}; or
\begin{table}[!ht]
\caption{Theorem~\ref{T4}(b) subgroups $H_0\lhdeq H\le\GammaO_{2m+1}(q)$ with $H$
  transitive on $\cP_m$ with $m\geq 3$.}\label{T:Oa}
\begin{tabular}{rlcccccc}
\toprule
$H_0$ &&$\Omega^{-}_{2m}(q)$ &$\Alt_8$&$\Alt_9$ &$2^6\ldotp\Alt_7$&$\nonsplit{2}{\PSL_3(4)}$&$\Omega_7(2)$\\
\textup{$(m,q)$}&&$(m,{\rm odd})$& $(3,3)$&$(3,3)$&$(3,3)$&$(3,3)$&$(3,3)$\\
\bottomrule
\end{tabular}
\end{table}
  \item $\cU=\cN_1^{-}$ and either $H_0\lhdeq H$ where $H_0$ is one of the groups listed in Table~\textup{\ref{T:Oc}}, or $H\le P_m$ where $H$ induces a subgroup
  of $\GammaL_m(q)$ which is transitive on $(m-1)$-subspaces.
  
\begin{table}[!ht]
\caption{Theorem~\ref{T4}(c) subgroups $H_0\lhdeq H\le\GammaO_{2m+1}(q)$ with $H$
  transitive on $\cN_1^-$.}\label{T:Oc}
\begin{tabular}{rlccccccccc}
\toprule
$H_0$ &&$\textup{F}_4(q)$&$\PSp_6(q)$& $\G_2(q)$ &$\SL_3(q)$ \\
$(m,q)$&& $(12,3^f)$&$(6,3^f)$&$(3,\textup{odd})$&$(3,3^f)$\\
\bottomrule
\end{tabular}
\end{table}

  \item $\cU=\cN_1^+$,  $m=3$, and $H_0\lhdeq H$ where $H_0$ is given in
    Table~\textup{\ref{T:Oc2}}; or

\begin{table}[!ht]
\caption{Theorem~\ref{T4}(d) subgroups $H_0\lhdeq H\le\GammaO_{7}(q)$ with $H$
  transitive on $\cN_1^+$.}\label{T:Oc2}
\begin{tabular}{rlccccccc}
\toprule
$H_0$ &	&$\G_2(q)$ 		&$\SU_3(q)$& ${}^2\G_2(q)$&$\Alt_9$	&$\Omega_7(2)$\\
$q$&		&$\textup{odd}$&$3^f$		& $3^{2e+1}$	& $3$			&$3$\\
\bottomrule
\end{tabular}
\end{table}

  \item $\cU=\cN_2^{\eps}$, $m=3$, $q=3^f$, and $\G_2(q)\lhdeq H$. 
\end{enumerate}
\end{theorem}

\begin{proof}
 Let $G=\GammaO_{2m+1}(q)$. By Table~\ref{T-defn}, $\cU$ is $\cP_k$ or $\cN_k^\eps$ for some $k\leq m$ and $\eps=\pm$, and by our discussion above, $G$ leaves $\cU$ invariant. Let $H\leqslant G$ act transitively on $\cU$  such that $\Omega_{2m+1}(q)\not\leqslant H$. If $m=1$ then the result follows from Theorem~\ref{T:Omega3}, while if $m=2$ then it follows from Theorem~\ref{T:KleinSp4}.  Thus from now on we assume that $m\geq 3$. 
 Let $U\in \cU$. Then we have a factorisation $G=HG_U$.  Since $m\geq 3$ and $q$ is odd, \cite{King1981} and \cite{King1981a} imply that $\GO_n(q)$ acts primitively on $\cU$ unless $q=3$ and $\cU=\cN_2^+$. We will treat the latter case separately at the end. Thus we assume for now that $G$ acts primitively on $\cU$ and so $G_U$ is maximal in $G$.

We denote the image of each $X\leqslant \GammaO_{2m+1}(q)$ in $\PGammaO_{2m+1}(q)$ by $\overline{X}$, and let $L=\POmega_{2m+1}(q)$. Then we have a factorisation $\PGammaO_{2m+1}(q)=\overline{H}\,\overline{G}_U$. Since $\Omega_{2m+1}(q)$ is quasisimple, Lemma~\ref{lem:derived} implies that $L\not\leqslant \oH$. Let $\overline{B}$ be maximal among core-free subgroups of $\PGammaO_{2m+1}(q)$ containing $\overline{H}$ and let $G^*=\overline{B}L$. Then by Lemma~\ref{lem:maxmin}, $\overline{B}$ is maximal in $G^*$ and we have a core-free maximal factorisation $G^*=\overline{B}(G^*)_U$. In particular $\cU$ and $\overline{B}$ are given by \cite{Factns}. Moreover, $\overline{B}=N_{\PGammaO_{2m+1}(q)}(\overline{B}\cap L)$ and so letting $B$ be the full preimage of $\oB$ in $\GammaO_n(q)$, the Correspondence Theorem implies that $B=N_G(B\cap \Omega_{2m+1}(q))$.

Besides a large number of factorisations with $m=q=3$ listed
on~\cite{Factns}*{p.\,13}, we must consider the cases
in Table~\ref{T:OOddCases} listed on~\cite{Factns}*{p.\,11, p.\,12}. We note the case $\cU=\cP_1$ when $m=3$ was missed in \cite{Reg}*{Lemma 4.1}.
\begin{table}[!ht]
\caption{Infinite families for $L=\POmega_{2m+1}(q)$ where  $q$ is odd. Further, $q=3^f$ in cases (c3) and (c4), and $b_2\leqslant 2$ in case (c3).}\label{T:OOddCases}
\begin{tabular}{rlccccccccc}
\toprule
\textup{Case} && (a)& (b)& (c1)&(c2)&(c3)& (c4)& (d)&(e)\\
$\cU$ && $\cP_1$&$\cP_m$& $\cN_1^{-}$& $\cN_1^-$& $\cN_1^-$&$\cN_1^-$&$\cN_1^+$&  $\cN_2^\eps$\\
$m$&&$3$&$\ge3$&$\ge3$&$3$&$6$&$12$&$3$&$3$\\
$\oB\cap L$&&$\G_2(q)$&$N_1^{-}$&$P_m$&$\G_2(q)$&$\PSp_6(q)\ldotp b_2$&$\textup{F}_4(q)$& $\G_2(q)$& $\G_2(q)$\\
\bottomrule
\end{tabular}
\end{table}

\smallskip

{\sc Case}~$\cU=\cP_1$. By~\cite{Factns}, the only possibility is $m=3$ and $B=N_G(\G_2(q))$. Moreover, $\G_2(q)$ is indeed transitive on $\cP_1$, as listed in case (a) of the Theorem. Suppose now that $\G_2(q)\not\leqslant H$. Then we have a factorisation $B=HB_U$, where $U\in\cP_1$. As observed in \cite{Factns}*{p100}, $B_U$ is a parabolic subgroup of $B$, and there are no such factorisations in the list for almost simple groups with socle $\G_2(q)$ given in \cite{Factns}*{Table~5}. 

\smallskip
{\sc Case}~$\cU=\cP_m$.  By~\cite{Factns}, either $B=N_1^-$ or $(m,q)=(3,3)$. We may take
$U=\langle e_1,\dots, e_m\rangle\in\cP_m$. Suppose first that $H\leqslant N_1^-=N_G(\Omega_{2m}^-(q))$. We follow~\cite{Factns}*{3.4.1, p.\,57} and set $v=e_1+\lambda f_1$ where
$\lambda\in\F_q^\times$ is a nonsquare and $B=G_{\langle v\rangle}$.    Suppose first that $\Omega_{2m}^-(q)\not\leqslant H$. Then $N_1^-=B=HB_U$. This factorisation yields the factorisation $B^{v^\perp}=H^{v^\perp}(B_U)^{v^\perp}$ which then yields a factorisation of an almost simple group $X$ with socle $\POmega_{2m}^-(q)$. Moreover, as $B_U$ fixes the totally singular $(m-1)$-subspace $U\cap v^\perp=\langle e_2,\ldots,e_m\rangle$, we obtain a core-free factorisation of $X$ with one factor being the stabiliser of a totally singular $(m-1)$-subspace. However, such a factorisation does not exist if $m\ge4$
by~\cite{Factns}*{pp.\,11--13} and Theorem~\ref{T:KleinSU4} implies that when $m=3$ we must have $q=3$.  We note that for $(m,q)=(3,3)$, other possibilities for $B$ also occur and so at this point it is convenient to
use {\sc Magma} to list the subgroups of $\GammaO_7(3)=\GO_7(3)$ that act transitively
on $\cP_3$. The new possibilities are recorded in Table~\ref{T:Oa}. 

We conclude this case by showing that $A=\Omega_{2m}^-(q)=\Omega_{2m+1}(q)_v$ acts transitively on $\cP_m$ and so any $H$ with $\Omega_{2m}^-(q)\lhdeq H$ acts transitively, as recorded in Table~\ref{T:Oa}. Now by \cite{BG}*{Lemma 2.5.10}, $A$ acts transitively on the set of all totally singular $(m-1)$-subspaces of $v^\perp$ and so it remains to show that $A_{U\cap v^\perp}$ acts transitively on the set of elements of $\cP_m$ that contain $U\cap v^\perp$. Now $(U\cap v^\perp)^\perp=\langle e_1,e_2,\ldots,e_m,f_1,d\rangle$ and note that $\langle e_1,f_1,d\rangle \perp \langle e_2,\ldots, e_m\rangle$. Thus an element $a+b$ with $a\in \langle e_1,f_1,d\rangle$ and $b\in \langle e_2,\ldots, e_m\rangle$ is singular if and only if $a$ is singular. Hence the elements of $\cP_m$ containing $U\cap v^\perp$ correspond to the totally singular 1-subspaces of $\langle e_1,f_1,d\rangle$. Since $v=e_1+\lambda f_1\in \cN_1^-$ it follows that $W:=\langle e_1-\lambda f_1,d\rangle\leqslant v^\perp$ is a nondegenerate 2-subspace of minus type. Moreover, \cite{KL}*{Lemmas~4.1.1 and~4.1.12} imply that $A_{U\cap v^\perp,W }$ induces at least $\SO_2^-(q)$ on $W$ and hence for each $\mu\neq 0$, is transitive on all vectors $w\in W$ with $Q(w)=\mu$. Since $W$ is of minus type, any singular 1-subspace of $v\perp W$ is spanned by a vector of the form $v+a(e_1-\lambda f_1)+bd$ for some $a,b\in\F_q$.  Such a vector is singular if and only if $b^2=\lambda(a^2-1)$. However, in this case $Q(a(e_1-\lambda f_1)+bd)=-\lambda a^2+b^2=-\lambda$, a constant. Hence $A_{U\cap v^\perp,W }$ is transitive on all such vectors and thus all singular 1-subspaces in $\langle v,W\rangle$. Hence $A$ is transitive on $\cP_m$.

\smallskip
{\sc Case}~$\cU=\cN_1^{-}$. Here there are four possibilities for $B$ from Table~\ref{T:OOddCases}, namely (c1)~$P_m$,  (c2)~$N_G(\G_2(q))$ with $m=3$, (c3)~$N_G(\PSp_6(q))$ with $m=6$, and
(c4)~$N_G(\textup{F}_4(q))$ with $m=12$. Moreover, cases (c3) and (c4) only occur when $q=3^f$.
Let $U\in\cN_1^{-}$ where $\dim(U)=1$, $\dim(U^\perp)=2m$
and the quadratic form restricted to $U^\perp$ has type~$\Or^{-}$. 

(c1)~Suppose first that $H\leqslant P_m$. As seen above, we have that $N_1^-\cap P_m$ fixes an $(m-1)$-dimensional subspace of the totally singular $m$-subspace $W$ fixed by $P_m$.
Restricting to $W$ gives an epimorphism
$\pi\colon P_m\twoheadrightarrow \GammaL_m(q)$. This gives a factorisation $\pi(P_m)=\pi(H)\pi(P_m\cap N_1^-)$. Since $P_m\cap N_1^-$ fixes a hyperplane of $W$ it follows that $\pi(H)$ is a subgroup of $\GammaL_m(q)$ that is transitive on hyperplanes as in case (c).

(c2)~Suppose now that $H\leqslant N_G(\G_2(q))=B$ where $m=3$. Then $B=HB_U$. By~\cite{Factns}*{p.\,100}, $\G_2(q)_U=\SU_3(q).2$, and so by the classification of factorisations of almost simple groups with socle $\G_2(q)$, see \cite{Factns}*{Table~5}, we deduce that either $\G_2(q)\lhdeq H$ or $q=3^f$ and $\SL_3(q)\lhdeq H$ as listed in Table~\ref{T:Oc}.

(c3)~Suppose next that $H\leqslant N_G(\PSp_6(q))=B$ where $m=6$ and $q=3^f$. Note that \cite{Factns}*{Lemma~A on p.\,85} implies that $\PSp_6(q)$ acts transitively on $\cN_1^-$ and so any group $H$ with $\PSp_6(q)\lhdeq H$ is an example as in Table~\ref{T:Oc}. Suppose then that $\PSp_6(q)\not\leqslant H$. Then we have a factorisation $B=HB_U$ and hence a factorisation of an almost simple group with socle $\PSp_6(q)$. By~\cite{Factns}*{p.\,86} we have that $\PSp_6(q)_U=(\Sp_2(q)\times\Sp_2(q^2))/\{\pm I\}$. Moreover, $\PSp_6(q)_U$ is contained in an $N_2$ subgroup of $\PSp_6(q)$ and hence $B=HN_2$. However, by \cite{Factns} when $q=3^f$ an almost simple group with socle $\PSp_6(3^f)$ does not have a core-free factorisation with one factor being an $N_2$ subgroup, contradicting $\PSp_6(q)\not\leqslant H$.

(c4)~Finally, suppose that $H\leqslant N_G(\textup{F}_4(q))$ where $m=12$ and $q=3^f$. If $\textup{F}_4(q)\not\leqslant H$ then we obtain a core-free factorisation of an almost simple group with socle $\textup{F}_4(q)$. However, by \cite{Factns}*{Table~5} there are no such factorisations when $q$ is odd. Thus we must have $\textup{F}_4(q)\lhdeq H$ as in Table~\ref{T:Oc}.

\smallskip
{\sc Case}~$\cU=\cN_1^+$.  By~\cite{Factns}, either $m=3$ and $H\leqslant N_G(\G_2(q))$ or $(m,q)=(3,3)$. Moreover, $\G_2(q)$ acts transitively on $\cN_1^+$ giving us one of the cases in Table~\ref{T:Oc2}. Let $U\in\cN_1^+$. Suppose that $H\leqslant N_G(\G_2(q))=B$ such that $\G_2(q)\not\leqslant H$. Then we have a factorisation $B=HB_U$. By~\cite{Factns}*{p.\,100} we have $\G_2(q)_U=\SL_3(q).2$ and so by the classification of factorisations of almost simple groups with socle $\G_2(q)$ \cite{Factns}*{Table~5} we deduce that $q=3^f$ and either $\SU_3(q)\lhdeq H$ or $f$ is odd and ${}^2\G_2(q)\lhdeq H$. These cases are listed in Table~\ref{T:Oc2}. It remains to consider the exceptional factorisations when $(m,q)=(3,3)$. In this case $\GammaO(7,3)=\GO(7,3)$ and a \textsc{Magma} calculation gives the remaining cases in Table~\ref{T:Oc2}. Note that ${}^2\G_2(3)\cong\PGammaL_2(8)$.

\smallskip
{\sc Case}~$\cU=\cN_2^\eps$ with $(\eps,q)\neq (+,3)$. By~\cite{Factns}, we again have that $m=3$ and $B=N_G(\G_2(q))$. Moreover, by~\cite{Factns}*{Lemma~A p.\,100}, $\G_2(q)$ is transitive on $\cN_2^\eps$ for all $q$ (including the case $(\eps,q)=(+,3)$ where $N_2^+$ is not maximal in $G$). Suppose that $\G_2(q)\not\leqslant H$. Then we have a factorisation $B=HB_U$, where $U\in\cN_2^\eps.$ By \cite{Factns}*{p.\,101}, $|\G_2(q)_U|=2q(q^2-1)(q-\eps 1)$. The possible factors for a core-free factorisation of an almost simple group with socle $\G_2(q)$ are given in \cite{Factns}*{Table~5}, but in no case is the order of the intersection of a factor with $\G_2(q)$ equal to $|\G_2(q)_U|$. Thus $\G_2(q)\lhdeq H$, as listed in the case (e) of the Theorem.

It remains to deal with the case where $q=3$ and $\cU=\cN_2^+$. Let $U\in\cU$. Then $U$ contains four 1-dimensional subspaces with two being singular and two being nonsingular. Moreover, the nonsingular subspaces are not isometric and so one lies in $\cN_1^+$ and one lies in $\cN_1^-$. Thus if $H\leqslant \GammaO_{2m+1}(q)$ acts transitively on $\cN_2^+$ then it also acts transitively on both $\cN_1^+$ and on $\cN_1^-$. By parts (c) and (d) it follows that $m=3$ and $\G_2(3)\lhd H$, as listed in case (e) of the Theorem. As mentioned in the previous paragraph, $\G_2(3)$ is indeed transitive.
\end{proof}

\subsection{Orthogonal groups with odd dimension \texorpdfstring{$n=2m+1\ge3$}{} and
\texorpdfstring{$q$}{} even}\label{S:OoddEvenq}

As discussed above, the quadratic form~$Q$ on  $V=(\F_q)^{2m+1}$ has trivial radical, while the radical  $V^\perp$ of its associated symmetric bilinear form $(\cdot,\cdot)$  is the nonsingular 1-subspace $\langle d\rangle$. Moreover, the  form $(\cdot,\cdot)$ induces a nondegenerate alternating form on $V/V^\perp$ given by $(u+V^\perp,v+V^\perp)=(u,v)$ for all $u,v\in V$.
Since $\GammaO_{2m+1}(q)$ acts faithfully on $V/V^\perp$ it follows that $\GammaO_{2m+1}(q)\cong\GammaSp_{2m}(q)$. 
The families $\cU$ of subspaces we treat in this case (as a careful reading of Table~\ref{T-defn} reveals) are $\cP_k$ for $k\leq m$, $\cN_k^\eps$ for even $k\leq m$ and  $\eps\in\{+,-\}$, $\cN_1$ (all nonsingular $1$-subspaces apart from  $\langle d\rangle$),  and $\cN_{2m}^\eps$ for $\eps\in\{+,-\}$.

\begin{theorem}\label{T:Ooddqeven}
  Suppose that  $H\le\GammaO_{2m+1}(q)$, where $m\geq1$ and $q$ is even, such that $\Omega_{2m+1}(q)'\not\le H$. Suppose that $H$ acts transitively on a set $\cU$ of subspaces of $V=(\F_q)^{2m+1}$ given by Table~$\ref{T-defn}$. Then either $m=1$ and $H$ is given by Theorem~$\ref{T:Omega3}$, or $m\geq2$ and one of the following holds:

\begin{enumerate}[{\rm (a)}]
  \item $\cU=\cP_1$,  $H_0\lhdeq H$ where $(H_0,m,q)$ is given in
    Table~\textup{\ref{tab:Ooddqeven}};~or

\begin{table}[ht!]
\caption{Theorem~\ref{T:Ooddqeven}(a),(c) subgroups $H_0\lhdeq H\le\GammaO_{2m+1}(q)$ where $H$ is transitive
on $\cP_1$  and on $\cN_1$ and $q$ is even.}\label{tab:Ooddqeven}
\begin{tabular}{rlccc}
\toprule
$H_0$ &&$\Sp_{2a}(q^b)$ &$\textup{G}_2(q^b)'$\\
\textup{$m$}&&$ab>a$&$3b$\\
\bottomrule
\end{tabular}
\end{table}

  \item $\cU=\cP_m$,  and $\Omega_{2m}^{-}(q)\lhdeq H$ where $m\ge2$; or
\item $\cU=\cN_1$, and $H_0\lhdeq H$ where $(H_0,m,q)$ is given in
    Table~\textup{\ref{tab:Ooddqeven}}; or 
  \item $\cU=\cN_2^\eps$ and $H_0\lhdeq H$ where $(H_0,m,q)$ is given in Theorem~$\ref{T3}${\rm(c)}; or
\item  $\cU=\cN_{2m}^-$, and either $H$ is described in ~\cite{Reg}*{Lemma~4.6}, or $q=2$ or $4$ and $\Omega_{2m}^+(q)\lhdeq H$, or $(m,q)=(2,2)$ and  $H$ is a transitive subgroup of $\Sym_6$; or
\item $\cU=\cN_{2m}^{+}$ and $H_0\lhdeq H$ where $H_0$
  is given in Table~\textup{\ref{T:Spd}} and $b_4\in\{1,2,4\}$.
\begin{table}[!ht]
\caption{Theorem~\ref{T:Ooddqeven}(f) subgroups $H_0\lhdeq H\le\GammaO_{2m+1}(q)$ where $H$ is transitive
on $\cN_{2m}^{+}(\F_q^{\;2m+1})$ and $q$ is even.}\label{T:Spd}
\begin{threeparttable}
\begin{tabular}{rlcccccccc}
\toprule
$H_0$ &&$\Sp_{2a}(q^b)$ &$\SU_{a}(q^{b_4})$&$\Omega_{2a}^{-}(q^{b_4})$  &$\G_2(q^b)'$&$\Sz(q^b)$\\
\textup{$m$}&&$ab$ &$ab_4$\;\textup{($a$ odd)}&$ab_4$& $3b$&$2b$\;\textup{($b$ odd)} \\
\textup{$q$}&&\textup{even} &$2,4\tnote{*}$&$2,4 \tnote{*}$&$q$&$2^\textup{odd}$ \textup{with} $q^b\geq 8$\\
\midrule
$H_0$ &  &$\PSL_2(17)$&$3_{+}^{1+2}$&\\
\textup{$m$}&&$4$&$3$&\\
\textup{$q$}&&$2$&$2$&\\
\bottomrule
\end{tabular}
\begin{tablenotes}\footnotesize
\item[*] If $q=4$ then $b_4=1$; if $q=2$ then $b_4\in\{1,2\}$.
\end{tablenotes}
\end{threeparttable}
\end{table}
\end{enumerate}
\end{theorem}

\begin{proof}
Let $G=\GammaO_{2m+1}(q)$. The possibilities for $\cU$ were given before the statement, and in each case $G$ leaves $\cU$ invariant.  Let $H\leqslant G$ act  transitively on  $\cU$  such that $\Omega_{2m+1}(q)\not\leqslant H$. If $m=1$ then the possibilities for $H$ are given by Theorem~\ref{T:Omega3}. We assume from now on that $m\geq2$.  Let $\overline{V}=V/V^\perp$ and for each $U\in \cU$ let $\overline{U}=(U+V^\perp)/V^\perp$.    Note that $\Omega_{2m+1}(q)=\SO_{2m+1}(q)=\GO_{2m+1}(q)\cong\Sp_{2m}(q)$, unless $(m,q)=(2,2)$, in which case $|\SO_5(q):\Omega_5(2)|=2$.  Hence, if $\Omega_{2m+1}(q)'\not\leqslant H$  then $H$ does not induce $\Sp_{2m}(q)'$ on $\overline{V}$.

\smallskip
{\sc Case}~$\cU=\cP_k$ for some $k\leq m$.  Since $U\in\cU$ is totally singular,  $U\cap V^\perp=\{0\}$ and so $\overline{U}$ is a totally isotropic $k$-subspace of $\overline{V}$.  
Moreover, since $\Sp_{2m}(q)$ acts transitively on the set of totally isotropic  $k$-subspaces of $\overline{V}$ it follows that all such subspaces of $\overline{V}$ arise in this manner. Moreover, by~\cite{Taylor}*{p.\,143}, $U$ is the unique totally singular $k$-subspace of $U+V^\perp$ and so there is a one-to-one correspondence between $\cU$ and the set of totally isotropic $k$-subspaces of $\overline{V}$.  
Thus $H$ acts transitively on $\cU$ if and only if it acts transitively on the set of totally isotropic $k$-subspaces of $\overline{V}$. Thus by Theorem~\ref{T3} we have $k=1$ or $m$ and $H$ is listed there. This gives parts (a) and (b) and note that $\Sp_{2a}(q^b)\cong \Omega_{2a+1}(q^b)$.

\smallskip
{\sc Case}~$\cU=\cN_1$. Here $U\in \cU$ is a nonsingular 1-subspace other than $V^\perp=\langle d\rangle$. Then $U+V^\perp$ contains a unique singular 1-subspace $\langle w\rangle$ and the nonsingular 1-subspaces of $U+V^\perp$ are $\langle d+\lambda w\rangle$ for $\lambda\in\F_q$. Moreover, any singular 1-subspace $\langle w\rangle$ in $V$ gives rise to $q-1$ nonsingular 1-subspaces in $V^\perp+\langle w\rangle$ 
distinct from $V^\perp$. Hence $H\leqslant G$ is transitive on $\cU$ if and only if $H$ is transitive on $\cP_1$ and for each $\langle w\rangle \in\cP_1$, the subgroup $H_{\langle w\rangle}$ acts transitively on the set of $1$-subspaces $\langle d+\lambda w\rangle$ for $\lambda\in\F_q\setminus\{0\}$, equivalently,  $H_{\langle w\rangle}$ acts transitive on the nonzero vectors in $\langle w\rangle$.   Hence by part (a),  $H_0\lhdeq H$ where $H_0$ is given in Table~\ref{tab:Ooddqeven}. Moreover, for both possibilities, $(H_0)_{\langle w\rangle}$ induces a cyclic group of order $q-1$ on $\langle w\rangle$ (see \cite{KL}*{Lemma~4.1.2} when $H_0=\Sp_{2a}(q^b)$ and \cite{rW}*{Section~4.3.5} for $H_0=G_2(q^b)'$). Thus both possibilities for $H_0$ are also transitive on $\cN_1$, proving part~(c). 

\smallskip
{\sc Case}~$\cU=\cN_k^\eps$ for some even $k\leq m$ and $\eps\in\{+,-\}$. Here $U\in\cU$ is a nondegenerate $k$-subspace. Hence $V^\perp\cap U=\{0\}$ and $\overline{U}$ is a nondegenerate $k$-subspace of $\overline{V}$. Since $\Sp_{2m}(q)$ is transitive on the set of nondegenerate $k$-subspaces of $\overline{V}$, all such subspaces of $\overline{V}$ arise in this way. Thus by Theorem~\ref{T3}, $k=2$ and $H$ is one of the groups in  Table~\ref{T:Spc}.  In particular, $\cU=\cN_2^+$ or $\cN_2^-$. Let $W=U+V^\perp$, a nondegenerate $3$-subspace of $V$.  Now $H$ is transitive on $\cN_2^\eps$ if and only if $H_{W}$ is transitive on the set of $2$-subspaces of $W$ contained in  $\cN_2^\eps$.   
For $H$ as in the first and last columns of Table~\ref{T:Spc}, it follows from \cite{Factns}*{3.2.1(a), p.\,47} that  $\Hbar_{\overline{U}}^{\overline{U}}\rhd \Sp_2(q)$ and so $H_W^W\lhdeq\GO_3(q)$. By Witt's Lemma,  $\GO_3(q)$ is transitive on the set of nondegenerate 2-subspaces of $W$ of each isometry type, as required.
 We note that in the last column, the $\Sp_2(q^2)$ is occurring as an extension field group and not as $\Omega_4^-(q)$ as it is not possible for the stabiliser of a nondegenerate hyperplane to be transitive on $\cN_2^\eps$.

For the remaining columns of Table~\ref{T:Spc}, $\soc(H)=\G_2(q^{m/3})'$. By~\cite{Cooperstein}*{p.\,27 and Lemma~5.4}, $\G_2(q^{m/3})$ acts transitively on the set of nondegenerate 2-subspaces of a 6-dimensional space over $\F_{q^{m/3}}$ equipped with a nondegenerate alternating form such that the stabiliser $K$ of such a subspace $X$ is the maximal subgroup $\SL_2(q^{m/3})\times\SL_2(q^{m/3})$. Moreover, by 
\cite{rW}*{Section~4.3.6} we see that $K$ induces $\SL_2(q^{m/3})\cong\GO_3(q^{m/3})$ on the nondegenerate 3-subspace corresponding to $X$ of the 7-dimensional vector space over $\F_{q^{m/3}}$ equipped with a nondegenerate quadratic form.   Thus, if  $X=\langle e_1,f_1\rangle$, then the stabiliser in $K$ of the $\F_{q}$-span $Y$ of $\{e_1,f_1\}$ induces $\SL_2(q)\cong\Sp_2(q)$ on $Y$. Hence  $\GO_3(q)\lhdeq H_W^W$ acts transitively on the set of subspaces in $\cN_2^\eps$ contained in $W$ and so each group in Table~\ref{T:Spc} that is transitive on the $\cN_2$ subspaces of $\Vbar$ is transitive on $\cN_2^\eps$, proving part (d).

\smallskip
{\sc Case}~$\cU=\cN_{2m}^\eps$ for some $\eps\in\{+,-\}$. Here $U$ is a nondegenerate $2m$-subspace and $(U+V^\perp)/V^\perp=\overline
{V}$.  By~\cite{Pollatsek1971}*{Theorem~1.5}, $\Sp_{2m}(q)_U=\GO_{2m}^\eps(q)$ is maximal in $\Sp_{2m}(q)$  and so $G_U$ is maximal in $G$. We assume the same setup as in the proof of Theorem~\ref{T3},
and in particular that $H\leqslant B$ where $\overline{B}$ is a core-free subgroup of $\PGammaSp_{2m}(q)$, and that $G^*=\overline{B}(G^*)_U$ is a core-free maximal factorisation, where $G^*$ is an almost simple group with socle $L=\PSp_{2m}(q)$. Moreover, $B=N_G(B\cap \Omega_{2m+1}(q))$.

Suppose first that $\cU=\cN_{2m}^{-}$. Here, when $(m,q)\neq (2,2)$,  the transitive
subgroups of $G$ not containing $\Sp_{2m}(q)$ were computed
in~\cite{Reg}*{Lemmas~4.6} with reference to~\cite{Reg}*{Lemma~4.5}.  We note that however, the reference to \cite{Reg}*{Lemma~4.5} forgot to include the possibility that $\Omega_{2m}^+(q)\lhdeq H$ when $q=2$ or 4.
When $(m,q)=(2,2)$ we have $|\cN_4^-|=6$ and $\GO_5(2)\cong\Sym_6$. This gives part (e).

Finally, suppose that $\cU=\cN_{2m}^{+}$. 
Suppose first that $(m,q)=(2,2)$. Then a \textsc{Magma} computation shows that either $\Alt_6=\Omega_5(2)$ or $\Alt_5$ are normal subgroups of $H$, or $H=F_{20}$. Moreover, both classes of $\Alt_5$-subgroups in $G$ are transitive. The two $\Alt_5$ subgroups correspond to $\Omega_4^-(2)$ and $\Omega_3(4)$, each of which appears in Table~\ref{T:Spd}, while there is a unique conjugacy classes  of $F_{20}$ subgroups and so these must be normalising $\Omega_2^-(2^2)\cong \C_5$ as in the third column of Table~\ref{T:Spd}. 
Thus we may assume that $(m,q)\ne(2,2)$, and then  
it follows from~\cite{Factns}*{p.\,10, p.\,13} that we have the five cases
in Table~\ref{T:SpDsubcases}.  For reasons which will become clear later we deal with (f1) last.
\begin{table}[!ht]
\caption{Subcases when $L=\PSp_{2m}(q)$ and $\cU=\cN_{2m}^{+}$.}\label{T:SpDsubcases}
\begin{tabular}{rlccccccc}
\toprule
\textup{Cases} && (f1)& (f2)& (f3)& (f4)& (f5)\\
$B\cap \Sp_{2m}(q)$ && $\Sp_{2a}(q^b)\ldotp b$& $\GO^{-}_{2m}(q)$& $\Sz(q)$& $\G_2(q)$& $\PSL_2(17)$\\
$m$&&$ab>a$&$\ge2$&$2$&$3$&$4$\\
$q$&&${\rm even}$&$2,4$&$2^\textup{even}\ge8$&${\rm even}$&$2$\\
\bottomrule
\end{tabular}
\end{table}

(f2)~Here $q=2$ or $4$ and $B=N_G(\Omega_{2m}^-(q))$.  We record $\Omega_{2m}^-(q)\lhdeq H$ in Table~\ref{T:Spd}.  Suppose now that $\Omega_{2m}^-(q)\not\leqslant H$. Then we have a factorisation $B=HB_U$. By~\cite{Factns}*{3.2.4(e), pp.\;51--52} $(\GO_{2m}^-(q))_U$ is contained 
in the stabiliser in $\GO^-_{2m}(q)$ of a nonsingular 1-subspace of the natural $(2m)$-dimensional module for $\GO^-_{2m}(q)$. Thus $H$ is given by Theorem~\ref{T9}. If $m\geq 4$ then either $\SU_m(q)\lhdeq H$ or $H$ is given in Table~\ref{T:OEvenN1},  and hence appears in Table~\ref{T:Spd}.  For $m=3$, Theorem~\ref{T:KleinSU4} implies that either $\SU_3(q)\lhdeq H$ or $q=2$ and  $3_+^{1+2}\lhd H$. Finally, if $m=2$ then note that we have already considered the case $q=2$, while for $q=4$ a \textsc{Magma} calculation shows that no examples exist.

(f3)~In this case $m=2$ with $q\geq 8$ and $B=N_G(\Sz(q))$. By~\cite{Factns}*{Table~2}, $\Sz(q)$ is transitive on $\cU$ and so $\Sz(q)\lhdeq H$ is listed in Table~\ref{T:Spd}. Suppose now that $\Sz(q)\not\leqslant H$. Then we have a factorisation $B=HB_U$. However, \cite{Factns}*{Table~5} implies that no such factorisation exists.

(f4)~In this case $m=3$ and $B=N_G(\G_2(q))$. By~\cite{Factns}*{Table~2}, $\G_2(q)$ is transitive on $\cU$ and so $\G_2(q)'\lhdeq H$ is listed in Table~\ref{T:Spd}. Moreover, a \textsc{Magma} computation shows that when $q=2$, the group $\G_2(2)'\cong\PSU_3(3)$ is also transitive. Suppose now that $\G_2(q)'\not\leqslant H$. Then we have a factorisation $B=HB_U$. By~\cite{Factns}*{(5.2.3b)\;p.\,111} we have that $(\G_2(q))_U=\SL_3(q).2$. However, by \cite{Factns}*{Table~5} there are no such factorisations of an almost simple group with socle $\G_2(q)$ for $q$ even.

(f5)~In this case $G=\GO_9(2)$ and $B=\PSL_2(17)$. We record $\PSL_2(17)$ in Table~\ref{T:Spd} and now suppose that $H<B$.
Now $|\cN_8^{+}|=136$ and as there is a single conjugacy class of subgroups of $B=\PSL_2(17)$ of
index $136$ namely dihedral groups $D_{18}$, we have
$B_U=D_{18}$. Hence we
have a factorisation $B=\PSL_2(17)=H B_U$. Note that $D_{18}$ is the subgroup of $\PSL_2(17)$ arising from $\GammaL_1(17^2)$. However, by \cite{Factns} no such   factorisation exists (though one does exist of $\PGL_2(17)$).

(f1)~Here $B= N_G(\Sp_{2a}(q^b))$. It is shown in \cite{Factns}*{3.2.1(d)} that $(\Sp_{2a}(q^b).b)\cap \GO^+_{2m}(q) = \GO_{2a}(q^b).b$ 
(and the argument holds for all $b$). It follows that the subgroup $\Sp_{2a}(q^b)$ is also transitive on $\cN_{2m}^+$, and hence  if 
 $\Sp_{2a}(q^b)\lhdeq H$ then $H$ is transitive on $\cU$ and we have listed this in Table~\ref{T:Spd}. Suppose now that $\Sp_{2a}(q^b)\not\leqslant H$. Then we have a factorisation $B=HB_U$. As just noted, $(\Sp_{2a}(q^b))_U=\GO^+_{2a}(q^b)$ and so the result follows by induction on $m$ and noting that $q^b\neq 2$. 
\end{proof}

\section{Orthogonal groups  in even dimension of plus type}\label{S:OPlus}

Let $V=(\F_q)^{2m}$, with $m\geq4$, equipped with a nondegenerate quadratic form $Q$ of plus type and associated bilinear form $(\dot,\dot)$. We let $\{e_1,\ldots,e_m,f_1,\ldots,f_m\}$ be a basis for $V$ such that $Q(e_i)=Q(f_i)=0$ and $B(e_i,f_j)=\delta_{ij}$ for all $i,j$.

In Subsection~\ref{SS:OPlus}, we find all subgroups of $\GammaO_{2m}^+(q)$ that are transitive on a set $\cU$ given by Table~\ref{T-defn}, namely $\cP_k$ for $k< m$, $\cP_m^\eps$ for  $\eps\in\{+,-\}$, $\cN_k^\eps$ for even $k\leq m$ and  $\eps\in\{+,-\}$, and if   $q$ is odd also  $\cN_k^\eps$ for odd $k\leq m$ and  $\eps\in\{+,-\}$, and if $q$ is even also $\cN_1$. The union $\cP_m=\cP_m^+\cup\cP_m^-$ is a $\GammaO_{2m}^+(q)$-orbit, and also  if $k$ and $q$ are both odd then   $\cN_k=\cN_k^+\cup\cN_k^-$ is a $\GammaO_{2m}^+(q)$-orbit. Subgroups acting transitively on these latter sets are determined in Subsection~\ref{SS:OPlus2}, namely in Proposition~\ref{P:Both} for $\cP_m$ and Proposition~\ref{P2} for $\cN_k$.    The eight-dimensional case needs special attention and we do some preparatory work in Subsection~\ref{SS:O8}.

\begin{remark}\label{R:errata}
 The subgroups of $\GammaO_{2m}^+(q)$ that are transitive on $\cN_1$ when $q$ is even, or on $\cN_1^\eps$, where $\eps\in\{+,-\}$, when $q$ is odd are determined in~\cite{Reg}*{Lemma~4.5}. However,  some of the groups in~\cite{Reg}*{Lemma~4.5} are given as subgroups of $\PGammaO_{2m}^+(q)$ instead of $\GammaO_{2m}^+(q)$, and in addition $\Alt_6$ is incorrectly listed when $(m,q)=(4,3)$ (see the proof of Lemma~\ref{lem:m4q3}~(e)). We therefore give a revised version of this classification as part (e)  of Theorem~\ref{T10} (see, in particular, Table~\ref{T:OPlusN1}). See also Remark \ref{rem:Reg}(e).
\end{remark}

The formulas for $|\cN_1^\eps|, |\cP_1|,|\cP_m^\eps|$, and $|\cN_2^{\eps}|$ (see~\cite{Factns}*{pp.\,62,\,63,\,67,\,70}) are:
\begin{align}
  |\cN_1|&=q^{m-1}(q^m-1) \quad \textrm{ for $q$ even}, &  |\cN_1^\eps| =\frac{q^{m-1}(q^m-1)}{2} \quad \textrm{ for $q$ odd,}\\
  |\cP_1|&=\frac{q^{2m}-1}{q-1}-q^{m-1}(q^m-1),\label{E:OPlus1}\\ 
  |\cP_m^\eps|&=\prod_{i=1}^{m-1}(q^i+1)=|\cP_{m-1}|, &|\cN_2^{\eps}|=\frac{q^{2m-2}(q^m-1)(q^{m-1}+\eps 1)}{2(q-\eps 1)}.\label{E:OPlus2}
\end{align}
The formula for $|\cP_1|$ factors to give $|\cP_1|=\frac{(q^m-1)(q^{m-1}+1)}{q-1}$. When $m=4$, we see that $|\cP_1|=(q+1)(q^2+1)(q^3+1)=|\cP_4^{\cmt{$\eps$}}|$.

\subsection{Eight-dimensional orthogonal groups of  plus type}\label{SS:O8}

The groups $\POmega_8^+(q)$ have more outer automorphisms than their counterparts in higher dimensions and this affects both their subgroup structure and actions on subspace families. We first make some technical comments, in particular concerning our notation.

\begin{remark}\label{R:tri}
When $L=\POmega_8^{+}(q)$ we have  $|\Aut(L):\PGammaO_8^+(q)|=3$. Let $\tau$ be an outer automorphism of $L$ of order~$3$ that is not in $\PGammaO_8^+(q)$. It is well-known that $\tau$ induces a  3-cycle $(\mathcal{P}_1, \mathcal{P}_4^+, \mathcal{P}_4^-)$.  We note that $\tau$ does not induce an automorphism of the matrix group $\Omega_8^+(q)$ when $q$ is odd. Indeed the stabiliser in $\Omega_8^+(q)$  of a totally singular 1-subspace is not isomorphic to the stabiliser of a totally singular 4-subspace.

There are two types of subgroups $\Omega_7(q)$ of $L=\POmega_8^{+}(q)$ which
we denote by $\Omega_7(q)^R$ and $\Omega_7(q)^I$, where $\Omega_7(q)^R$ acts reducibly on $V$ fixing a nonsingular 1-subspace,
and $\Omega_7(q)^I$ acts irreducibly on~$V$. If $q$ is odd, there are \emph{four} $L$-conjugacy classes of $\Omega_7(q)^I$ and \emph{two} of $\Omega_7(q)^R$; we note that $\Omega_7(q)^R$ has preimage $2\times \Omega_7(q)$ in $\Omega_8^+(q)$ while $\Omega_7(q)^I$ has preimage $\nonsplit{2}{\Omega_7(q)}$ \cite{BHRD}*{Table~8.50}. When $q$ is even, there are \emph{two} $L$-conjugacy classes of subgroups $\Omega_7(q)^I$ and \emph{one} of $\Omega_7(q)^R$; and in this case each preimage is isomorphic to $\Omega_7(q)\cong\Sp_6(q)$.  Furthermore,  $\tau$ induces 3-cycles    $(\Omega_7(q)^I,\Omega_7(q)^I,\Omega_7(q)^R)$ on these classes, for all values of $q$. We use  ${}^\wedge\Omega_7(q)^I$ to denote the minimal preimage of $\Omega_7(q)^I$ in $\Omega_8^+(q)$. Hence ${}^\wedge\Omega_7(q)^I=\nonsplit{2}{\Omega_7(q)}$ when $q$ is odd,  and is equal to $\Omega_7(q)^I$ when $q$ is even.

Let $M_2$ denote an imprimitive subgroup of $L$ preserving a decomposition $V=E\oplus F$ into totally singular 4-subspaces
$E$ and $F$. There are two $L$-conjugacy classes of subgroups $M_2$ and one $L$-conjugacy class of subgroups $N_2^{+}$, and $\tau$ induces a 3-cycle $(M_2,M_2,N_2^{+})$ on these classes. 
Finally, we note that $\Omega_8^+(q)$ has two conjugacy classes of subgroups isomorphic to $\SU_4(q)$. Letting $K_1$ and $K_2$ be the images in $L$ of two such non-conjugate subgroups, the outer automorphism $\tau$ induces a 3-cycle $(N_L(K_1),N_L(K_2),N_2^-)$.
These observations are important in the proof of Theorem~\ref{T10} when considering actions on $\cP_1$, $\cN_2^+$ and $\cN_2^-$.
\end{remark}

Arguments for the group $\POmega_8^+(3)$ are so delicate that we have separated them out in the following lemma.

 \begin{lemma}\label{lem:m4q3}
 Let $H\leqslant \GammaO^+_8(3)$ with $\Omega_8^+(3)\not\leqslant H$.
 \begin{enumerate}[{\rm (1)}]
 \item If $H$  acts transitively on a set $\cU$ of subspaces given in Table~$\ref{T-defn}$ then $H$ is listed in Theorem~$\ref{T10}$.
 \item If $H$ acts transitively on $\cP_4$ then $H$ is given in Proposition~$\ref{P:Both}$.
 \item If $H$ acts transitively on $\cN_1$ then $H$ is given in Proposition~$\ref{P2}$.
 \end{enumerate}
 \end{lemma}

\begin{proof}
(1) Suppose first that $\cU$ is a set of subspaces given in Table~$\ref{T-defn}$. Let  $G$ be the stabiliser in $\GammaO^+_8(3)=\CO^+_8(3)$ of $\cU$ and let $H\leqslant G$ act transitively on $\cU$ such that $\Omega_8^+(3)\not\leqslant H$. Then $G=\CO^+_8(3)$ unless $\cU=\cP_4^+,\cP_4^-,\cN_1^\eps$ or $\cN_3^\eps$. In the first two of the exceptional cases $G$ is an index 2 subgroup of $\CO_8^+(3)$ containing $\SO_8^+(3)$ while in the last two cases we have $G=\GO_8^+(3)$.   For $X\leqslant \CO^+_8(3)$ denote the image of $X$ in $\PCO^+_8(3)$ by $\overline{X}$. We now deal with each possibility for $\cU$ from Table~$\ref{T-defn}$, in most cases using \textsc{Magma}. 

(a) First consider $\cU=\cN_4^\eps$, where $\eps\in\{+,-\}$, and $G=\CO^+_8(3)$. Let $L=\Omega_8^+(3) $ and note that $L$ is transitive on $\cU$. Now  by \cite{BG}*{Table~4.1.2}, $|\cU|=\frac12 3^8(3^2+1)^2(3^4+3^2+1)$ if $\eps=+$ and  $|\cU|=\frac12 3^8(3^2-1)(3^6-1)$ if $\eps=-$. In either case $|\cU|>3^{15}$, and since $H$ is transitive on $\cU$ we have  
$|H|>3^{15}$.  Also (see \cite{KL}*{Table~2.1C}) $|G:L|=8$ and  $|\POmega_8^+(3)|=\frac14 3^{12}(3^6-1)(3^4-1)^2(3^2-1)
<\frac14 .3^{28}<3^{29}$, and it follows that $|L|<3^{30}$ and  $|H\cap L| \geq \frac{1}{8} |H|> 3^{13}>  |L|^{1/3}$. Thus $H\cap L$ is one of the groups from \cite{AB}*{Table~7 or Proposition~4.23}.  Moreover for both values of $\eps$ the order $|H\cap L|$ is divisible by $3^4+3^2+1 = 7\times 13$, and  the only one of these subgroups with this property is  $H\cap L\cong \Omega_7(3)$.  In particular $H\cap L$ is maximal in $L$ and since, for $U\in\cU$,  $L$ is transitive on $\cU$ and its overgroup $L_{\{U,U^\perp\}}$ is maximal in $L$, we have a maximal core-free factorisation $L= (H\cap L)( L_{\{U,U^\perp\} })$.  However there are no such factorisations by \cite{Factns}*{Table~4}.

(b)  Next $\cU=\cN_3^\eps$ and $G=\GO^+_8(3)$. Note that $\cN_3^+$ and $\cN_3^-$ are fused in $\CO_8^+(3)$. By~\textsc{Magma} and Lemma~\ref{lem:derived}, we deduce that $\nonsplit{2}{\Omega_7(3)}\lhdeq H$, which appears in Theorem~\ref{T10}~(g).

(c) Next consider $\cU=\cN_2^+$ and $G=\CO^+_8(3)$.  A \textsc{Magma} calculation and Lemma~\ref{lem:derived} shows that $\nonsplit{2}{\Omega_7(3)}\lhdeq H$, which appears in Table~\ref{T:OPlusd}. 

(d) The next  case  is $\cU=\cN_2^-$ and $G=\CO^+_8(3)$.  By~\textsc{Magma} and Lemma~\ref{lem:derived}, we deduce that either $\nonsplit{2}{\Omega_7(3)}\lhdeq H$ (which appears in Table~\ref{T:OPlusd}), or $3^6\rtimes \PSL_4(3)\lhdeq \oH$ and $H$ is contained in $P_4$. We note that since $\SL_4(3)$ is quasisimple, Lemma~\ref{lem:derived} applied to the subgroup $\PSL_4(3)$ of $\oH$ implies that $3^6\rtimes \SL_4(3)\lhdeq H$, as in Theorem~\ref{T10}~(f).

(e) Next consider $\cU=\cN_1^\eps$ and $G=\GO^+_8(3)$.  Note that $\cN_1^+$ and $\cN_1^-$ are fused in $\CO_8^+(3)$.  The possibilities for $H$ are given by  \cite{Reg}*{Lemma~4.5} but there only the projective versions are listed.  Combining  \cite{Reg}*{Lemma~4.5} and a \textsc{Magma} calculation we obtain the following list of examples:
\begin{enumerate}[(i)]
\item $H\rhdeq \SL_4(3)$, $\SU_4(3)$ or $\Sp_4(3)$,
\item $H\rhdeq \nonsplit{2}{\Omega_7(3)}$, $\nonsplit{2}{\Omega_8^+(2)}$,  $\nonsplit{2}{\Sp_6(2)}$, $\nonsplit{2}{\SU_4(2)}$, $\nonsplit{2}{\Alt_9}$,
    \item $H\leqslant P_4$ and modulo the unipotent radical $H$ induces a subgroup of $\GL_4(3)$ which is transitive on 1-subspaces.
\end{enumerate}
The subgroups in parts (i) and (ii) appear in Table~\ref{T:OPlusN1}, and those in part (iii) are as in Theorem~\ref{T10}~(e). Note that in (ii), $\nonsplit{2}{\Omega_7(3)}$ is the group 
${}^\wedge\Omega_7(3)^I$, and  $\PSU_4(2)\cong\PSp_4(3)$ but this occurs differently to the group $\Sp_4(3)$ in (i). Also no groups with normal subgroup $\Alt_6$, as listed in \cite{Reg}*{Lemma~4.5(vii)}  for $m=4, q=3$, act transitively on $\cN_1^\eps$.

(f) Here $G=\CO_8^+(3)$ acts on $|\cP_1|=1120$ singular $1$-subspaces. A  {\sc Magma}  calculation shows that $H$ contains one of the following as a normal subgroup $\nonsplit{2}\Omega_7(3)$, $  \nonsplit{2}{\Omega^+_8(2)}$, $\nonsplit{2}{\Alt_8}$, $\nonsplit{2}{\Alt_9}$, $\nonsplit{2}{\Omega_7(2)}$,  $\nonsplit{4}{\PSL_3(4)}$, $\SU_4(3)$, $2^{1+6}_+.\Alt_7$, or $2^{1+6}_+.\Alt_8$. These all appear in Table~\ref{T:OPlusa}.

(g)  When $\cU=\cP_2$, a \textsc{Magma} calculation reveals that there are no possibilities for $H$.

(h) When $\cU=\cP_3$ we have  $G=\CO_8^+(3)$ and a \textsc{Magma} calculation shows that $H$ contains $\nonsplit{2}{\Omega_8^+(2)}$ as a normal subgroup, as in Theorem~\ref{T10}~(c).

(i)  Finally suppose that $\cU=\cP_4^\eps$ and $G$ is the index two subgroup of $\CO^+_8(3)$ that fixes each $\Omega_8^+(3)$-orbit on totally singular 4-subspaces. Note that if $H$ is transitive on $\cP_1$ then $\oH^\tau$ is transitive on $\cP_4^+$ or $\cP_4^-$ where $\tau$ is a suitable triality automorphism. The $\CO_8^+(3)$-conjugacy class consisting of $\nonsplit{2}\Omega_7(3)$ subgroups that are transitive on $\cP_1$ splits into two conjugacy classes under $G$. Applying triality to their images in $\oG$ and pulling back to $G$ we find that both $\Omega_7(3)^R\lhdeq H$ and $\nonsplit{2}\Omega_7(3)\lhdeq H$ are transitive on $\cP_4^\eps$ for $\eps=\pm$  (and both lie in subgroup families in Table~\ref{T:OPlusb}). Similarly, the  $\textup{CO}_8^+(3)$-conjugacy class of $\textup{SU}_4(3)$ subgroups that are transitive on $\cP_1$ splits into two conjugacy classes under $G$. Applying triality to their images in $\oG$ and pulling back to $G$ we find that both $\SU_4(3)\lhdeq H$ and $\Omega_6^-(3)^R\lhdeq H\leqslant N_2^-$,  are transitive on $\cP_4^\eps$ for $\eps=\pm$  (and both lie in subgroup families  in Table~\ref{T:OPlusb}). 
Also, using \cite{BHRD}*{Table~8.50} and \textsc{Magma} we see that all of $2^7:\Alt_8$ and $2^{1+6}_+.\Alt_8$, and $2^6:\Alt_7$ and $2^{1+6}_+.\Alt_7$ are transitive on $\cP_4^\eps$, for $\eps=\pm$. Finally, a {\sc Magma}~\cite{Magma} computation yields the following remaining examples for a normal subgroup of $H$: $\nonsplit{2}{\Omega^+_8(2)}$, $\Omega_7(2)$, $\nonsplit{2}{\Omega_7(2)}$ $\nonsplit{4}{\PSL_3(4)}$, $\nonsplit{2}{\PSL_3(4)}$, $\nonsplit{2}{\Alt_8}$, $\Alt_8$, $\nonsplit{2}{\Alt_9}$ and $\Alt_9$. All of these groups are in Table~\ref{T:OPlusb}.

\smallskip
(2)\  Next suppose that $\cU=\cP_4$. Then a \textsc{Magma} calculation shows that $H$ satisfies one of the following, and these are the groups occurring in Table~\ref{T:OPlusBoth} for Proposition~\ref{P:Both}  when $(m,q)=(4,3)$.
\begin{enumerate}[(i)]
\item $H\rhdeq \Omega_7(3)^R$, $\Omega_6^-(3)^R$;
\item  $H=\nonsplit{2}{\GO_7(2)}$, $\nonsplit{2}{\GO^+(8,2)}$;
\item $H\rhdeq \Alt_8$, $\Alt_9$, $\nonsplit{2}{\PSL_3(4)}$, $\nonsplit{4}{\PSL_3(4)}$
\item $H\rhdeq 2^7:\Alt_8$,  $2^6:\Alt_7$, 
\end{enumerate}

(3)\  Finally, suppose that $\cU=\cN_1$ and that $H$ is transitive on $\cU$. Then $H$ has an index two subgroup $H^+$ that is transitive on $\cN_1^+$ and $\cN_1^-$, and so $H^+$ has already been determined in part (1). A \textsc{Magma} calculation then shows that one of the following holds, and so $H$ is as given by Proposition~\ref{P2}.
\begin{enumerate}[(i)]
\item $H\rhdeq \SL_4(3)$, $\SU_4(3)$ or $\Sp_4(3)$;
\item $H= \nonsplit{2}{\SO_7(3)}$;
    \item $H\leqslant P_4$, $H$ contains an element of $\CO_8^+(3)\backslash\GO_8^+(3)$, and  modulo the unipotent radical of $P_m$,  $H\cap\GO_8^+(8,3)$  induces a subgroup of $\GL_4(3)$ that is transitive on $1$-subspaces. \qedhere
\end{enumerate}
\end{proof}

\subsection{Actions on subspace families from Table~\texorpdfstring{\ref{T-defn}}{}}\label{SS:OPlus}

We now consider all the subspace families $\cU$ from Table~\ref{T-defn}, for dimensions $2m\geq 6$.

\begin{theorem}\label{T10}
  Suppose that  $H\le\GammaO_{2m}^{+}(q)$, where $m\geq 3$, such that $\Omega_{2m}^{+}(q)\not\le H$. Let 
  $V=(\F_q)^{2m}$ and suppose that $H$ acts transitively on a set $\cU$  
  of subspaces of $V$ given by Table~$\textup{\ref{T-defn}}$. 
Then one of the following holds:
  \begin{enumerate}[{\rm (a)}]
\item $\cU=\cP_1$ and $H_0\lhdeq H$ where $(H_0,m,q)$ is given in
    Table~$\ref{T:OPlusa}$; or
\begin{table}[!ht]
\caption{(a) Subgroups $H_0\lhdeq H\le\GammaO^+_{2m}(q)$ where $H$ is transitive
on $\cP_1$.}\label{T:OPlusa}
\begin{tabular}{rlccccccc}
\toprule
$H_0$ &&$\Alt_7$&$\Alt_9$&$\nonsplit{2}{\Alt_8}$&$\nonsplit{2}{\Alt_9}$&$2^{1+6}_+\ldotp\Alt_7$&$2^{1+6}_+\ldotp\Alt_8$&\\
\textup{$m$}&&$3$&$4$&$4$&$4$&$4$&$4$&\\
\textup{$q$}&&$2$&$2$&$3$&$3$&$3$&$3$&\\
\midrule
$H_0$ &&$\nonsplit{4}{\PSL_3(4)}$&$\nonsplit{2}{\Omega_7(2)}$&$\nonsplit{2}{\Omega_8^+(2)}$&${}^\wedge\Omega_7(q)^I$&$\SU_m(q)$ \\
\textup{$m$}&&$4$&$4$&$4$&$4$&\textup{even}\\
\textup{$q$}&&$3$&$3$&$3$&\textup{all}&\textup{all}\\
\bottomrule
\end{tabular}
\end{table}

    \item $\cU=\cP_2$, $m=3$, $q=2$ and $\Alt_7\lhdeq H$; or

    \item $\cU=\cP_3$, $m=4$, $q=3$ and $\nonsplit{2}{\Omega_8^+(2)}\lhdeq H$; or

\item $\cU=\cP_m^\eps$, and either $m=3$ and $H\leqslant N_{A}(\Omega_2^-(q^2))\leqslant N_4^-$ where $A$ is the stabiliser in $\GammaO_6^+(q)$ of $\cP_3^+$, or $H_0\lhdeq H$ where $(H_0,m,q)$ is given in
  Table~$\ref{T:OPlusb}$; or
\begin{table}[!ht]
  \caption{(d) Subgroups $H_0\lhdeq H\le\GammaO^+_{2m}(q)$
    where $H$ is transitive on $\cP_m^+$ or $\cP_m^-$ (excluding the case $m=3$ and  $H\leqslant N_A(\Omega_2^-(q^2))\leqslant N_4^-$.)}\label{T:OPlusb}
\begin{tabular}{rlcccccccc}
\toprule
$H_0$ &&$\Omega_{2m-1}(q)^R$ &$\Omega^{-}_{2m-2}(q)^R$  &${}^\wedge\Omega_7(q)^I$&$\SU_4(q)$&$2^4\colon \C_5$&$2^4\ldotp\Alt_5$&$\Alt_5$&$\Alt_7$\\

\textup{$m$}&& $\ge3$&$\ge3$&$4$&$4$&$3$&$3$&$3$&$3$\\ 
\textup{$q$}&& \textup{all}& \textup{all}&\textup{all}&\textup{all}&$3$&$3$&$3$&$2$\\
\midrule
$H_0$ &&$2^{1+6}_+\kern-2pt\ldotp\kern-2pt\Alt_7$& $2^6:\Alt_7$& $2^{1+6}_+\kern-2pt\ldotp\kern-2pt\Alt_8$&$2^7:\Alt_8$&$\nonsplit{2}{\Omega_8^+(2)}$&$\nonsplit{2}{\Omega_7(2)}$&$\Omega_7(2)$&\\
\textup{$m$}&&$4$&$4$&$4$&$4$&$4$&$4$&$4$&\\
\textup{$q$}&&$3$&$3$&$3$&$3$&$3$&$3$&$3$&\\
\midrule 
$H_0$ &&$\nonsplit{4}{\PSL_3(4)}$&$\nonsplit{2}{\PSL_3(4)}$&$\nonsplit{2}{\Alt_8}$&$\Alt_8$&$\nonsplit{2}{\Alt_9}$&$\Alt_9$&$\Alt_9$\\
\textup{$m$}&&$4$&$4$&$4$&$4$&$4$&$4$&$4$\\
\textup{$q$}&&$3$&$3$&$3$&$3$&$3$&$3$&$2$\\
\bottomrule
\end{tabular}
\end{table}

\item $\cU=\cN_1$ (with $q$ even) or $\cU=\cN_1^\varepsilon$ (with $q$ odd),  and either $H_0\lhdeq H$ where $(H_0,m,q)$ is listed in Table~$\ref{T:OPlusN1}$, or $H\leq P_m$ such that, modulo the unipotent radical of $P_m$, $H$ induces a subgroup of $\GammaL_m(q)$ that is transitive on $1$-subspaces; or
\begin{table}[!ht]
\caption{(e) Subgroups $H_0\lhdeq H\le\GammaO_{2m}^{+}(q)$, where $H$ is transitive on~$\cU=\cN_1$ with $q$ even, or on $\cU=\cN_1^\varepsilon$ with $q$ odd. }\label{T:OPlusN1}
\begin{tabular}{rlcccccccc}
\toprule
$H_0$ &&$\SL_m(q)$ &$\SL_{m/2}(q^2)$ & $\SL_{m/4}(q^4)$ & $\Sp_m(q)$ & $\Sp_{m/2}(q^2)$ &$\Sp_{m/4}(q^4)$\\
\textup{$m$}&& \textup{all}& \textup{even} & \textup{$\geq 8$ even} & \textup{even}& \textup{even} & \textup{$\geq 8$ even} &\\
\textup{$q$}&& \textup{all} & $2,4$ &$2$&\textup{all} & $2,4$ &$2$&\\
\midrule
$H_0$ &&$\SU_m(q)$ &$\SU_{m/2}(q^2)$ & $\SU_{m/4}(q^4)$ & $\Omega_m^+(q^2)$ &$\Omega_{m/2}^+(q^4)$\\
\textup{$m$}&& \textup{even}& \textup{even} &$\geq 8$ \textup{even} & \textup{even} & $m/2\geq 2$ \textup{even}\\
\textup{$q$}&& \textup{all} & $2,4$ &$2$& $2,4$ &$2$\\
\midrule
$H_0$ && $\G_2(q)'$ & $\G_2(q^2)$ &$\G_2(q^4)$ & $\SL_2(q^2)$ &$\SL_2(q^4)$ & $\SL_2(q^8)$\\
\textup{$m$}&& $6$ &$12$ & $24$ &$4$ &$8$ &$16$\\
\textup{$q$}&&\textup{even} & $2,4$& $2$ &\textup{even} & $2,4$ &$2$\\
\midrule
$H_0$ && ${}^\wedge\Omega_7(q)^I$ &$\Omega_7(q^2)$ &$\Omega_7(q^4)$&$\nonsplit{2}{\Omega_8^-(q^{1/2})}$ &$\Omega^-_8(q)$&$\Omega^-_8(q^2)$\\
\textup{$m$}&& $4$ &$8$&$16$&$4$ & $8$ & $16$\\
\textup{$q$}&& \textup{all} &$2,4$ &$2$ &\textup{all} &$2,4$ &$2$ \\
\midrule
$H_0$ && $\Omega_9(q)$ &$\Omega_9(q^2)$ &$\Omega_9(q^4)$ & $\Alt_6$ &$\Alt_7$ &$\Alt_9$\\
\textup{$m$}&& $8$ & $16$ &$32$ &$4$&$4$&$4$\\
\textup{$q$}&&  \textup{all} &$2,4$ &$2$&$2$&$2$&$2$\\
\midrule
$H_0$ &&$\nonsplit{2}{\Omega_8^+(2)}$& $\nonsplit{2}{\Sp_6(2)}$&$\nonsplit{2}{\SU_4(2)}$& $\nonsplit{2}{\Alt_9}$&$\PSU_4(3)$ & $\nonsplit{3}{M_{22}}$\\
\textup{$m$}&& $4$ & $4$ &$4$ &$4$&$6$&$6$\\
\textup{$q$}&& $3$& $3$& $3$& $3$& $2$& $2$&\\
\midrule
$H_0$ && $\nonsplit{3}{\textup{Suz}}$ &$\textup{Co}_1$\\
\textup{$m$}&& $12$ & $12$ \\
\textup{$q$}&& $2$ &$2$\\
\bottomrule
\end{tabular}
\end{table}

\item $\cU=\cN^\pm_2$ and $H_0\lhdeq H$ where $(H_0,m,q)$ is given in
    Table~\textup{\ref{T:OPlusd}}, or  $\cU=\cN_2^-$ and $H\leq P_m$ such that, modulo the unipotent radical of $P_m$ either  $H$ induces at least $\SL_m(q)$ or $m=3$ and $H$ induces a subgroup that is transitive on $1$-subspaces; or
\begin{table}[!ht]
\begin{threeparttable}
\caption{(f) Subgroups $H_0\lhdeq H\le\GammaO_{2m}^{+}(q)$ where $H$    is transitive on~$\cU=\cN_2^\pm$. In the first case
  $H$ equals $[2^{10}]\ldotp31\ldotp5$.}\label{T:OPlusd}
\begin{tabular}{rlccccccc}
\toprule
$\cU$&&$\cN_2^-$&$\cN_2^-$&$\cN_2^+$&$\cN_2^\eps$& $\cN_2^\eps$ &$\cN_2^-$&$\cN_2^-$\\
$H_0$ &&$[2^{10}]\ldotp31\ldotp5$&$\SL_m(q)$ &$\SU_m(q)$&${}^\wedge\Omega_7(q)^I$& $\Alt_9$ & $\Alt_8$\tnote{*} & $\GO_8^-(q^{1/2})$\tnote{**}\\
\textup{$(m,q)$}&&$(5,2)$&$(\ge3,2\;{\rm or }\;4)$&$(\textup{even},2 \textup{ or }4)$&$(4,q)$& $(4,2)$& $(4,2)$ & $(4,4 \textup{ or } 16)$\\
\bottomrule
\end{tabular}
\begin{tablenotes}\footnotesize
\item[*] Even though $\Alt_8\cong\SL_4(2)$, this group is a different subgroup than the one appearing in the second column when $(m,q)=(4,2)$. Indeed this group is self-normalising and contained in an $\Alt_9$ while the $\SL_4(2)$ in the second column is normalised by a group of twice its order.
\item[**] This is the image under triality of a subfield subgroup and existence is only confirmed for $q=4$. Moreover, when $q=16$ we have $|H:H_0|=2$. See Remark \ref{rem:pablo}.
\end{tablenotes}
\end{threeparttable}
\end{table}

    \item $\cU=\cN_3^\eps$, $m=4$, $q$ is odd and ${}^\wedge\Omega_7(q)^I\lhdeq H$; or

    \item $\cU=\cN_4^-$, $m=4$, and either $q=2$ and $H=\Omega_7(2)^I$,  or $q=4$ and $\Omega_7(4)^I.2\lhdeq H$.
   \end{enumerate}
\end{theorem}

\begin{proof}
Let $G$ be the setwise stabiliser in $\GammaO^+_{2m}(q)$ of $\cU$ and let $H\le G$ act transitively on $\cU$ such that $\Omega_{2m}^{+}(q)\not\le H$. Note that $G=\GammaO^+_{2m}(q)$ unless $\cU=\cP_m^\eps$,  which consists of totally singular $m$-subspaces, or 
$\cU=\cN_k^\eps$, the set  of nondegenerate $k$-subspaces of a given isometry type $\eps$ with $kq$ odd. In these exceptional cases $G$ is an index two subgroup of $\GammaO^+_{2m}(q)$, and $G\geqslant \Omega^+_{2m}(q)$ if $\cU=\P_m^\eps$, or $G\geqslant \GO_{2m}^+(q)$ if $\cU=\cN_k^\eps$. By Theorem~\ref{T:KleinSL4} we may assume that $m\geq4$, and by Lemma~\ref{lem:m4q3} we may assume also that $(m,q)\neq (4,3)$.  If $\cU=\cN_1$ (with $q$ even) or $\cU=\cN_1^\varepsilon$ (with $q$ odd), then part (e) follows from  \cite{Reg}*{Lemma~4.5}, see also Remark~\ref{R:errata}.

Next suppose that  $\cU=\cN_2^+$ and $q=2$. Then $U\in\cN_2^+$ contains a unique nonsingular vector $ v$ and so $G_U\leqslant G_v$. This implies that $H$ acts transitively on $\cN_1$ and so is given by part (e). By~\eqref{E:OPlus2}, $|\cN_2^+|=2^{2m-3}(2^m-1)(2^{m-1}+1)$. Hence if $m>4$ then $|H|$ is divisible by a primitive prime divisor of $2^{2m-2}-1$. It follows that  either $m$ is even and $\SU_m(2)\lhdeq H$, or $m=6$ and $\nonsplit{3}{M_{22}}\lhdeq H$, or $m=4$. A \textsc{Magma} calculation shows that  $N_{G}(\nonsplit{3}{M_{22}})$ is not transitive on $\cN_2^+$ when $m=6$, while another shows that when $m=4$ either 
 $\Alt_9$, $\Omega_7(2)^I$ or $\SU_4(2)$ is normal in $H$. Moreover, Lemma~\ref{lem:SUorthog} shows that $\SU_m(2)$ is transitive on $\cN_2^+$ for all even $m$. Thus $H$ is as listed in Table~\ref{T:OPlusd}.

 We also note that a \textsc{Magma} calculation shows that, when $\cU=\cN_4^+$ and $(m,q)=(4,2)$, there are no possibilities for $H$. 

 The remaining families $\cU$ to be considered are therefore $\cP_k$ for $k< m$, $\cP_m^\eps$, $\cN_k^\eps$ for even 
 $k\leq m$ with $(\cU,q)\ne (\cN_2^+,2)$ and $(\cU,m,q)\ne (\cN_4^+,4,2)$, and
 if $q$ is odd also $\cN_k^\eps$ with $k$ odd  and $3\leq k\leq m$.  Let $U\in \cU$. 
 If $\cU=\cN_m^\eps$ and $U^\perp$ is isometric to $U$ then $G_U$ is an index 2 subgroup of $G_{\{U,U^\perp\}}$, which by \cite{King1981a}  is maximal in $G$. In all other cases, $G_U=G_{\{U,U^\perp\}}$, and moreover $G$ acts primitively on $\cU$: if 
 $(\cU,q)= (\cN_2^+,3)$ this follows from \cites{KL,BHRD} (recalling that here $G=\GammaO_{2m}^+(q)$), and in all other cases it follows from \cite{King1981}, \cite{King1981a} and \cite{King1982} (noting that $G$ contains either $\GO^+_{2m}(q)$ or $\Omega_{2m}^+(q)$). 
 Thus in particular, in all cases  $G_{\{U,U^\perp\}}$ is maximal in $G$.

Denote the image of any subgroup $X$ of $\GammaO^+_{2m}(q)$ in $\PGammaO_{2m}^+(q)$ by $\overline{X}$ and let $L=\POmega_{2m}^+(q)$. Then we have a factorisation $\oG=\oH\,\oG_{\{U,U^\perp\}}$. Since $\Omega_{2m}^+(q)$ is quasisimple, Lemma~\ref{lem:derived} implies that $L\not\leqslant \oH$. Let $\oB$ be maximal among core-free subgroups of $G$ containing $\oH$ and let $G^*=\oB L$. Then by Lemma~\ref{lem:maxmin} and the fact that $(m,q)\neq (4,3)$ we have that $\oB$ is maximal in $G^*$ and we have a core-free maximal factorisation $G^*=\oB(G^*)_{\{U,U\}}$.

Suppose first that $\cU=\cN_m^\varepsilon$   and $U$ is isometric to $U^\perp$.  Since we have a core-free maximal factorisation $G^*=\oB(G^*)_{\{U,U\}}$, \cite{Factns} implies that $m=4$, $q=2$ or $4$, $\cU=\cN_4^-$, and $\oB=N_G(\Omega_7(q)^I)$.  A \textsc{Magma} calculation shows that when $q=2$ the only example we get is $H=\Omega_7(2)^I$. When $q=4$, \cite{BHRD}*{Table~8.50} implies that $N_{\oG}(\Omega_7(4)^I)=\Omega_7(4)^I.2$, where the extra 2 is a field automorphism. Constructing an appropriate $\Omega_8^+(4).2$ as a subgroup of $\Omega_{16}^+(2)$ in \textsc{Magma}, enables us to calculate the intersection of $N_{\oG}(\Omega_7(4)^I)$ with the stabiliser of a subspace in $\cN_4^-$ and show that $N_{\oG}(\Omega_7(4)^I)$ is transitive on $\cN_4^-$ while $\Omega_7(4)^I$ is not. Thus $\Omega_7(4)^I.2\lhdeq H$. This gives part (h).
 
 In all remaining cases  $G_U=G_{\{U,U^\perp\}}$, and $G_U$ is maximal in $G$. Then $\cU$ and the possibilities for $\oB$ are given by \cite{Factns} and are collated in Table~\ref{T:OrthogCases} according to the case of the theorem in which the family $\cU$ is addressed.. Note that  the Case~(g) appears in disguise in~\cite{Factns}*{Table~4, p.\,14}
because when $m=4$ the subgroup $(\PSp_2(q)\otimes\PSp_4(q)).2$ is isomorphic to $ (\POmega_3(q)\times\POmega_5(q)).2$  and is mapped to an $N_3$-subgroup under triality. (In constructing Table~\ref{T:OrthogCases} recall that $(m,q)\ne (4,3)$ and also that our notation for the entries $P_m, P_{m-1}$ in the tables in \cite{Factns} is $\cP_m^\eps$.)

\begin{table}[!ht]
\caption{Remaining possibilities for $\cU$ when $L=\POmega^+_{2m}(q)$.}\label{T:OrthogCases}
\begin{tabular}{rcccccc}
\toprule
\textup{Cases} && ${\rm (a)}$& ${\rm (d)}$& ${\rm (f)}$& ${\rm (g)}$\\
$\cU$&&$\cP_1$&$\cP_m^\eps$&$\cN_2^\eps$&$\cN_3^\eps$\\
\bottomrule
\end{tabular}
\end{table}

By Lemma~\ref{lem:maxmin}, $\oB=N_{\oG}(\oB\cap L)$ and so letting $B$ be the full preimage of $\overline{B}$ in $\GammaO^+_n(q)$, the Correspondence Theorem implies that 
$B=N_G(B\cap \Omega_{2m}^-(q))$.  Table~\ref{T:OCasesParts} lists the groups $B_0$ such that $B=N_G(B_0)$ for each possibility for $\cU$.  We provide the following extra justification for when $m=4$:
By~\cite{Factns}*{Table~4}, if $L=\Omega^+_8(4)$  then there is a 
factorisation $G=AB$ where $A\cong N_2^-$ and $B\cong N_2^+$. However $N_2^-$ 
is not transitive on $\cN_2^+$ and similarly $N_2^+$ is not transitive on 
$\cN_2^-$ as, for a given $U\in \cN_2^-$, there exist $W_1,W_2\in\cN_2^+$ such that $U\cap W_1=0$ and $\dim(U\cap W_2)=1$. Indeed, as observed in Remark~\ref{R:tri}, under 
triality an $N_2^+$ subgroup becomes an $M_2$ subgroup, that is the stabiliser of a decomposition of $V$ into a pair of maximal totally singular $4$-subspaces, and an $N_2^-$ subgroup 
becomes an $\SU_4(q)$ subgroup. 
Thus the factorisation is as listed  in cases
$(\textrm{f}^-2)$ and $(\textrm{f}^+1)$ respectively.
Similarly, in cases (f$^+2$) and (g), the subgroup $B$ must be $N_G({}^\wedge\Omega_7(q)^I)$ as $\Omega_7(q)^R$ does not act transitively on $\cN_2^{+}$ or $\cN_3^\eps$.  Furthermore, \cite{Factns}*{Table~4} lists a factorisation $L=AB$ with $A\cong N_2^-$ and $B=P_1,P_4^+$, or $P_4^-$.  Moreover, there are two possible $P_i$ for each choice of $A$. Since the orthogonal complement of a subspace in $\cN_2^-$ contains totally singular $1$-subspaces, it is not possible for $N_2^-$ to be transitive on $\cP_1$. Thus $N_2^-$ is transitive on $\cP_4^-$ and $\cP_4^+$. This gives case (d2). Moreover, the image of $N_2^-$ under triality is a $\SU_4(q)$ subgroup and so such a subgroup is transitive on $\cP_1$ and either $\cP_4^-$ or $\cP_4^+$. This gives cases (a2) and (d4).

\begin{table}[!ht]
\begin{threeparttable}
\caption{Options for $\cU$, $(m,q)$ and $B=N_G(B_0)$.}\label{T:OCasesParts}
\begin{tabular}{rcccccccc}
\toprule
\textup{Cases} && ${\rm (a1)}$& ${\rm (a2)}$& ${\rm (a3)}$& \\
$\cU$&&$\cP_1$&$\cP_1$&$\cP_1$&\\
$(m,q)$&&$(4,2)$&$(\textup{even},{\rm all})$&$(4,{\rm all})$&\\
$B_0$&&$*$&$\SU_m(q)$&${}^\wedge\Omega_7(q)^I$&\\
\midrule
\textup{Cases} && ${\rm (d1)}$& ${\rm (d2)}$& ${\rm (d3)}$& ${\rm (d4)}$& ${\rm (d5)}$&  \\
$\cU$&&$\cP_m^\eps$&$\cP_m^\eps$&$\cP_m^\eps$&$\cP_m^\eps$&$\cP_m^\eps$&\\
$(m,q)$&&$(\ge4,{\rm all})$&$(\ge4,{\rm all})$&$(4,{\rm all})$&$(4,{\rm all})$&$(4,2)$&\\
$B_0$&&$N_1$&$N_2^{-}$&$\Omega_{7}(q)^I$& $\SU_4(q)$&$\Alt_9$&\\
\midrule
\textup{Cases} && {\rm (f}${}^{+}1$)& {\rm (f}${}^{+}2$)& {\rm (f}${}^{-}1$)&{\rm (f}${}^{-}2$)&{\rm (f}${}^{-}3$) & {\rm (f}${}^{-}4$) & {\rm (f}${}^{-}5$)\\
$\cU$&&$\cN_2^{+}$&$\cN_2^{+}$&$\cN_2^{-}$&$\cN_2^{-}$&$\cN_2^{-}$&$\cN_2^{-}$&$\cN_2^{-}$\\
$(m,q)$&&$(\textup{even},4)$&$(4,>2)$&$(\ge4,{\rm all})$&$(\ge4,2\;{\rm or }\;4)$\tnote{*}&$(4,{\rm all})$ &$(4,2)$& $(4,4 \textup{ or} 16)$\\
$B_0$&&$\SU_m(q)$&${}^\wedge\Omega_7(q)^I$&$P_m$&$\GL_m(q)$&${}^\wedge\Omega_7(q)^I$& $\Alt_9$ & $\Omega_8^-(q^{1/2})$\tnote{**}\\
\midrule
\textup{Cases} && {\rm (g)} \\
$\cU$&&$\cN_3^{\eps}$ \\
$(m,q)$&& $(m,\textup{odd})$&\\
$B_0$ && ${}^\wedge\Omega_7(q)^I$\\
\bottomrule
\end{tabular}
\begin{tablenotes}\footnotesize
\item[*] $(m,q)\neq (4,2)$ as in this case $\GL_4(2).2\leqslant \Omega_7(2)^I$.
\item[**] This is the image under triality of a subfield subgroup and existence is only confirmed for $q=4$. See Remark \ref{rem:pablo}.
\end{tablenotes}
\end{threeparttable}
\end{table}

\smallskip
\textsc{Case $\cU=\cP_1$.} We consider cases (a1) to (a3) in Table~\ref{T:OCasesParts}.

(a1)~Here $\GammaO_8^+(2)=\textup{SO}_8^+(2)$ acts faithfully on $|\cP_1|=135$ singular vectors. We use {\sc Magma} and find that there are eight conjugacy classes of transitive subgroups of $\GO_8^+(2)$. Moreover, if $\Omega_8^+(2)\not\leqslant H$ then $H$ normalises one of $\SU_4(2)\cong\PSp_4(3)$, $\Alt_9$ or $\Omega_7(2)^I$, as listed in  Table~\ref{T:OPlusa}.

(a2)~In this case $B=N_G(\SU_m(q))$ where $m$ is even and $H\le B$ is transitive on $\cP_1$. Consider~\cite{Factns}*{3.6.3(a), p.\,70}, and the
argument on~\cite{Factns}*{3.5.2(a), p.\,60}. Let $k=\F_q$ and $K=\F_{q^2}$.
Here $V=K^m$ admits a $K$-hermitian form $[\, ,\,]$ and the $k$-quadratic form $Q$ preserved by ${\rm O}_{2m}^+(q)$ satisfies $Q(v)=[v,v]$ for $v\in V$.
Let $U\in \cU$. Then $KU:=\{\lambda u\mid \lambda\in K, u\in U\}$
is a singular $1$-subspace of the hermitian space $V=K^m$.  Since $m\ge4$, the subgroup $T=\SU_m(q)$ is transitive on
the set of nonzero singular vectors of $K^m$ by~\cite{KL}*{Lemma~2.10.5,\;p.\,49}. Moreover, the stabilizer $T_{KU}$ is transitive
on the set of $q+1$ singular $1$-dimensional $k$-subspaces of $KU$. Thus $\SU_m(q)$ is transitive on $\cU$ as recorded in Table~\ref{T:OPlusa}.
Suppose now that $\SU_m(q)\not\le H$. Then we have the factorisation $B=B_U\,H$. Since $B_U$ fixes $KU$, which is a 1-dimensional $K$-subspace that is totally singular with respect to the hermitian form, it follow that $H$ acts transitively on the set of totally singular 1-subspaces over $K$. Since $m$ is
even and $(m,q)\neq (4,3)$, it follows from  Theorem~\ref{T1} that there are no such factorisations.

(a3) In this case $m=4$ and $G=P_1\,B$ where $B=N_G({}^\wedge\Omega_7(q)^I)$. Write $P_1=G_U$ where $U\in\cP_1$.
In light of Case~(a1) and our assumption that $(m,q)\neq (4,3)$, we will assume $q\ge4$. By~\cite{Factns}*{Table~4} we see that ${}^\wedge\Omega_7(q)^I$ is transitive on $\cP_1$ and so 
we obtain the example ${}^\wedge\Omega_7(q)^I\lhdeq H$ in Table~\ref{T:OPlusa}. 
Suppose now that ${}^\wedge\Omega_7(q)^I\not\le H$. Since $H\le B$, and $H$ is transitive on $\cP_1$, we have $B=B_U\,H$ and $\overline{B}=\overline{B}_U\oH$. 
Set $S:=\overline{B}^{(\infty)}=\Omega_7(q)^I$ and let $X$ be the natural 7-dimensional module for $S$ over $\F_q$. The formula~\eqref{E:OPlus1} for $|\cP_1|$ gives
\[
  |\cU|=\frac{(q^4-1)(q^3+1)}{q-1}=|G:G_U|=|S:S_U|.
\]
Since $\cU=|S:S_U|$ is coprime to $q$, it follows from \cite{Seitz}*{(1.6)} that any maximal subgroup of $S$ containing $S_U$ is a parabolic subgroup. Comparing $|\cU|$ with the indices of parabolic subgroups given in \eqref{E:OSizeN1Pm}, we deduce that $S_U$ is the stabiliser in $S$ of a totally singular $3$-subspace of $X$. Thus $\oH$ acts transitively on the set of totally singular $3$-subspaces of $X$.

Since $q\geq 4$, Theorem~\ref{T4}(a) (see Table~\ref{T:Oa}) when $q$ is odd and  Theorem~\ref{T:Ooddqeven} when $q$ is even, imply that $\Omega_6^-(q)\lhdeq \oH$. In particular, $H$ fixes a unique nondegenerate subspace $W$ of $X$. Thus letting $\tau$ be a triality automorphism of $\POmega_8^+(q)$ such that $\oB^\tau$ normalises an $\Omega_7(q)^R$ subgroup we have that $\Omega_6^-(q)^\tau\lhdeq \oH^\tau\leqslant N_2^-$.  Since $\tau$ cyclically permutes $N_2^-$ with two $\mathcal{C}_3$-subgroups of type $\SU_4(q)$ (see Remark~\ref{R:tri}), $\Omega_6^-(q)\cong \SU_4(q)/\langle \pm I\rangle$, and ${}^\wedge\Omega_7(q)=\nonsplit{2}{\Omega_7(q)}$ when $q$ is odd, 
it follows that $\SU_4(q)\lhdeq H$. Thus $H$ is already listed in Table~\ref{T:OPlusa}.

\smallskip

\textsc{Case $\cU=\cP_m^\eps$.} There are two orbits of $L$ on the set of totally singular $m$-subspaces:
one is labelled $\cP_m^+$ and the other $\cP_m^-$.  The two orbits are fused in $\PGammaO^+_{2m}(q)$. Hence there is no loss of generality in assuming that  
$\cU=\cP_m^+$ and $U\in\cP_m^+$,  a (maximal) totally singular $m$-subspace of $V$. 
When $m=4$,  triality induces a 3-cycle $(\cP_m^+,\cP_m^-,\cP_1)$ by Remark~\ref{R:tri}. Thus a suitable image of $\oH$ under triality is transitive on $\cP_1$ and so is given in part~(a). Recall that  triality induces the 3-cycles $(\Omega_7(q)^I,\Omega_7(q)^I,\Omega_7(q)^R)$ and $(N_L(K_1), N_L(K_2),N_2^-)$ on subgroups of $\POmega_8^+(q)$, where $K_1$ and $K_2$ are the images in $L=\POmega_8^+(q)$ of subgroups isomorphic to $\SU_4(q)$. Then also using \cite{BHRD}*{Table~8.50} to obtain the correct structure for $H$ in the other cases, we list all possibilities obtained in this way in Table~\ref{T:OPlusb}. 
Thus from now on we may assume that $m\geq 5$, and in particular we only need to deal with the cases (d1) and (d2) in Table~\ref{T:OCasesParts}.

(d1) In this case  $B=G_{\langle v\rangle}$ for some nonsingular vector $v$. Set $W=\langle v\rangle^\perp$.
Observe that $(v,v)=Q(v+v)-Q(v)-Q(v)=2Q(v)=0$ if and only if $q$ is even.
Hence $N_1$ preserves the decomposition 
$V=\langle v\rangle\oplus W$ when $q$ is odd, and preserves the chain $0<\langle v\rangle<W<V$ when
$q$ is even. In both cases the restriction $N_1^W$ of $N_1$ to $W$ 
contains $\Omega_{2m-1}(q)$ by Witt's lemma, and we note that $\Omega_{2m-1}(q)\cong\Sp_{2m-2}(q)$ when $q$ is even. Moreover, $\Omega_{2m-1}(q)=B^{(\infty)}=(B^W)^{(\infty)}$, and so if $\Omega_{2m-1}(q)\leqslant H^W$ then Lemma~\ref{lem:derived} implies that $\Omega_{2m-1}(q)^R\lhdeq H$, as listed in Table~\ref{T:OPlusb}.

Suppose now that $\Omega_{2m-1}(q)\not\le H^W$. 
Now $Y:=\langle v\rangle^\perp\cap U$ is a totally singular
$(m-1)$-subspace that is stabilised by $B\cap P_m$.
Hence $B\cap P_m\le B_Y$ and $(B\cap P_m)^W\le (B_Y)^W=P_{m-1}(B^W)$.
Thus we have a factorisation $B^W=(B\cap P_m)^W\,H^W=P_{m-1}(B^W)\,H^W$ and therefore $H^W$ is transitive on the set of totally singular $(m-1)$-subspaces of $W$.
When $q$ is odd and recalling that $m\geq 5$, we deduce from Theorem~\ref{T4} that $\Omega_{2m-2}^-(q)\lhdeq H^W$. Similarly, when $q$ is even we note that $\Omega_{2m-1}(q)\cong\Sp_{2m-2}(q)$ and we deduce from Theorem~\ref{T3}(b) that again  $\Omega_{2m-2}^-(q)^R\lhdeq H^W$. In both cases, it follows from Lemma~\ref{lem:derived} that $\Omega_{2m-2}^-(q)^R\lhdeq H$, as listed in Table~\ref{T:OPlusb}.

(d2) In this case $B= N_2^{-}=G_{W^\perp}=G_W$ for some
 nondegenerate $(2m-2)$-subspace $W$ such that $W^\perp\in \cN_2^{-}$. In particular, $N_2^{-}$ preserves the
decomposition $V=W\oplus W^\perp$ and by  \cite{KL}*{Lemma~4.1.1}, $N_2^{-}$ contains $\Omega_2^{-}(q)\times \Omega_{2m-2}^{-}(q)$ as a normal subgroup. 
Since $\GO_2^{-}(q)$ is solvable it follows that $B^{(\infty)}=(B^W)^{(\infty)}=\Omega_{2m-2}^-(q)$. Thus by Lemma~\ref{lem:derived}, if $\Omega_{2m-2}^-(q)\leqslant H^W$ it follows that $\Omega_{2m-2}^-(q)^R\lhdeq H$, as listed in Table~\ref{T:OPlusb}.

Suppose next that $\Omega_{2m-2}^-(q)\not\leqslant H^W$, so we have a factorisation $B^W=(B\cap P_m)^WH^W$. 
Observe that $U\cap W^\perp=0$ as $U\in\cP_m$ and $W^\perp\in\cN_2^{-}$. Therefore $\dim(U\oplus W^\perp)=m+2$, and hence also 
$\dim(U\cap W)=m-2$ as $(U\oplus W^\perp)^\perp=U^\perp\cap W=U\cap W$. Since $\dim(W)=2m-2$
and $W$ has Witt index $\frac{2m-2}{2}-1=m-2$, it follows that $Y:=U\cap W$ is a (maximal) 
totally singular $(m-2)$-subspace of $W$. Thus $B\cap P_m\leqslant B_Y$ and we have a factorisation $B^W=(B_Y)^WH^W$. In particular, $H^W$ is a subgroup of $\GammaO^-_{2m-2}(q)$ that is transitive on totally singular $(m-2)$-subspaces and does not contain $\Omega_{2m-2}^-(q)$. Since $m\geq 5$, Theorem~\ref{T9} implies that no such group exists.

\smallskip
\textsc{Case $\cU=\cN_2^+$.} Here $G=\GammaO_{2m}^+(q)$. We follow the case distinction in Table~\ref{T:OCasesParts}, and we note that in each case $q>2$.   

(f${}^{+}$1)~In this case, $m$ is even and  $q=4$. Furthermore, $G$ factorises as $G=G_U\,B$ where $G_U=N_2^{+}$ and $B=N_G(\SU_m(4))$. The case $\SU_m(4)\lhdeq H$ is recorded in Table~\ref{T:OPlusd} so suppose that $\SU_m(4)\not\le H$.
  Then $B=B_U\,H$.
Let $K=\F_{16}$ and $k=\F_4$. It follows from~\cite{Factns}*{p.\,70} that the $K$-subspace $KU$ spanned by $U\in\cN_2^{+}$ has $K$-dimension~2. Hence $KU$ lies
in $\cN_2(\SU_m(4))$ and 
is preserved by $B_U$. Therefore $H$ is a core-free subgroup of $B=N_G(\SU_m(4))$ which is transitive on $\cN_2(\SU_m(4))$, but there is no such subgroup by Theorem~\ref{T1}.

(f${}^{+}$2)~In this case $m=4$ and $q>2$.  We consider the factorisation $G=G_U\,B$ where $G_U=N_2^{+}$ and $B=N_G({}^\wedge\Omega_7(q)^I)$.  
Applying $\tau$ to the factorisation $\oG=N_2^{+}\,\oB$ and recalling that $\oG=\PGammaO_8^+(q)$, gives $\oG=(N_2^{+})^\tau\,(\oB)^\tau$ where $(N_2^{+})^\tau=M_2$ and
$(\oB)^\tau=N_{\oG}(\Omega_7(q)^I)^\tau=N_{\oG}(\Omega_7(q)^R)$. Thus we have a factorisation $(\oB)^\tau=(M_2\cap (\oB)^\tau)(\oH)^\tau$. Now $(\oB)^\tau=\oG_{\langle v\rangle}$ for some nonsingular vector $v$ and we let $W=v^\perp$. When $q$ is odd, $(\oB)^\tau$ fixes the decomposition $V=\langle v\rangle\perp W$ while when $q$ is even, $(\oB)^\tau$ preserves the chain $\langle v\rangle <W<V$. 
Following~\cite{Factns}*{3.6.1(b), pp.\,63--64} we see that $M_2\cap (\oB)^\tau$ fixes a pair $\{U_1,U_2\}$ of complementary totally singular $3$-subspaces of $W$.  Suppose first that $q$ is even. Then $\Omega_7(q)\cong\Sp_6(q)$ and $M_2\cap (\oB)^\tau$ is contained in the stabiliser in $\GammaSp_6(q)$ of a pair of maximal totally singular subspaces. However, \cite{Factns} then implies that $\Sp_6(q)\lhdeq (\oH)^\tau$ and so $\Omega_7(q)^I\lhdeq H$ as listed in Table~\ref{T:OPlusd}. Next suppose that $q$ is odd.  Then $M_2\cap (\oB)^\tau$ fixes the nonsingular $1$-subspace $W\cap U_1^\perp\cap U_2^\perp$. Thus $(\oH)^\tau$ acts transitively on the set $\cN_1^+$ for $W$ and so is given by Theorem~\ref{T4}(c).  
As $m=4$, we have $|\cN_2^{+}|=\frac{q^6(q^4-1)(q^3+1)}{2(q-1)}$ by \eqref{E:OPlus2}.  Letting $q=p^f$ we see that $|\oH^\tau|$ must be divisible by a primitive prime divisor $r$ of $p^{4f}-1$ (one always exists). Note that $r>4f$ and so does not divide the order of the first three groups listed in Table~\ref{T:Oc2}. 
Further, since $3^6$ divides $|\cN_2^{+}|$ when $q=3$, a similar argument
eliminates the last two cases given by Table~\ref{T:Oc2}. Thus we must have  $\Omega_7(q)^R\lhdeq (\oH)^\tau$ and hence ${}^\wedge\Omega_7(q)^I\lhdeq H$ as listed in Table~\ref{T:OPlusd}  along with the example ${}^\wedge\Omega_7(2)^I\lhdeq H$ already obtained when $q=2$.

\smallskip
\textsc{Case $\cU=\cN_2^-$.} We again follow the case distinction in Table~\ref{T:OCasesParts}.

(f${}^{-}$1)~In this case $m\ge4$ and $q=p^f\ge2$. We consider the factorisation $G=G_U\,P_m$ where $G_U=N_2^{-}$. Following~\cite{Factns}*{(3.6.2b), p.\,68}, we write $U\in\cN_2^-$ as
$U=\langle e_1+f_1,e_1+e_2+\kappa f_2\rangle$ where $x^2+x+\kappa$ is an irreducible quadratic over $\F_q$. Also, we may write $P_m=G_E$ where $E=\langle e_1,\dots,e_m\rangle$.
A direct calculation shows that $Y:=E\cap U^\perp=\langle e_3,\dots,e_m\rangle$. Clearly $G_E\cap G_U\le G_Y$ so $G_E\cap G_U\le G_E\cap G_Y$.
As $H\le G_E$ and $H$ is transitive on $\cN_2^-$, we have $G_E=(G_E\cap G_U)H$, that is to say, $P_m=(P_m\cap G_Y)H$.
Now $G_E=P_m$ is an extension of a $p$-group by a subgroup of $\GammaL_m(q)$. The homomorphism $\pi\colon G_E\to\GammaL_m(q)$ such that $\SL_m(q)\lhdeq\pi(G_E)$ maps the factorisation $P_m=(P_m\cap G_Y)H$
to $\GammaL_m(q)=P_{m-2}(\GammaL_{m}(q))\pi(H)$. One possibility is that $\SL_m(q)\le\pi(H)$. Here $H\le P_m$ and modulo its unipotent radical (the aforementioned $p$-group),
$\pi(H)$ contains $\SL_m(q)$ as stated in Theorem~\ref{T10}~(f). Suppose now that $\SL_m(q)\not\le\pi(H)$. Then by~\cite{Factns}*{pp.\,10, 13} we have $(m,q)=(5,2)$ and $\pi(H)\le 31\ldotp5$.
However, there are $155=31\cdot5$ subspaces of $(\F_2)^5$ of codimension~$2$. Thus $\pi(H)=31\ldotp5$, and $\GL_5(2)=P_3(\GL_5(2))\pi(H)$ is an exact factorisation.
A computation with {\sc Magma} shows that if $H$ satisfies
$\GO_{10}^+(2)=N_2^-\,H$ and $\pi(H)=31\ldotp5$, then the unipotent radical of $H$ has order $2^{10}$, and we have $H=[2^{10}]\ldotp31\ldotp5$ as in Table~\ref{T:OPlusd}.

(f${}^{-}$2)~In this case $m\ge4$ and $q=2$ or $4$, with $(m,q)\neq (4,2)$. We consider the factorisation $G=G_U\,B$ where $G_U=N_2^{+}$ and $B=N_G(\GL_m(q))$.  Write $B=G_{\{E,F\}}$ where $E=\langle e_1,\dots,e_m\rangle$ and $F=\langle f_1,\dots,f_m\rangle$. 
We view $B$ as $\GammaL(E)\ldotp2$ acting on $E$ where the $2$ acts as the inverse-transpose automorphism.
 If $\SL_m(q)\le H$, we obtain the example $\SL_m(q)\lhdeq H$ listed in Table~\ref{T:OPlusd}. Suppose now that $\SL_m(q)\not\le H$. We will argue that
this possibility does not arise.

Since $H\le B$ is transitive, we have a factorisation $B=B_U\,H$. Following~\cite{Factns}*{(3.6.2b), p.\,68}, we write $U\in\cN_2^-$ as
$U=\langle e_1+f_1,e_1+e_2+\kappa f_2\rangle$ where $x^2+x+\kappa$ is an irreducible quadratic over $\F_q$.
A direct calculation shows that $F\cap U^\perp=\langle f_3,\dots,f_m\rangle$ and $E_{m-2}:=E\cap U^\perp=\langle e_3,\dots,e_m\rangle$. Since $E':=(F\cap U^\perp)^\perp\cap E=\langle e_1,e_2\rangle$ and $F':=(E\cap U^\perp)^\perp\cap F=\langle f_1,f_2\rangle$,
it follows that $B_U=N_2^-\cap B=G_U\cap G_{\{E,F\}}$ preserves the set $\{E',F'\}$.
Therefore $B_U$ preserves the decomposition $E=E'\oplus E_{m-2}$. Since $B=B_U\,H$ and $B_U\le\textup{Stab}(E'\oplus E_{m-2})$, we have
$\GammaL(E)\ldotp2=\textup{Stab}(E'\oplus E_{m-2})\,H$. However, no such factorisation exists by~\cite{Factns}*{pp.\,10,\,12,\,13}. In summary, the case $\SL_m(q)\not\le H$ does not arise.

(f${}^{-}$3) In this case $m=4$ and $H\leqslant B=N_G({}^\wedge\Omega_7(q)^I)$. The possibility ${}^\wedge\Omega_7(q)^I\lhdeq H$ is listed in Table~\ref{T:OPlusd}. As noted in Remark~\ref{R:tri}, a triality automorphism $\tau$ of $\POmega_8^+(q)$ maps $\oB$ to $N_1$, and $N_2^-$ to an $\SU_4(q)$ subgroup. Thus applying $\tau$ to the factorisation $\oG= \oG_U\,\oB$ we get a factorisation $(\oG)^\tau=(\oG_U)^\tau\, N_1$, where $(\oG)^\tau$ contains $\POmega^+_8(q)$ as a normal subgroup. Moreover, $(\oH)^\tau\leqslant N_1$. Suppose that ${}^\wedge\Omega_7(q)^I\not\leqslant H$. Then by Lemma~\ref{lem:derived}, $\Omega_7(q)^R\not\leqslant (\oH)^\tau$.  Hence we obtain a core-free factorisation $N_1=(\oH)^\tau (N_1\cap (\oG_U)^\tau)$.  Let $v$ be the vector such that $N_1=(\oG)_{\langle v\rangle}$.  Let $K=\F_{q^2}$ and $k=\F_q$. Then following ~\cite{Factns}*{p.\,60, 64} we see that $N_1\cap (\oG_U)^\tau$ fixes the $K$-span of $v$, which is a $1$-subspace over $K$ that is nondegenerate with respect to the hermitian form preserved by $(\oG_U)^\tau$. Thus $N_1\cap (\oG_U)^\tau$ also fixes the 3-dimensional nondegenerate $K$-subspace $(Kv)^\perp$ and hence fixes a 6-dimensional $k$-subspace of $v^\perp$ of minus type. Thus $(\oH)^\tau$ acts transitively on the $\cN_6^-$ subspaces of the 7-dimensional $k$-subspace $v^\perp$. When $q$ is odd this also means that $(\oH)^\tau$ acts transitively on the $\cN_1^-$ subspaces of $v^\perp$. Thus when $q$ is even, $(\oH)^\tau$ is given by Theorem~\ref{T:Ooddqeven}(e) and when $q$ is odd it is given by Theorem~\ref{T4}(b). By~\eqref{E:OPlus2}, 
$$ 
|\cN_2^-|=\frac{q^6(q^3-1)(q^4-1)}{2(q+1)}
$$
and so $|H|$ must be divisible by a primitive prime divisor $r$ of $p^{4f}-1$ where $q=p^f$. Note that $r$ always exists and $r>4f$. Thus none of the possibilities given by Theorem~\ref{T4}(b) for $q$ odd has order divisible by $r$. Hence $q$ is even. However, the only possibilities given by Theorem~\ref{T:Ooddqeven}(e) that have order divisible by $r$ satisfy $\Omega_6^+(q)\lhdeq H^\tau$ when $q=2$ or 4. Thus applying triality we deduce that $\SL_4(q)\lhdeq H$ as listed in Table~\ref{T:OPlusd}.

(f${}^{-}$4) In this case $(m,q)=(4,2)$ and $H\leqslant N_G(\Alt_9)=\Alt_9$. A \textsc{Magma} calculation shows that $H=\Alt_9$ or $\Alt_8$ as recorded in Table ~\ref{T:OPlusd}.  We also note that the $\Alt_8$ provided here is different from the subgroup $\SL_4(2)$ already listed as the $\SL_4(2)$ stabilises two 4- subspaces while $\Alt_8$ only fixes one.

(f${}^{-}$5)  In this case $m=4$, $q=4$ or $16$, and $H\leqslant N_G(\Omega_8^-(q^{1/2}))$ where $\Omega_8^-(q^{1/2})$ is the image under triality of a subfield subgroup. We refer to Remark \ref{rem:pablo} for details about this case. In particular,  $\Omega_8^-(q^{1/2})$ is not transitive on $\cN_2^-$, the group $\GO_8^-(2)$ is transitive on $\cN_2^-$ when $q=2$, and the group $\GO_8^-(4).2$ is possibly transitive on $\cN_2^-$ when $q=4$. We record both these as possibilities in Table~\ref{T:OPlusd}. Suppose now that $\Omega_8^-(q)\not\leqslant H$. Then we have a factorisation $N_G(\Omega_8^-(q^{1/2}))=HX$, where $X$ is the stabiliser in $N_G(\GO_8^-(q^{1/2}))$ of a subspace in $\cN_2^-$. Moreover, $X$ is contained in the stabiliser $Y$ of an elliptic 4-dimensional subspace of the natural module for $\GO_8^-(q^{1/2})$. Thus we have a factorisation $N_G(\GO_8^-(q^{1/2}))=HY$. However, by \cite{Factns}, no such factorisation exists, a contradiction.

\smallskip
\textsc{Case $\cU=\cN_3^\eps.$} In this case $m=4$, $q$ is odd and  $B=N_G({}^\wedge\Omega_7(4)^I)$. As noted earlier, applying an appropriate triality automorphism $\tau$ maps $\oB$ to $N_1$ and $\oG_U$ to a $\mathcal{C}_4$ subgroup of type $\PSp_2(q)\otimes\PSp_4(q)$. (Note that $\PSp_2(q)\cong\PSL_2(q)\cong \POmega_3(q)$ and $\PSp_4(q)\cong \POmega_5(q)$). This gives us a factorisation $(\oG)^\tau= (\oG_U)^\tau\,N_1$, where $(\oG)^\tau$ contains $\POmega_8^+(q)$ as a normal subgroup. Let $v$ be the vector such that $N_1=(\oG)_{\langle v\rangle}$.
By~\cite{Factns}*{p.\,65--66} we see that $N_1\cap (\oG_U)^\tau$ fixes a nondegenerate 4-subspace $W$ containing $v$ and also fixes the 3-subspace $v^\perp\cap W$. Now $H\leqslant B$ and so $(\oH)^\tau\leqslant N_1$ and acts transitively on the set of nondegenerate 3-subspaces of $v^\perp$. Thus by Theorem~\ref{T4}, $\Omega_7(q)^R\lhdeq (\oH)^\tau$. Applying triality and using Lemma~\ref{lem:derived} we deduce that ${}^\wedge\Omega_7(q)^I\lhdeq H$, as in part (g).  
\end{proof}

\subsection{Actions on~\texorpdfstring{$\cP_m$}{} and~\texorpdfstring{$\cN_1$}{}}\label{SS:OPlus2}

First we treat the subspace family $\cP_m=\cP_m^+\cup \cP_m^-$,  the set of all maximal totally singular subspace of $(\F_q)^{2m}$.  Recall that $\cP_m^+, \cP_m^-$ are the two $\Omega_{2m}^+(q)$-orbits in $\cP_m$  and that $\GammaO_{2m}^{+}(q)$ fuses these orbits. A totally singular $(m-1)$-subspace is contained in exactly two totally singular $m$-subspaces -- one from each $\Omega_{2m}^+(q)$-orbit.

\begin{proposition}\label{P:Both}
  Suppose that  $H\le\GammaO_{2m}^{+}(q)$, where $m\geq 3$, such that  $\Omega_{2m}^{+}(q)\not\le H$. Let 
  $V=(\F_q)^{2m}$ and suppose that $H$ acts transitively on $\cP_m$. Then   
either $m=3$ and
  $H\le N_{\GammaO_6^+(q)}(\Omega_2^-(q^2))\leqslant N_4^-$,  or $H_1\lhdeq H$ where $(H_1,m,q)$ is given in
    Table~$\ref{T:OPlusBoth}$
\begin{table}[!ht]
  \caption{Proposition~\ref{P:Both} subgroups $H_1$ with  $H_1\lhdeq H\le\GammaO^+_{2m}(q)$ where $H$ is transitive on $\cP_m$. We exclude all   $H\le N_{\GammaO_6^+(q)}(\Omega_2^-(q^2))\leqslant N_4^-$ with $m=3$.}\label{T:OPlusBoth}
\begin{tabular}{rlcccccccc}
\toprule
$H_1$ &&$\Omega_{2m-1}(q)^R$ &$\Omega^{-}_{2m-2}(q)^R$ &$2^4\colon \C_5$&$2^4\colon\Alt_5$&$\Alt_5$&$\Sym_7$\\
\textup{$m$}&& $\ge3$&$\ge3$&$3$&$3$&$3$&$3$\\
\textup{$q$}&& \textup{all}& \textup{all}&$3$&$3$&$3$&$2$\\
\midrule
$H_1$ &&$2^6\colon\Alt_7$&$2^7\colon\Alt_8$& $\nonsplit{2}{\GO_8^+(2)}$&$\nonsplit{2}{\GO_7(2)}$&$\nonsplit{4}{\PSL_3(4)}$& $\nonsplit{2}{\PSL_3(4)}$&\\
\textup{$m$}&&$4$&$4$&$4$&$4$&$4$&$4$\\
\textup{$q$}&&$3$&$3$&$3$&$3$&$3$&$3$\\
\midrule 
$H_1$ &&$\Alt_8$&$\Alt_9$ &$\Sym_9$&\\
\textup{$m$}&&$4$&$4$&$4$&\\
\textup{$q$}&&$3$&$3$&$2$&\\
\bottomrule
\end{tabular}
\end{table}
\end{proposition}

\begin{proof}
  Let $\cU=\cP_m$,  and let $G=\GammaO^+_{2m}(q)$, so $G$ fuses the two 
  $\Omega_{2m}^{+}(q)$-orbits $\cP_m^+$ and $\cP_m^-$ in $\cP_m$, and
 the setwise stabiliser $G^+$ of $\cP_m^+$  has index 2 in $G$.
  Clearly $\{\cP_m^{-},\cP_m^{+}\}$ is a system of imprimitivity for $G$,
  and hence for $H$.
  As $H$ is transitive on $\cP_m$, $H^+:=H\cap G^+$ is an index 2 subgroup of $H$ that fixes each of
 $\cP_m^+$ and $\cP_m^-$ setwise, and is transitive on both.
  Now $\Omega_{2m}^+(q)\not\leqslant H^+$, by our assumptions on $H$, and so $H^+$ is given by Theorem~\ref{T10}(d). 
  One possibility is that $m=3$ and $H^{+}\le N_A(\Omega_2^-(q^2))\leqslant N_4^-$ where $A$ is the stabiliser in $\GammaO_6^+(q)$ of $\cP_3^+$ and so this becomes $N_{\GammaO_6^+(q)}(\Omega_2^-(q^2))\leqslant N_4^-$. Suppose now that this is not the case. Then  by 
  Theorem~\ref{T10}(d), $H^+$
  contains a normal subgroup $H_0$ given by Table~\ref{T:OPlusb}.

  In all but one case $H_0$ is perfect and is equal to the
soluble residual $(H^+)^{(\infty)}$ of $H^{+}$. In particular $H_0$ is a characteristic subgroup of $H^+$ and hence is normal in $H$.  
  In the exceptional case  $H_0=2^4\colon 5$ in $\Omega_6^+(3)$, and again 
  $H_0$ is a characteristic subgroup of $H^+$ and is normal in $H$.  
  Note that $H_0$ in this proof refers to the subgroup $H_0$ listed
  in Table~\ref{T:OPlusb}. The subgroup $H_1$ listed in
  Table~\ref{T:OPlusBoth} contains $H_0$ but may be larger,
  such as $H_1=\langle H_0,h\rangle$ where $h$ is described below.

  The subgroups $H$ determined by Table~\ref{T:OPlusb} are transitive
  on $\cP^{+}_m$ \emph{or} $\cP^{-}_m$, but we require $H$ to be
  transitive on \emph{both}. This holds if and only if there is an
  element  $h\in  G\setminus G^{+}$ such that $h$ normalises $H_0$. 
 Such an element $h$ interchanges $\cP^{+}_m$ and $\cP^{-}_m$, and
  conjugates the stabiliser in $H$ of a subspace of $\cP_m^{+}$
  to the stabiliser of a subspace of $\cP_m^{-}$.

  {\sc Case } $H_0=\Omega_{2m-1}(q)^R$. Here $H_0$ stabilises a nonsingular $1$-subspace. Suppose first that $q$ is even. Then there is a unique $\Omega_{2m}^+(q)$-orbit on nonsingular $1$-subspaces and hence a unique $\Omega_{2m}^+(q)$-conjugacy class of subgroups $H_0$. Hence $H_0$ is normalised by some element $h\in \GO_{2m}^+(q)\backslash \Omega_{2m}^+(q)$. As discussed above, $h$ interchanges $\cP_m^+$ and $\cP_m^-$, and in this case $H_0$ itself is transitive on both $\cP_m^+$ and $\cP_m^-$, so $\langle H_0,h\rangle$ acts transitively on $\cP_m$. When $q$ is odd there are two $\Omega_{2m}^+(q)$-orbits on nonsingular $1$-subspaces and hence two $\Omega_{2m}^+(q)$-conjugacy classes of subgroups $H_0$. These orbits and conjugacy classes are fixed setwise by $\GO_{2m}^+(q)$ but fused in $\CO_{2m}^+(q)$. Hence $H_0$ is normalised by some element $h\in\GO_{2m}^+(q)\backslash\SO_{2m}^+(q)$. Again, $h$ interchanges $\cP_m^+$ and $\cP_m^-$ and $\langle H_0,h\rangle$ is transitive on $\cP_m$, and we record $H_1=H_0$ in Table~\ref{T:OPlusBoth}.

  {\sc Case } $H_0=\Omega_{2m-2}^-(q)^R$. Here there is a unique $\Omega_{2m}^+(q)$-orbit on elliptic subspaces of codimension two and hence a unique $\Omega_{2m}^+(q)$-conjugacy class of subgroups $H_0$.   Hence $H_0$ is normalised by an element of $\GO_{2m}^+(q)\backslash \SO_{2m}^+(q)$, which  interchanges $\cP_m^+$ and $\cP_m^-$ and so $N_G(H_0)$ is transitive on $\cP_m$. We record $H_1=H_0$ in Table~\ref{T:OPlusBoth}.
  
  {\sc Case } $m=4$ and $H_0=\Omega_7(q)^I$.  Here $\overline{H_0}$ is the image under a triality automorphism of $\overline{\Omega_7(q)^R}$. We saw in the first above that $\Omega_7(q)^R$ is transitive on both $\cP_4^+$ and $\cP_4^-$. Moreover,  the stabiliser of a nonsingular vector $\langle u\rangle$ is not transitive on $\cP_1$ as there are singular vectors in $u^\perp$ and $V\backslash u^\perp$. Thus $\Omega_7(q)^R$ is not transitive on $\cP_1$.  As explained in Remark~\ref {R:tri}, a triality automorphism induces a 3-cycle $(\cP_1,\cP_4^+,\cP_4^-)$, and it follows that $H_0$ is transitive on $\cP_1$ and is transitive on exactly one of $\cP_4^+, \cP_4^-$. Thus we do not obtain groups transitive on $\cP_4$ in this case.

  {\sc Case } $H_0=\SU_4(q)$.  Here $\overline{H_0}$ is the image under a triality automorphism of $\overline{\Omega_6^-(q)^R}$. We saw in the second above that $\Omega_6^-(q)^R$ is transitive on both $\cP_4^+$ and $\cP_4^-$. Moreover,  the stabiliser of an elliptic subspace $U$ of codimension two is not transitive on $\cP_1$ as there are singular vectors in $U$ and $V\backslash U$. Thus $\Omega_6^-(q)^R$ is not transitive on $\cP_1$.  Arguing as in the previous case we deduce that there are no groups transitive on $\cP_4$ in this case.

In all other cases $m\in\{3,4\}$ and $q\in\{2,3\}$. 
If $(m,q)=(4,3)$ then the groups occurring in Table~\ref{T:OPlusBoth} are precisely those identified in Lemma~\ref{lem:m4q3}~(2). The remaining cases follow from a \textsc{Magma} calculation informed by Theorem~\ref{T:KleinSL4}(b).
\end{proof}

Now we consider the families $\cN_k$ where both $k$ and $q$ are odd. Recall that in this situation 
$\cN_k=\cN_k^+\cup \cN_k^-$, where $\cN_k^+, \cN_k^-$ are the two $\Omega_{2m}^+(q)$-orbits in $\cN_k$, and that
the group $\CO_{2m}^+(q)$ interchanges $\cN_k^+$ and $\cN_k^-$ and hence is transitive on $\cN_k$.

\begin{proposition}\label{P2}
  Suppose that  $H\le\GammaO_{2m}^{+}(q)$, where $m\geq 3$ and $q$ is odd, such that  $\Omega_{2m}^{+}(q)\not\le H$. Let 
  $V=(\F_q)^{2m}$ and suppose that $H$ acts transitively on the set $\cU=\cN_k$ for some odd $k\leq m$. Then   
either 
  \begin{enumerate}[{\rm (a)}]
  \item $\cU=\cN_1$, $H\leqslant P_m$ fixes a totally singular $m$-subspace $W$ 
 and, modulo
    the unipotent radical, $H$ induces a subgroup of
    $\GammaL_m(q)$ which is transitive  on $1$-subspaces; or
  \item $\cU=\cN_1$, and $H_1\lhdeq H$ where $(H_1, m, q)$  is given in
    Table~$\ref{T:OPlusBothQKOdd}$; or
\begin{table}[!ht]
  \caption{Proposition~\ref{P2}(b) subgroups $H_1$ with $H_1\lhdeq H\le\GammaO^+_{2m}(q)$
    where $H$ is transitive on $\cN_1$. We exclude cases when
    $H\le P_m$.}\label{T:OPlusBothQKOdd}
\begin{tabular}{rlcccccc}
\toprule
$H_1$ &&$\SL_m(q)$ &$\SU_m(q)$&$\Sp_m(q)$&$\Omega_9(q)$&$\nonsplit{2}{\Omega_7(q)}$&$\nonsplit{2}{\Omega_8^{-}(q_0)}$\\
\textup{$m$}&& $\ge3$&$2a\ge4$&$2a\ge4$&$8$&$4$&$4$\\
\textup{$q$}&& \textup{odd}&\textup{odd}&\textup{odd}&\textup{odd}&\textup{odd}&\textup{$q_0^2$ odd}\\
\bottomrule
\end{tabular}
\end{table}
  \item $\cU=\cN_3$, $m=4$, and $\nonsplit{2}{\Omega_7(q)^I}\lhdeq H$.
  \end{enumerate}
\end{proposition}

\begin{proof}
  Let $\cU=\cN_k$,  and let $G=\GammaO^+_{2m}(q)$, so $G$ fuses the two 
  $\Omega_{2m}^{+}(q)$-orbits $\cN_k^+$ and $\cN_k^-$ in $\cN_k$, and
 the setwise stabiliser $G^+$ of $\cN_k^+$  has index 2 in $G$.
  Clearly $\{\cN_k^{-},\cN_k^{+}\}$ is a system of imprimitivity for $G$,
  and hence for $H$.
  As $H$ is transitive on $\cN_k$, $H^+:=H\cap G^+$ is an index 2 subgroup of $H$ that
  fixes $\cN_k^+$ and $\cN_k^-$ setwise and is transitive on both, and $H=\langle H^{+},h\rangle$ for any
  $h\in H\setminus H^+$ with $h^2\in H^{+}$.
  Now $\Omega_{2m}^+(q)\not\leqslant H^+$, by our assumptions on $H$, and so $H^+$ is given by
  Theorem~\ref{T10}. 
In particular, $k=1$ or $(m,k)=(4,3)$.
  
  Suppose first that $(m,k)=(4,3)$. Then by Theorem~\ref{T10}~(g) we have $\nonsplit{2}{\Omega_7(q)^I}\lhdeq H^{+}$, and moreover $\nonsplit{2}{\Omega_7(q)^I}$ is transitive on $\cN_3^+$.   By~\cite{BHRD}*{Table~8.50, p.\,402}, the group $\Omega_8^{+}(q)$ has four conjugacy classes of subgroups isomorphic to
  $\nonsplit{2}{\Omega_7(q)}$ and by \cite{K87} these
  are fused by $\GO_8^{+}(q)$.   Hence there is a unique conjugacy
  class of subgroups of $\GO_8^{+}(q)$ isomorphic to $\nonsplit{2}{\Omega_7(q)}$. Thus there exists $h$ in $\CO_8^{+}(q)\setminus\GO_8^{+}(q)$
  normalising $\nonsplit{2}{\Omega_7(q)^I}$, and since $h\notin\GO_8^+(q)$ it interchanges $\cN_3^+$ and $\cN_3^-$. Thus $\nonsplit{2}{\Omega_7(q)^I}$ is also transitive on $\cN_1^-$ and  $\langle \nonsplit{2}{\Omega_7(q)^I},h\rangle$ is transitive on $\cN_1$. Since $\nonsplit{2}{\Omega_7(q)^I}$ is characteristic in $H^+$,  part  (c) holds.

We assume from now on that $k=1$.  If $(m,q)=(4,3)$ then the transitive subgroups $H$ are those given by Lemma~\ref{lem:m4q3}(3). Thus we may in addition assume that $(m,q)\neq(4,3)$. Then by Theorem~\ref{T10}(e) either $H^+$ fixes a totally singular $m$-subspace $W$ and $H^+$ induces a subgroup of $\GammaL_m(q)$ that is transitive on the set of $1$-subspaces of $W$, or  $H^+$ contains a normal subgroup $H_0$ as follows (recall that $q$ is odd):
  \begin{enumerate}[(i)]
  \item  $\SL_m(q)$;
  \item  $\Sp_m(q)$ or $\SU_m(q)$ with $m$ even;
  \item  $\nonsplit{2}{\Omega_7(q)}$ with $m=4$;
  \item$\nonsplit{2}{\Omega_8^-(q^{1/2})}$ with $m=4$;
  \item $\Omega_9(q)$ with $m=8$.
\end{enumerate}  
  
  Suppose first that $H^+$ fixes a totally singular $m$-subspace $W$ and induces a subgroup of $\GammaL_m(q)$ that is transitive on the set of $1$-subspaces of $W$. Thus $H^+$ acts irreducibly on $W$ and hence also on $V/W$. Since the action of $H^+$ on $V/W$ is dual to the action of $H^+$ on $W$, it follows that either $W$ is the only totally singular $m$-subspace fixed by $H^+$ or $H^+$ fixes precisely two totally singular $m$-subspaces, namely $W$ and a complementary space $W'$. Since $N_G(H^+)$ fixes setwise the set of totally singular $m$-subspaces fixed by $H^+$ it follows that either $H$ fixes $W$, or $H$ fixes the decomposition $V=W\oplus W'$. In the first case $H$ is as in part (a) of Proposition~\ref{P2}. In the second case $H$, and hence also $H^+$, is contained in the normaliser in $G$ of $\SL_m(q)$. Then as seen in the analysis of case ($\delta$) of the proof of \cite{Reg}*{Lemma~4.5} the group $H^+$ has a normal subgroup $H_0$ as listed in cases (i) or (ii) above. We deal with these below.

 It remains to consider the cases (i)--(v). In all such instances $H_0$ is characteristic in $H^+$ and hence is normal in $H$. Thus to complete the proof it suffices to show that examples arise in each case.

  For case (i), we see from \cite{Factns}*{Table~1} that the stabiliser  $K$ in $\GO^+_{2m}(q)$ of a pair of complementary totally singular $m$-subspaces is transitive on $\cN_1^+$. Moreover, $\SL_m(q)$ is characteristic in $K$. The group $\CO_{2m}^+(q)$ fixes the $\GO_{2m}^+(q)$-conjugacy class of such subgroups and so there exists $h\in\CO^+_{2m}(q)\backslash\GO_{2m}^+(q)$ that normalises $K$. Such an element $h$ interchanges $\cN_1^+$ and $\cN_1^-$, and it follows that $K$ is transitive on $\cN_1^-$ and $\langle K,h\rangle$ is transitive on $\cN_1$. Thus examples arise in this case.
  
 For case (ii) we see from \cite{Factns}*{Table~1} that the normaliser $K$ of $\SU_m(q)$ in $\GO^+_{2m}(q)$ is transitive on $\cN_1^+$ when $m$ is even. By~\cite{KL}*{Proposition~4.3.18} there is only one conjugacy class in $\GO^+_{2m}(q)$ of such subgroups and so arguing as in the previous case we again find an example.  Next suppose that $\Sp_m(q)$ is normal in $H^+$. Then by the proof of \cite{Reg}*{Lemma~4.5}, the group $\Sp_{m}(q)$ is contained in the stabiliser $K$ of a pair of totally singular $m$-subspaces. As $K$ contains $\GL_m(q)$, there is only one $\GO^+_{2m}(q)$-conjugacy class, and hence only one $\CO_{2m}^+(q)$-conjugacy class, of such subgroups $\Sp_m(q)$. 
Thus there exists $h\in\CO^+_{2m}(q)\backslash\GO_{2m}^+(q)$ that normalises the subgroup $H_0=\Sp_m(q)$, and such an element $h$ interchanges $\cN_1^+$ and $\cN_1^-$. So if $H^+$ is transitive on $\cN_1^+$ it is also transitive on $\cN_1^-$ and  $\langle H^+,h\rangle$ is transitive on $\cN_1$,  and we again get examples.

Finally we deal with cases (iii)-(v). If $m=4$ and $H_0=\nonsplit{2}{\Omega_7(q)}^I$ then we have already seen in the case $(m,k)=(4,3)$ that there is a unique conjugacy class of subgroups of $\GO_8^+(q)$ isomorphic to $\nonsplit{2}{\Omega_7(q)}^I$.  Since $\nonsplit{2}{\Omega_7(q)}^I$ is transitive on $\cN_1^+$ then the same argument as before yields an example here as well.
If $m=4$ and $H_0=\nonsplit{2}{\Omega_8^-(q^{1/2})}$  then by \cite{BHRD}*{Table~8.50, p.\,402}, the group $\Omega_8^{+}(q)$ has four conjugacy classes of subgroups isomorphic to  $\nonsplit{2}{\Omega_8^-(q^{1/2})}$ and by \cite{K87} these are all fused in $\GO^+_8(q)$. Hence the argument above yields an example here.
Finally suppose that $m=8$.  Then it is seen in the last line of the the proof of the proposition of \cite{Factns}*{Appendix~3} that $N_G(\Omega_9(q))$ is transitive on $\cN_1$.
\end{proof}

\section{Transitive generalised quadrangles}\label{S:ExceptIsos2}

A \emph{generalised quadrangle} $\mathcal{Q}=(\cP,\mathcal{L})$ of
\emph{order} $(s,t)$ is a point-line incidence
relation satisfying:
\begin{enumerate}
\item[(GQ1)] each $\ell\in\mathcal{L}$ is incident with $s+1$ points
  $p\in\cP$,
  \item[(GQ2)] each $p\in\cP$ is incident with $t+1$ lines
  $\ell\in\mathcal{L}$, and 
  \item[(GQ3)] for each point $p$ not on a line $\ell$, there is
    a unique point $p'\in\cP$ and a unique line $\ell'\in\mathcal{L}$, such that
    $p$ is on $\ell'$, and $p'$ is on $\ell$ and $\ell'$.
\end{enumerate}
To avoid degeneracies we henceforth assume that $\mathcal{Q}$ is \emph{thick}, that is,  $s,t>1$.
It well known that $|\cP|=(s+1)(st+1)$, $|\mathcal{L}|=(t+1)(st+1)$ and
$\sqrt{t}\le s\le t$, see~\cite{BLS}*{Lemma~2.1}. 

A beautiful theorem of Buekenhout and Lef\'evre~\cite{BF} says that
if the point set $\cP$ and line set
$\mathcal{L}$ of a generalised quadrangle $\mathcal{Q}$ are chosen from the
points and lines of a projective space $\textup{PG}(d,q)$, then $\mathcal{Q}$
is a \emph{classical generalised quadrangle}. By this we mean that
$(\cP,\mathcal{L})=(\cP_1,\cP_2)$ where $\cP_i$ is the set of
totally singular $i$-subspaces obtained from the natural module
for $\Sp_4(q)$,
$\GO_5(q)$,
$\GO_6^{-}(q)$,
$\GU_4(q)$, or
$\GU_5(q)$, and incidence is containment. The five thick classical generalised quadrangles are listed in Table~\ref{tab:GQ} along with their automorphism groups. 

\begin{table}[!ht]\renewcommand{\arraystretch}{1.2}
\caption{The classical generalised quadrangles $\mathcal{Q}$. Let $q_i=q^i+1$.}\label{tab:GQ}
\begin{tabular}{rccccc}
\toprule
$\mathcal{Q}$&$W_3(q)$&$Q_4(q)$&$Q^{-}_5(q)$&$H_3(q^2)$&$H_4(q^2)$\\
$(s,t)$&$(q,q)$&$(q,q)$&$(q,q^2)$&$(q^2,q)$&$(q^2,q^3)$\\
$\Aut(\mathcal{Q})$&$\textup{P}\Gamma\Sp_4(q)$&$\textup{P}\Gamma\textup{O}_5(q)$&$\textup{P}\Gamma\textup{O}_6^{-}(q)$&$\textup{P}\Gamma\textup{U}_4(q)$&$\textup{P}\Gamma\textup{U}_5(q)$\\ 
\bottomrule
\end{tabular}
\end{table}

We note that the axioms for a generalised quadrangle are symmetric and so we can interchange the role of points and lines, that is, if $\mathcal{Q}=(\cP,\mathcal{L})$ is a generalised quadrangle then $(\mathcal{L},\cP)$ is also a generalised quadrangle and is referred to as the \emph{dual} of $\mathcal{Q}$.
The dual of $W_3(q)$ is isomorphic to $Q_4(q)$ and the dual of $Q_5^-(q)$ is isomorphic to $H_3(q^2)$. The dual of $H_4(q^2)$ is not usually regarded as being a classical generalised quadrangle.

From our results we are able to determine all the point- or line-transitive 
subgroups of $\Aut(\mathcal{Q})$ where $\mathcal{Q}$ is a classical generalised
quadrangle. In each case $\soc(\Aut(\mathcal{Q}))$ is transitive on the
set $\cP$ of points and the set $\mathcal{L}$ of lines of $\mathcal{Q}$, so it suffices to determine the
transitive subgroups not containing $\soc(\Aut(\mathcal{Q}))$. 

 The groups that act regularly on the set of points of a classical generalised quadrangle were determined in \cite{BamG} and the only examples are $3^{1+2}_{\pm}$ acting on $Q_5^-(2)$ and $\C_{513}\rtimes \C_9$ acting on $Q_5^-(8)$. (Here $3_{\pm}^{1+2}$ denote extraspecial groups of order
$27$ where $3_{+}^{1+2}$ has exponent~$3$ and $3_{-}^{1+2}$ has exponent~$9$.)

\begin{theorem}\label{C:GQ}
  Let $\mathcal{Q}=(\cP,\mathcal{L})$ be a classical generalised quadrangle as
  in Table~\textup{\ref{tab:GQ}}. Suppose $H\le\Aut(\mathcal{Q})$ where
  $\soc(\Aut(\mathcal{Q}))\not\leq H$.
  \begin{enumerate}[{\rm (a)}]
    \item If $H\le\Aut(\mathcal{Q})$ is transitive on the points $\cP$, then
    there exists $H_0\lhdeq H$ where $(\mathcal{Q},H_0,q)$ is given in
    Table~$\ref{tab:transGQP1}$.
    \item If $H\le\Aut(\mathcal{Q})$ is transitive on the lines $\mathcal{L}$, then
    there exists $H_0\lhdeq H$ where $(\mathcal{Q},H_0,q)$ is given in
    Table~$\ref{tab:transGQP2}$.
  \end{enumerate}
\end{theorem}

\begin{table}[!ht]\renewcommand{\arraystretch}{1.2}
  \caption{Subgroups $H_0\lhdeq H\le\Aut(\mathcal{Q})$ where $H$ is transitive
   on the set $\cP$ of points of the classical generalised
   quadrangle~$\mathcal{Q}$.}\label{tab:transGQP1}
\begin{tabular}{ccccccccc}
\toprule
$\mathcal{Q}$&$W_3(q)$ &$W_3(q)$&$W_3(q)$&$W_3(q)$&$Q_4(q)$ &$Q_4(q)$&$Q_4(q)$&$H_3(q^2)$\\
$H_0$ &$\PSp_2(q^2)$&$2^4\ldotp\C_5$&$2^4\ldotp\Alt_5$&$\Alt_5$&$\Sp_2(q^2)$&$2^4\ldotp F_{20}$&$2^4\ldotp\Alt_5$&$\PSL_3(4)$\\
\textup{$q$}&\textup{all}&$3$&$3$&$3$&\textup{even}&$3$&$3$&$3$\\ \midrule
$\mathcal{Q}$&$Q^{-}_5(q)$&$Q^{-}_5(q)$&$Q^{-}_5(q)$&$Q^{-}_5(q)$&$Q^{-}_5(q)$&$Q^{-}_5(q)$&$Q^{-}_5(q)$ \\
$H_0$ &$\SU_3(q)$&$3_\pm^{1+2}$&$(\C_3)^3$&$\PSL_2(7)$&$\PSL_3(4)$&$\nonsplit{3}{\Alt_7}$&$\C_{513}\rtimes\C_9$\\
\textup{$q$}&\textup{all}&$2$&$2$&$3$&$3$&$5$&$8$\\
\bottomrule
\end{tabular}
\end{table}

\begin{table}[!ht]
  \caption{Subgroups $H_0\lhdeq H\le\Aut(\mathcal{Q})$ where $H$ is transitive
    on the set~$\mathcal{L}$ of lines of the classical generalised
    quadrangle~$\mathcal{Q}$.}\label{tab:transGQP2}
\begin{tabular}{ccccccccc}
\toprule
$\mathcal{Q}$&$W_3(q)$ &$W_3(q)$&$W_3(q)$&$Q_4(q)$&$Q_4(q)$&$Q_4(q)$&$Q_4(q)$&$Q^{-}_5(q)$\\
$H_0$ &$\Omega_4^{-}(q)$&$2^4\ldotp F_{20}$&$2^4\ldotp\Alt_5$&$\Omega_4^{-}(q)$&$2^4\ldotp \C_5$&$2^4\ldotp\Alt_5$&$\Alt_5$&$\PSL_3(4)$\\
\textup{$q$}&\textup{even}&$3$&$3$&\textup{all}&$3$&$3$&$3$&$3$\\ \midrule
$\mathcal{Q}$&$H_3(q^2)$&$H_3(q^2)$&$H_3(q^2)$&$H_3(q^2)$&$H_3(q^2)$&$H_3(q^2)$&$H_3(q^2)$ \\
$H_0$ &$\SU_3(q)$&$3_\pm^{1+2}$&$(\C_3)^3$&$\PSL_2(7)$&$\PSL_3(4)$&$\nonsplit{3}{\Alt_7}$&$\C_{513}\rtimes \C_9$\\
\textup{$q$}&\textup{all}&$2$&$2$&$3$&$3$&$5$&$8$\\
\bottomrule
\end{tabular}
\end{table}

\begin{proof}
  The entries in Tables~\ref{tab:transGQP1} and ~\ref{tab:transGQP2}  are obtained from
  Theorems~\ref{T1}, \ref{T:KleinSU4}, \ref{T3}, \ref{T:KleinSp4}, and~\ref{T:Ooddqeven},   and by taking $\cU$ to be the appropriate set of totally singular subspaces and taking the projective versions of the matrix groups listed. For $Q_5^-(8)$ and $H_3(8^2)$ we consult the proofs of Theorem~\ref{T:KleinSU4} and~\ref{T1}, respectively, to obtain the extra information provided for the projective groups involved when $\C_{19}\lhdeq H$. Also for $Q^-_5(2)$ and $H_3(2^2)$ we use \textsc{Magma} to see that the groups $3^{1+2}_{\pm}$ and $(\C_3)^3$ remain in the projective setting and are not `quotiented out' by a subgroup of order three.
\end{proof}

\begin{remark}\label{rem:dual}
We note that since $W_3(q)$ and $Q_4(q)$ are dual, any group that is transitive on the set of points of $W_3(q)$ should also appear as a group that is transitive on the set of lines of $Q_4(q)$ and vice versa. Similarly,  any group that is transitive on the set of lines of $W_3(q)$ will also appear as a group that is transitive on the set of points of $Q_4(q)$ and vice versa. Since $\PSp_2(q^2)\cong\PSL_2(q^2)\cong\Omega_4^-(q)$, and $\Omega_2(q^2)\cong \PSL_2(q^2)\cong\Omega_4^-(q)$, Tables~\ref{tab:transGQP1} and~\ref{tab:transGQP2} satisfy this condition. The same duality occurs between $H_3(q^2)$ and $Q_5^-(q)$.

We also note that for $q$ even,   the groups $\PSL_2(q^2)$ acting transitively on the set of points of $W_3(q)$ are distinct from the groups $\Omega_4^-(q)$ that are transitive on the set of lines of $W_3(q)$. They are interchanged by a graph automorphism of $\PSp_4(q)$ \cite{Asch}*{Section~14}. Similarly, the groups $\Sp_2(q^2)$ that are transitive on the set of points of $Q_4(q)$ for $q$ even are distinct from the groups $\Omega_4^-(q)$ that acts transitively on the set of lines.
\end{remark}

Recall that $(p,\ell)\in \cP\times\mathcal{L}$ is a \emph{flag} if $p$ lies on $\ell$, and
an~\emph{antiflag} otherwise.  The sets $\mathcal{F}$ of flags and the set
$\mathcal{A}$ of antiflags have cardinalities:
\begin{align*}
  |\mathcal{F}|&=|\cP|(t+1)=(s+1)(t+1)(st+1),\quad\textup{and}\\
  |\mathcal{A}|&=(|\cP|-(s+1))|\mathcal{L}|=st(s+1)(t+1)(st+1).
\end{align*}

We note that if $G$ acts flag-transitively or antiflag-transitively on a generalised quadrangle $\mathcal{Q}$ then it also acts flag-transitively or antiflag-transitively on its dual.  It is well known that the actions of $\PSp_4(q)$, $\PSU_4(q)$ and $\PSU_5(q)$ on totally isotropic 1-subspaces have rank~3, see for example \cites{KLrank3,L1,L2}. Thus if $G$ is one of these three groups and $p$ is a point of the corresponding generalised quadrangle $\mathcal{Q}$ then $G_p$ has two orbits on the remaining points: the points collinear with $p$ and the points not collinear with $p$. Since each pair of points lies on at most one line, it follows that $G_p$ acts transitively on the set of lines of $\mathcal{Q}$ incident with $p$ and so $G$ acts flag-transitively on $\mathcal{Q}$.
Kantor \cite{Kantor2} conjectured that, up to duality, the only flag-transitive generalised quadrangles are the classical ones and the unique generalised quadrangles of order $(3,5)$ and $(15,17)$. This conjecture is still wide-open but recently, Bamberg, Li and Swartz \cite{BLS}  showed that, up to duality, the only antiflag-transitive generalised quadrangles are the classical ones and the unique generalised quadrangle of order $(3,5)$.

\begin{theorem}\label{T:flag}
  Let $\mathcal{Q}$ be a classical generalised quadrangle as
  in Table~\textup{\ref{tab:GQ}} and let $H\leqslant \Aut(\mathcal{Q})$ with $\soc(\Aut(\mathcal{Q}))\not\leqslant H$.
  \begin{enumerate}[{\rm (a)}]
  \item If $H$ is flag-transitive,
    then either
    \begin{enumerate}[{\rm (i)}]
    \item   $\mathcal{Q}$ is $W_3(3)$ or $Q_4(3)$, and
    $H\in\{2^4\ldotp F_{20}, 2^4\ldotp\Alt_5, 2^4\ldotp\Sym_5\}$; or
    \item $\mathcal{Q}$ is $H_3(3^2)$ or $Q_5^-(3)$, and $H=\PSL_3(4).2$ (two possibilities) or $\PSL_3(4).2^2$, where in the first case the extension is by a field or graph automorphism.
    \end{enumerate}
    Moreover, all six groups listed above are flag-transitive.
  \item The group $H$ is not antiflag-transitive.
    \end{enumerate}
\end{theorem}

\begin{proof}
  A flag- or antiflag-transitive subgroup
  $H\le\Aut(\mathcal{Q})$ must be transitive on both $\cP$ and $\mathcal{L}$,
  and hence must occur in both Tables~\ref{tab:transGQP1}
  and~\ref{tab:transGQP2} (for the same generalised quadrangle $\mathcal{Q}$). Thus noting Remark~\ref{rem:dual} we see that if $\soc(\Aut(\mathcal{Q}))\not\leqslant H$, then one of the following occurs:
  \begin{enumerate}
  \item $\mathcal{Q}=W_3(3)$ and either $2^4\ldotp F_{20}\lhdeq H$ or $2^4\ldotp\Alt_5\lhdeq H$;
  \item $\mathcal{Q}=Q_4(3)$ and either $2^4\ldotp F_{20}\lhdeq H$ or $2^4\ldotp\Alt_5\lhdeq H$;
\item $\mathcal{Q}=H_3(3^2)$ and $\PSL_3(4)\lhdeq H$;
\item $\mathcal{Q}=Q_5^-(3)$ and  $\PSL_3(4)\lhdeq H$.
  \end{enumerate}

  Suppose first that $\mathcal{Q}=W_3(3)$ or $Q_4(3)$.  The possibilities for $H$ are $2^4\ldotp F_{20}$, $2^4\ldotp\Alt_5$ and $2^4\ldotp\Sym_5$. Now $(|\mathcal{F}|,|\mathcal{A}|)=(160,1440)$ and 1440 does not divide $|2^4\ldotp\Sym_5|$ and so none of these groups are antiflag-transitive. A \textsc{Magma} calculation shows that all three groups are flag-transitive, as stated in part (a)(i).
  
Next suppose that $\mathcal{Q}=H_3(3^2)$ or $Q_5^-(3)$. Then $\PSL_3(4)\lhdeq H\leqslant \PSL_3(4).2^2$ and $(|\mathcal{F}|,|\mathcal{A}|)=(1120,30240)$. Order arguments again eliminate the possibility of an antiflag-transitive group.  A \textsc{Magma} calculation shows  $\PSL_3(4)$ is not transitive on lines of $H_3(3^2)$ but $\PSL_3(4).2$ is, where the extension is by a field or graph automorphism, but not a graph-field automorphism. Moreover, both are flag-transitive.
\end{proof}

The proof of Theorem~\ref{T:flag} actually shows something quite remarkable: any group that is both point-transitive and line-transitive is also flag-transitive.

\begin{corollary}
 Let $\mathcal{Q}$ be a classical generalised quadrangle as
  in Table~\textup{\ref{tab:GQ}} and let $H\leqslant \Aut(\mathcal{Q})$ be both point-transitive and 
  line-transitive. Then $H$ is flag-transitive.
\end{corollary}

The full automorphism group of a classical generalised quadrangle acts primitively on the set of points. Moreover, there has been considerable recent effort devoted to classifying the generalised quadrangles with a group of automorphisms that acts primitively on the set of points: it is known that, if $H\leqslant\Aut(\mathcal{Q})$ acts primitively on both the set of points and the set of lines then it must be an almost simple group \cite{GQSym}. If $H$ acts primitively on points and transitively on lines then it has a unique minimal normal subgroup $N$ \cite{BPP3}, and if the group $N$ is elementary abelian then $\mathcal{Q}$ must be one of the unique generalised quadrangles of order $(3,5)$ or $(15,17)$ \cite{BGPP}. Detailed information about a subgroup $H$ of automorphisms which is assumed to act primitively on points (without any other assumptions),  is given in \cite{BPP2}. We shed light on these questions, for the classical examples, by determining all groups that act primitively on the set of points of a classical generalised quadrangle.

\begin{theorem}
 Let $\mathcal{Q}$ be a classical generalised quadrangle as
  in Table~\textup{\ref{tab:GQ}}, and suppose
that $H\leqslant \Aut(\mathcal{Q})$ acts primitively on the set of points of $\mathcal{Q}$ and $\soc(\Aut(\mathcal{Q}))\not\leqslant H$. Then $\mathcal{Q}=H_3(3^2)$ and $\PSL_3(4)\lhdeq H$.
\end{theorem}
\begin{proof}
Note that $H$ acts faithfully on  the set  $\cP$ of points of $\mathcal{Q}$. Since $H$ is primitive 
on $\cP$, each of its nontrivial normal subgroups is transitive on $\cP$. In particular  $H$ is transitive on $\cP$, and so has a normal subgroup $H_0$ given in Table~\ref{tab:transGQP1}. 
Since $\soc(H_0)$ is characteristic in $H_0$, it is normal in $H$, and hence is 
transitive on $\cP$. Thus also 
$\soc(H_0)$ has a normal subgroup, say $H_1$ appearing in Table~\ref{tab:transGQP1}.
  If $\soc(H_0)$ is abelian then  $\soc(H_0)$ must be   elementary abelian
  and regular on $\cP$. However the only abelian group appearing in Table~\ref{tab:transGQP1} is $(\C_3)^3$ acting on $Q_5^-(2)$, and by \cite{BamG} this group is not point-regular on $Q_5^-(2)$.  
Thus $\soc(H_0)$ is nonabelian, and hence is a direct product of isomorphic nonabelian simple groups. 
If $H_0=\PSp_2(q^2)$, $\Omega_3(q^2)$ or $\SU_3(q)$ then $H$ preserves an extension field structure on the underlying vector space and so preserves the partition of $\cP$ given by the 1-dimensional subspaces over $\F_{q^2}$. Hence these examples are not primitive. 

This leaves four almost simple examples to consider. Suppose first that $\soc(H_0)=\PSL_3(4)$. If $\mathcal{Q}=Q_5^-(3)$, then the $H$-action on $\cP$ is equivalent to its action on the set of lines of the dual quadrangle $H_3(3^2)$. However, we saw in the proof of Theorem~\ref{T:flag} that $\PSL_3(4)$ is not transitive on the lines of $H_3(3^2)$, so its orbits form a system of imprimitivity for the $H$-action on these lines. Thus $H$ is not primitive on the lines of $H_3(3^2)$, and hence is not primitive on the point set $\cP$, which is a contradiction. Hence in this case $\mathcal{Q}=H_3(3^2)$, and a \textsc{Magma} calculation shows that $\PSL_3(4)$ acts primitively on the set points of $H_3(3^2)$, as in the statement.
If $\mathcal{Q}=Q_5^-(3)$ and $H_0=\PSL_2(7)$, then the number of points is $112$ which does not divide 
$|H_0|$,  so $H_0$ is not transitive on points (though $\PGL_2(7)$ is), which is a contradiction. Finally, if $\mathcal{Q}=W_3(3)$ and $H_0=\Alt_5$, then the number of points is $40$, and so $H_0$ is not transitive on points (though $\Sym_5$ is), ruling this case out as well.
    \end{proof}


\end{document}